\RequirePackage{fix-cm}

\documentclass[smallextended]{svjour3}       
\smartqed  
\usepackage{graphicx}
\usepackage[hidelinks]{hyperref}

\usepackage{hyperref}       
\usepackage{url}            
\usepackage{amssymb}
\usepackage{bm}
\usepackage{xfrac}
\usepackage{bbm}
\usepackage{centernot}
\usepackage{mathtools}

\usepackage{pgfplots}
\usepackage{appendix}
\usepackage[square,numbers]{natbib}
\usepackage[]{geometry}
\usepackage{bbding}
\usepackage[shortlabels]{enumitem}
\usepackage{amsmath}
\usepackage{algorithm}
\usepackage{subcaption, placeins}
\usepackage{algorithmic}
\usepackage{booktabs}       
\usepackage{amsfonts}       
\usepackage{nicefrac}       
\usepackage{microtype}      %
\usepackage[mathscr]{euscript}
 \newtheorem{assumption}{Assumption}

\newcommand{\xt}{\bm{\mathscr{X}}}

\newcommand{\blue}{\color{black}}

\begin{document}

\title{A New Perspective on Low-Rank Optimization} 

\titlerunning{A New Perspective on Low-Rank Optimization}

\author{\mbox{Dimitris Bertsimas \and Ryan Cory-Wright \and Jean Pauphilet }}

\authorrunning{D. Bertsimas, R. Cory-Wright, J. Pauphilet}

\institute{D. Bertsimas \at
              Sloan School of Management,   Massachusetts Institute of Technology, Cambridge, MA 02139\\
              ORCID: \href{https://orcid.org/0000-0002-1985-1003}{$0000$-$0002$-$1985$-$1003$} \email{dbertsim@mit.edu}           
           \and
           R. Cory-Wright \at
               Operations Research Center, Massachusetts Institute of Technology, Cambridge, MA 02139\\
        ORCID: \href{https://orcid.org/0000-0002-4485-0619}{$0000$-$0002$-$4485$-$0619$} \email{ryancw@mit.edu}  
               \and
                J. Pauphilet \at
              London Business School, London, UK\\
               ORCID: \href{https://orcid.org/0000-0001-6352-0984}{$0000$-$0001$-$6352$-$0984$} \email{jpauphilet@london.edu}  
}

\date{\vspace{-30mm}}

\maketitle

\begin{abstract}
A key question in many low-rank problems throughout optimization, machine learning, and statistics is to characterize the convex hulls of simple low-rank sets and judiciously apply these convex hulls to obtain strong yet computationally tractable convex relaxations. We invoke the matrix perspective function — the matrix analog of the  perspective function — and characterize explicitly the convex hull of epigraphs of {\color{black}simple matrix convex} functions under low-rank constraints. Further, we {\color{black}combine the matrix perspective function with orthogonal projection matrices--the matrix analog of binary variables which capture the row-space of a matrix--to develop} {\color{black} a matrix perspective reformulation technique that reliably obtains} strong relaxations for a variety of low-rank problems{\color{black},} including reduced rank regression, non-negative matrix factorization, and factor analysis. {\color{black}Moreover, w}e establish that these relaxations can be modeled via semidefinite constraints and thus optimized over tractably. The proposed approach parallels and generalizes the perspective reformulation technique in mixed-integer optimization and leads to new relaxations for a broad class of problems.
\end{abstract}

\section{Introduction}

Over the past decade, a considerable amount of attention has been devoted to low-rank optimization, resulting in theoretically and practically efficient algorithms for problems as disparate as matrix completion, reduced rank regression, or computer vision. In spite of this progress, almost no equivalent progress has been made on developing strong lower bounds for low-rank problems. Accordingly, this paper proposes a procedure for obtaining novel and strong lower bounds.

We consider the following low-rank optimization problem:
\begin{align}\label{prob:lrsdo}
    \min_{\bm{X} \in \mathcal{S}^n_+} \ & \langle \bm{C}, \bm{X} \rangle+\Omega(\bm{X})+\mu \cdot \mathrm{Rank}(\bm{X}) \ \text{s.t.} \ \langle \bm{A}_i, \bm{X}\rangle=b_i \ \forall i \in [m], \ \bm{X} \in \mathcal{K}, \ \mathrm{Rank}(\bm{X}) \leq k,
\end{align}
where $\bm{C}, \bm{A}_{\color{black}1}, \ldots \bm{A}_m \in \mathcal{S}^n$ are $n \times n$ symmetric matrices, $b_1, \ldots b_m \in \mathbb{R}$ are scalars, $[n]$ denotes the set of running indices $\{1, ..., n\}$, {\color{black}$\mathcal{S}^n_+$ denotes the $n \times n$ positive semidefinite cone,} and $\mu \in \mathbb{R}_+, k \in \mathbb{N}$ are parameters which controls the complexity of $\bm{X}$ by respectively penalizing and constraining its rank. The set $\mathcal{K}$ is a proper—i.e., closed, convex, solid and pointed—cone {\color{black}\citep[c.f.][Section 2.4.1]{boyd2004convex}}, and
{\blue $\Omega(\bm{X})=\mathrm{tr}(f(\bm{X}))$ for some matrix convex function $f$; see formal definitions and assumptions in Section \ref{sec:perspfns}.} 

For optimization problems with logical constraints, strong relaxations can be obtained by formulating them as mixed-integer optimization (MIO) problems and applying the so-called perspective reformulation technique  \citep[see][]{frangioni2006perspective,gunluk2010perspective}. In this paper, we develop a matrix analog of the perspective reformulation technique to obtain strong yet computationally tractable relaxations of {\color{black}low-}rank optimization problems of the form \eqref{prob:lrsdo}. 

\subsection{Motivating Example}\label{sec:motivexample}
In this section, we illustrate the implications of our results on a statistical learning example. To emphasize the analogy with the perspective reformulation technique in MIO, we first consider the best subset selection problem and review its {\color{black}perspective} relaxations. We then consider a reduced-rank regression problem -- the rank-analog of best subset selection -- and provide new relaxations that naturally arise from our Matrix Perspective Reformulation Technique (MPRT). 

\paragraph{Best Subset Selection:}
Given a data matrix $\bm{X} \in \mathbb{R}^{n \times p}$ and a response vector $\bm{y} \in \mathbb{R}^{\color{black}n}$, the $\ell_0-\ell_2$ regularized best subset selection problem is to solve {\color{black}\citep[c.f.][]{pilanci2015sparse,bertsimas2016best,bertsimas2017sparse, bertsimas2019sparse, xie2018ccp, atamturk2019rank}}:
\begin{align}\label{prob:bestsubset}
    \min_{\bm{w} \in \mathbb{R}^p} \quad & \frac{1}{2n}\Vert \bm{y}-\bm{X}\bm{w}\Vert_2^2 +\frac{1}{2\gamma} \Vert \bm{w}\Vert_2^2 +\mu \Vert \bm{w}\Vert_0,
\end{align}
where $\mu, \gamma >0$ are parameters which control $\bm{w}$'s sparsity and sensitivity to noise {\color{black}respectively}. 

Early attempts at solving Problem \eqref{prob:bestsubset} exactly relied upon weak implicit or big-$M$ formulations of logical constraints which supply low-quality relaxations and therefore do not scale well {\color{black}\citep[see][for discussions]{bienstock2010eigenvalue,hazimeh2020sparse}}. However, very similar algorithms now solve these problems to certifiable optimality with millions of features. Perhaps the key ingredient in modernizing these (previously inefficient) algorithms was invoking the perspective reformulation technique—a technique for obtaining high-quality convex relaxations of non-convex sets—first stated in \citet{stubbs1998branch} PhD thesis \citep[see also][]{stubbs1999branch, ceria1999convex} and popularized by \citet{frangioni2006perspective, akturk2009strong, gunluk2010perspective} {\color{black}among others}.


\paragraph{Relaxation via the Perspective Reformulation Technique:} 

By applying the perspective reformulation technique \cite{frangioni2006perspective, akturk2009strong, gunluk2010perspective} to the term $\mu\Vert \bm{w}\Vert_0+\frac{1}{2\gamma} \Vert \bm{w}\Vert_2^2$, we obtain the following reformulation:
\begin{align}\label{eqn:bestsubsetrelax}
    \min_{\bm{w}, \bm{\rho}\in \mathbb{R}^p, \bm{z} \in \{0, 1\}^p} \quad & \frac{1}{2n}\Vert \bm{y}-\bm{X}\bm{w}\Vert_2^2 +\frac{1}{2\gamma} \bm{e}^\top \bm{\rho} +\mu \cdot \bm{e}^\top \bm{z} \quad \text{s.t.} \quad  z_i \rho_i \geq w_i^2 \quad \forall i \in [p],
    \end{align}
{\color{black} where $\bm{e}$ denotes a vector of all ones of appropriate dimension.} 

Interestingly, this formulation can be represented using second-order cones \cite{gunluk2010perspective,pilanci2015sparse} and optimized over efficiently using projected subgradient descent {\color{black}\cite{bertsimas2019sparse}}. Moreover, it reliably supplies near-exact relaxations for most practically relevant cases of best subset selection \cite{pilanci2015sparse,bertsimas2017sparse}. In instances where it is not already tight, one can apply a refinement of the perspective reformulation technique to the term $\Vert \bm{y}-\bm{X}\bm{w}\Vert_2^2$ and thereby obtain the following (tighter {\color{black}yet} more expensive) relaxation \cite{dong2015regularization}:
\begin{align}\label{prob:dcl}
    \min_{\bm{w}\in \mathbb{R}^p, \bm{z} \in [0, 1]^p, \bm{W} \in S^p_+} \quad & \frac{1}{2n}\Vert \bm{y}\Vert_2^2-\frac{1}{n}\langle \bm{y}, \bm{X}\bm{w}\rangle+\frac{1}{2}\langle \bm{W}, \frac{1}{\gamma} \mathbb{I}+\frac{1}{n}\bm{X}^\top \bm{X}\rangle+\mu \bm{e}^\top \bm{z} \\
    \text{s.t.} \quad & \bm{W} \succeq \bm{w}\bm{w}^\top, \   z_i W_{i,i} \geq w_i^2 \ \forall i \in [p].\nonumber
\end{align}
Recently, a class of even tighter relaxations were developed by {\color{black}\citet{atamturk2019rank,han20202x2, frangioni2020decompositions}}. As they were developed by considering multiple binary variables simultaneously and therefore do not{\color{black}, to our knowledge,} generalize readily to the low-rank case (where we often have one low-rank matrix), we do not discuss (or generalize) them here.

\paragraph{Reduced Rank Regression:}
Given $m$ observations of a response vector $\bm{Y}_j \in \mathbb{R}^n$ and a predictor $\bm{X}_j \in \mathbb{R}^p$, an important problem in high-dimensional statistics is to recover a low-complexity model which relates $\bm{X}, \bm{Y}$. A popular choice for doing so is to assume that $\bm{X}, \bm{Y}$ are related via $\bm{Y}=\bm{X}\bm{\beta}+\bm{E}$, where $\bm{\beta} \in \mathbb{R}^{p \times n}$ is a coefficient matrix which we assume to be low-rank, $\bm{E}$ is a matrix of noise and we require that the rank of $\bm{\beta}$ is small in order that the linear model is parsimonious \cite{negahban2011estimation}. Introducing Frobenius regularization gives rise to the problem:
\begin{align}\label{eqn:rrr_orig}
    \min_{\bm{\beta} \in \mathbb{R}^{p \times n}} \quad \frac{1}{2m}\Vert \bm{Y}-\bm{X}\bm{\beta}\Vert_F^2+\frac{1}{2\gamma}\Vert \bm{\beta}\Vert_F^2+\mu \cdot \mathrm{Rank}( \bm{\beta}),
\end{align}
where $\gamma, \mu >0$ control the robustness to noise and the complexity of the estimator respectively and we normalize the ordinary least squares loss by dividing by $m$, the number of observations. 

Existing attempts at solving this problem generally involve replacing the low-rank term with a nuclear norm term \cite{negahban2011estimation}, which succeeds under some strong assumptions on the problem data but not in general. Recently, {we} proposed a new framework to model rank constraints, using orthogonal projection matrices which satisfy $\bm{Y}^2=\bm{Y}$ instead of binary variables which satisfy $z^2=z$ \cite{bertsimas2020mixed}. By building on this work, in this paper we propose a generalization of the perspective function to matrix-valued functions {\color{black}with positive semidefinite arguments} and develop a matrix analog of the perspective reformulation technique from MIO which uses projection matrices instead of binary variables.


\paragraph{Relaxations via the Matrix Perspective Reformulation Technique:}
By applying the matrix perspective reformulation technique (Theorem \ref{lemma:equivalence}) to the term ${\color{black}\frac{1}{2\gamma}}\Vert \bm{\beta}\Vert_F^2+\mu \cdot \mathrm{Rank}(\bm{\beta})$, we will prove that the following problem is a valid—and numerically high-quality—relaxation of \eqref{eqn:rrr_orig}:
\begin{align}\label{eqn:rrr_persp}
\min_{\bm{\beta} \in \mathbb{R}^{p \times n}, \bm{W} \in \mathcal{S}^n_+, \bm{\theta} \in S^p_+} \quad \frac{1}{2m}\Vert \bm{Y}-\bm{X}\bm{\beta}\Vert_F^2+\frac{1}{2\gamma} \mathrm{tr}(\bm{\theta})+\mu \cdot \mathrm{tr}(\bm{W}) \quad \text{s.t.} \quad \bm{W} \preceq \mathbb{I}, \begin{pmatrix} \bm{\theta} & \bm{\beta} \\ \bm{\beta}^\top & \bm{W}\end{pmatrix} \succeq \bm{0}.
\end{align}
The analogy between problems \eqref{prob:bestsubset}-\eqref{eqn:rrr_orig} and their relaxation{\color{black}s} \eqref{eqn:bestsubsetrelax}-\eqref{eqn:rrr_persp} is striking. The goal of the present paper is to develop the corresponding theory to support and derive the relaxation \eqref{eqn:rrr_persp}. Interestingly, the main argument that led \cite{dong2015regularization} to the improved relaxation \eqref{prob:dcl} for \eqref{prob:bestsubset} can be extended to reduced-rank regression. Combined with our MPRT, it leads to the relaxation:
\begin{align}
    \min_{\bm{\theta} \in \mathcal{S}^n_+, \bm{\beta} \in \mathbb{R}^{p \times n}, \bm{B} \in \mathcal{S}^n_+, \bm{W} \in \mathcal{S}^n_+} \quad & \frac{1}{2m}\Vert \bm{Y}\Vert_F^2-\frac{1}{m}\langle \bm{Y}, \bm{X}\bm{\beta}\rangle+\frac{1}{2}\langle \bm{B}, \frac{1}{\gamma}\mathbb{I}+\frac{1}{m}\bm{X}^\top \bm{X}\rangle+\mu \cdot \mathrm{tr}(\bm{W})\label{eqn:rrr_dcl}\\
    \text{s.t.} \quad & \begin{pmatrix} \bm{B} & \bm{\beta}\\ \bm{\beta} & \bm{W}\end{pmatrix} \succeq \bm{0}, \bm{W} \preceq \mathbb{I}.\nonumber
\end{align}
It is not too hard to see that this is a valid semidefinite relaxation: if $\bm{W}$ is a rank-$k$ projection matrix then, by the Schur complement lemma \cite[see][Equation 2.41]{boyd1994linear}, $\bm{\beta}=\bm{\beta}\bm{W}$, and thus the rank of $\bm{\beta}$ is at most $k$. Moreover, if we let $\bm{B}=\bm{\beta}\bm{\beta}^\top$ in a solution, we recover a low-rank solution to the original problem\footnote{Observe that the constraints in Problem \eqref{prob:dcl} are equivalent to the block matrix constraint $\begin{pmatrix} \mathrm{Diag}(\bm{z}) & \mathrm{Diag}(\bm{w}) \\ \mathrm{Diag}(\bm{w}) & \bm{W}\end{pmatrix} \succeq \bm{0}.$ This verifies that the reduced rank regression formulation is indeed a generalization of \cite{dong2015regularization}'s formulation for sparse regression.}. 
Actually, as we show in Section \ref{ssec:reformtech}, a similar technique can be applied to any instance of Problem \eqref{prob:lrsdo}, for which the applications beyond matrix regression are legion.

\subsection{Literature Review}
Three classes of approaches have been proposed for solving Problem \eqref{prob:lrsdo}: (a) heuristics, which prioritize computational efficiency and obtain typically high-quality solutions to low-rank problems efficiently but without optimality guarantees \citep[see][for a review]{nguyen2019low}; (b) relax-and-round approaches, which balance computational efficiency and accuracy concerns by relaxing the rank constraint and rounding a solution to the relaxation to obtain a provably near-optimal low-rank matrix \citep[][Section 1.2.2]{bertsimas2020mixed}; and (c) exact approaches, which prioritize accuracy over computational efficiency and solve Problem \eqref{prob:lrsdo} exactly in exponential time {\color{black}\citep[][Section 1.2.1]{bertsimas2020mixed}}. 

Of the three classes of approaches, heuristics currently dominate the literature, because their superior runtime and memory usage allows them to address larger-scale problems. However, recent advances in algorithmic theory and computational power have drastically improved the scalability of exact and approximate methods, to the point where they can now solve moderately sized problems which are relevant in practice \cite{bertsimas2020mixed}. Moreover, relaxations of strong exact formulations often give rise to very efficient heuristics (via tight relaxations of the exact formulation) which outperform existing heuristics. This suggests that heuristic approaches may not maintain their dominance going forward, and motivates the exploration of tight yet affordable relaxations of low-rank problems. 

\subsection{Contributions and Structure}
The main contributions of this paper are {\color{black}twofold}. First, we propose a general reformulation technique for obtaining high-quality relaxations of low-rank optimization problems: introducing an orthogonal projection matrix to model a low-rank constraint, and strengthening the formulation by taking the matrix perspective of an appropriate substructure of the problem. This technique can be viewed as a generalization of the perspective reformulation technique for obtaining strong relaxations of sparse or logically constrained problems \cite{frangioni2006perspective, gunluk2010perspective,bertsimas2019unified, han20202x2}. Second, by applying this technique, we obtain explicit characterizations of convex hulls of low-rank sets which frequently arise in low-rank problems. 
As the interplay between convex hulls of indicator sets and perspective functions has engineered algorithms which outperform state-of-the-art heuristics in sparse linear regression \cite{bertsimas2017sparse, hazimeh2020sparse} and sparse portfolio selection \cite{zheng2014improving,bertsimas2019unified}, we hope that this work will empower similar developments for low-rank problems. 

The rest of the paper is structured as follows: In Section \ref{sec:background} we supply some background on perspective functions and review their role in developing tight formulations of mixed-integer problems. In Section \ref{sec:perspfns}, we {\color{black}introduce} the matrix perspective function and its properties, {\color{black}extend the function's definition to allow semidefinite in addition to positive definite arguments,} and propose a matrix perspective reformulation technique (MPRT) which successfully obtains high-quality relaxations for low-rank problems which commonly arise in the literature.
We also connect the matrix perspective function to the convex hulls of {\color{black}epigraphs of simple matrix convex functions under rank constraints}. {\color{black}In Section \ref{sec:convhulls}, we} illustrate the utility of this connection by deriving tighter relaxations of several low-rank problems than are currently available in the literature. Finally, in Section \ref{sec:numres}, we numerically verify the utility of our approach on reduced rank regression, {\color{black}D-optimal design} and non-negative matrix factorization problems.

\paragraph{Notation:} We let nonbold face characters such as $b$ denote scalars, lowercase bold faced characters such as $\bm{x}$ denote vectors, uppercase bold faced characters such as $\bm{X}$ denote matrices, and calligraphic uppercase characters such as $\mathcal{Z}$ denote sets. We let $[n]$ denote the set of running indices $\{1, ..., n\}$ and $\mathbb{N}$ denote the set of positive integers. We let $\mathbf{e}$ denote a vector of all $1$'s, $\bm{0}$ denote a vector of all $0$'s, and $\mathbb{I}$ denote the identity matrix. We let $\mathcal{S}^n$ denote the cone of $n \times n$ symmetric matrices, $\mathcal{S}^n_+$ denote the cone of $n \times n$ positive semidefinite matrices, $\mathcal{S}^n_{+} \cap \mathbb{R}^{n \times n}_{+}$ denote the cone of $n \times n$ doubly non-negative matrices, and $\mathcal{C}^n_+:=\{\bm{U}\bm{U}^\top: \bm{U} \in \mathbb{R}^{n \times n}_{+}\}$ denote the cone of $n \times n$ completely positive matrices. Finally, we let $\bm{X}^\dag$ denote the Moore-Penrose pseudoinverse of a matrix $\bm{X}$; see \citet{johnson1985matrix, bhatia2013matrix} for general theories of matrix operators. Less common matrix operators will be defined as they are needed.

\section{Background on Perspective Functions}\label{sec:background}
In this section, we review perspective functions and their interplay with tight formulations of logically constrained problems. This prepares the ground for and motivates our study of matrix perspective functions and their interplay with tight formulations of low-rank problems. Many of our subsequent results can be viewed as (nontrivial) generalizations of the results in this section, since a rank constraint is a cardinality constraint on the singular values.

\subsection{Preliminaries}
{\blue Consider a proper closed convex function $f : \mathcal{X} \to \mathbb{R}$, where $\mathcal{X}$ is a convex subset of $\mathbb{R}^n$. The {perspective function} of $f$ is commonly defined for any $\bm{x} \in \mathbb{R}^n$ and any $t > 0$ as $(\bm{x}, t) \mapsto t f(\bm{x}/t)$. {\blue Its closure is defined by continuity for $t=0$ and} is equal to \citep[c.f.][Proposition IV.2.2.2 ]{hiriart2013convex}:} 
\begin{align*}
    g_f(\bm{x},t) = \begin{cases} t f(\bm{x} / t) & \mbox{ if } t > 0, \bm{x}/t \in \mathcal{X}, \\
    0 & \mbox{ if } t = 0, \bm{x} = 0, \\
    \blue f_\infty(\bm{x}) & \mbox{ if } t = 0, \bm{x} \neq 0, \\
    +\infty & \mbox{otherwise,}
    \end{cases}
\end{align*}
{\blue where $f_\infty$ is the recession function of $f$, as originally stated in \citep[][p. 67]{rockafellar1970convex} which is given by
\begin{align*}
    f_\infty(\bm{x}) = \lim_{t \rightarrow 0} t f\left( \bm{x}_0 - \bm{x} + \dfrac{\bm{x}}{t}\right) = \lim_{t \rightarrow +\infty} \dfrac{f(\bm{x}_0 + t \bm{x}) - f(\bm{x}_0)}{t},
\end{align*}
for any $\bm{x}_0$ in the domain of $f$. That is, $f_\infty(\bm{x})$ is the asymptotic slope of $f$ in the direction of $\bm{x}$.} 

The perspective function was first investigated by \citet{rockafellar1970convex}, who made the important observation that $f$ is convex {\blue in $\bm{x}$} if and only if $g_f$ is convex {\blue in $(\bm{x},t)$}. 
{\blue Among other properties, we have that,} for any $t>0$, $(\bm{x},t,s) \in \mathrm{epi}(g_f)$ if and only if $(\bm{x}/t, s/t) \in \mathrm{epi}(f)$ \citep[Proposition IV.2.2.1]{hiriart2013convex}.
We refer to the review by \citet{combettes2018perspective} for further properties of perspective functions.

{\blue Throughout this work, we refer to $g_f$ as the \textit{perspective function} of $f$ --although it technically is the closure of the perspective. We also consider a family of convex functions $f$ which satisfy:  
\begin{assumption} \label{ass:coercive}
The function $f : \mathcal{X} \to \mathbb{R}$ is proper, closed, and convex. $\bm{0} \in \mathcal{X}$ and for any $\bm{x} \neq \bm{0}$, $f_\infty(\bm{x}) = + \infty$.
\end{assumption} 
The condition $f_\infty(\bm{x}) = +\infty, \forall \bm{x} \neq \bm{0}$ is equivalent to 
$\lim_{\bm{x} \rightarrow \infty} {f(\bm{x})}/{\|\bm{x}\|} = +\infty,$
and means that, asymptotically, $f$ increases to infinity faster than any affine function. In particular, it is satisfied if the domain of $f$ is bounded or if $f$ is strictly convex. 
Under Assumption \ref{ass:coercive},} the definition of the perspective function of $f$ simplifies to
\begin{align}\label{defn:perspfn}
    g_f(\bm{x},t) = \begin{cases} t f(\bm{x} / t) & \mbox{ if } t > 0, \\
    0 & \mbox{ if } t = 0, \bm{x} = 0, \\
    +\infty & \mbox{otherwise.}
    \end{cases}
\end{align}

\subsection{The Perspective Reformulation Technique}
A number of authors have observed that optimization problems over binary and continuous variables admit tight reformulations involving perspective functions of appropriate substructures of {\color{black}the problem}, since \citet{ceria1999convex}, building upon the work of \citet[Theorem 9.8]{rockafellar1970convex}, derived the convex hull of a disjunction of convex constraints. 
To motivate our study of the {\color{black}matrix} perspective function in the sequel, we now demonstrate that a {\color{black}class of} logically-constrained problem{\color{black}s} admit reformulation{\color{black}s} in terms of perspective functions. We remark that this development bears resemblance to other works on perspective reformulations including \cite[]{bertsimas2019unified, han20202x2,frangioni2020decompositions}. 

Consider a logically-constrained problem of the form
\begin{equation}
    \begin{aligned}\label{eqn:original_minlp}
    \min_{\bm{z} \in \mathcal{Z}, {\bm{x} \in  \mathbb{R}^n}} \quad & \bm{c}^\top \bm{z} + f(\bm{x}) + \Omega(\bm{x}) \quad
    \text{s.t.} \quad x_i =0\ \text{if} \  z_i=0 \quad \forall i \in [n],
\end{aligned}
\end{equation}
where $\mathcal{Z} \subseteq \{0,1\}^n$, $\bm{c} \in \mathbb{R}^n$ is a cost vector, $f(\cdot)$ is a generic convex function {which possibly models convex constraints $\bm{x} \in \mathcal{X}$ for a convex set $\mathcal{X} \subseteq \mathbb{R}^n$ implicitly—by requiring that $g(\bm{x})=+\infty$ if $\bm{x} \notin \mathcal{X}$}, and $\Omega(\cdot)$ is a regularization function which satisfies the following assumption:
\begin{assumption}[Separability]
$\Omega(\bm{x})=\sum_{i \in [n]} \Omega_i(x_i)$, where each $\Omega_i$ {\blue satisfies Assumption \ref{ass:coercive}.} 
\end{assumption}

Since $z_i$ is binary, imposing the logical constraint ``$x_i=0$ if $z_i=0$'' plus the term $\Omega_i(x_i)$ in the objective is equivalent to $g_{\Omega}(x_i, z_i)+(1-z_i)\Omega_i(0)$ in the objective, where $g_{\Omega_i}$ is the perspective function of $\Omega_i$, and thus Problem \eqref{eqn:original_minlp} is equivalent to:
\begin{equation}
    \begin{aligned}\label{eqn:original_minlp.persp}
    \min_{\bm{z} \in \mathcal{Z}, {\bm{x} \in  \mathbb{R}^n}} \quad & \bm{c}^\top \bm{z} + f(\bm{x})+\sum_{i=1}^n \bigg(g_{\Omega_i}(x_i, z_i)+(1-z_i)\Omega_i(0)\bigg).
\end{aligned}
\end{equation}
Notably, while Problems \eqref{eqn:original_minlp}-\eqref{eqn:original_minlp.persp} have the same feasible regions, \eqref{eqn:original_minlp.persp} often has substantially stronger relaxations, as frequently noted in the perspective reformulation literature \cite{frangioni2006perspective, gunluk2010perspective, fischetti2016redesigning, bertsimas2019unified}. 

For completeness, we provide a formal proof of equivalence between \eqref{eqn:original_minlp}-\eqref{eqn:original_minlp.persp}; note that a related (although dual, and weaker as it requires $\Omega(\bm{0})=\bm{0}$) result can be found in \citep[Thm. 2.5]{bertsimas2019unified}:
\begin{lemma}\label{lemma:perspreformtech}
Suppose \eqref{eqn:original_minlp} attains a finite optimal value. Then, \eqref{eqn:original_minlp.persp} attains the same value.
\end{lemma}
\begin{proof}
It suffices to establish that the following equality holds:
\begin{align*}
    g_{\Omega_i}(x_i, z_i)+(1-z_i)\Omega_i(0)=\Omega_i(x_i)+\begin{cases}0 & \text{if} \ x_i=0 \ \text{or} \ z_i=1,\\ +\infty & \text{otherwise.} \end{cases}
\end{align*}
Indeed, this equality shows that any feasible solution to one problem is a feasible solution to the other with equal cost. We prove this by considering the cases where $z_i=0$, $z_i=1$ separately.
\begin{itemize}
    \item Suppose $z_i=1$. Then, $g_{\Omega_i}(x_i, z_i)=z_i \Omega_i(x_i/z_i) =\Omega_i(x_i)$ and $x_i=z_i\cdot x_i$, so the result holds.
        \item Suppose $z_i=0$. If $x_i=0$ we have $g_{\Omega_i}(0,0)+\Omega_i(0)=\Omega_i(0)$, and moreover the right-hand-side of the equality is certainly $\Omega_i(0)$. Alternatively, if $x_i \neq 0$ then both sides equal $+\infty$. \qed
\end{itemize}
\end{proof}
In Table \ref{tab:perspfncomparison.mio}, we present examples of penalties $\Omega$ for which {\color{black}Assumption \ref{ass:coercive} holds and} the perspective reformulation technique is applicable. 
We remind the reader that the exponential cone is \citep[c.f.][]{chares2009cones}:
\begin{align*}
    \mathcal{K}_{\text{exp}}=\{\bm{x} \in \mathbb{R}^3: x_1 \geq x_2 \exp(x_2/x_3), x_2 >0\} \cup \{(x_1, 0, x_3) \in \mathbb{R}^3: x_1 \geq 0, x_3 \leq 0\},
\end{align*}
while the power cone is defined for any $\alpha \in (0, 1)$ as \citep[c.f.][]{chares2009cones}:
\begin{align*}
    \mathcal{K}_{\text{pow}}^{\alpha}=\{\bm{x} \in \mathbb{R}^3: x_1^\alpha x_2^{1-\alpha} \geq \vert x_3\vert\}.
\end{align*}

\begin{table}[h!]
\centering\footnotesize
\caption{Convex substructures which frequently arise in MIOs and their perspective reformulations. {\blue For conciseness, we give $g_\Omega(x,z)$ for $z>0$ only, i.e., the first case in \eqref{defn:perspfn}, $g_\Omega(x,z)$ for $z=0$ being defined as in Equation \eqref{defn:perspfn}.}}
\begin{tabular}{@{}l l l l@{}} \toprule
 Penalty & $\Omega(x)$  & $g_{\Omega}(x, z)$ if $z>0$ & Formulation\\\midrule
 Big-$M$ & $\begin{cases} 0 & \text{if} \ \vert x \vert \leq M,\\ +\infty & \text{otherwise}\end{cases}$ &  $\begin{cases} 0 & \text{if} \ \vert x \vert \leq Mz\\ +\infty & \text{otherwise}\end{cases}$ & $\vert x \vert \leq M z$ \\
 \midrule
 Ridge & $\frac{1}{2\gamma}x^2$ & $x^2/ 2\gamma z$ & \parbox{3cm}{$\begin{aligned}\min \quad & \theta \\  \text{s.t.} \quad & \theta z \geq \frac{1}{2\gamma} x^2\end{aligned}$} \\ \midrule
 Ridge $+$ Big-$M$ & $\frac{1}{2\gamma}x^2,\ \text{if} \ \vert x \vert \leq M$ & $x^2/ 2\gamma z,\ \text{if} \ \vert x \vert \leq Mz$ &
  \parbox{3cm}{$\begin{aligned}\min \quad & \theta \\ \text{s.t.} \quad & \theta z \geq \frac{1}{2\gamma} x^2, \ \vert x \vert \leq M z \end{aligned}$} \\ \midrule
 Power & $\vert x\vert^p$, $p {\blue >} 1$ & $\vert x\vert^p z^{1-p}$ & \parbox{3cm}{$\begin{aligned} \min \quad & \theta \\  \text{s.t.} \quad & (\theta, z, x) \in \mathcal{K}_{\text{pow}}^{1/p} \end{aligned}$} \\ \midrule
 Log$_\epsilon$ {\color{black}+ Big-$M$} & \color{black} $-\log (x+\epsilon),\  \text{if} \ 0 \leq x \leq M$ & \color{black} $-z\log(x/z+\epsilon),\ \text{if} \ x \leq Mz$ &
 \parbox{3cm}{$\begin{aligned} \min \quad & \theta \\ \text{s.t.} \quad & (x+z\epsilon, z, -\theta) \in \mathcal{K}_{\text{exp}},\\ & x \leq Mz \end{aligned}$} \\\midrule 
 Entropy & $x \log x$ & $x\log(x/z),\ \text{if} \ x>0$ &
 \parbox{3cm}{$\begin{aligned} \min \quad & \theta \\ \text{s.t.} \quad & (z, x, -\theta) \in \mathcal{K}_{\text{exp}},\\ & x \leq Mz \end{aligned}$} \\ \midrule
\color{black}Softplus+Big-$M$ & 
\color{black}$ \log(1+\exp(x)),\ \text{if} \ \vert x\vert \leq M$ & \color{black} $z\log(1+\exp(x/z)),\ \text{if} \ \vert x \vert \leq Mz$ 
& \color{black} \parbox{3cm}{\begin{align*} \min \quad & \theta \\ \text{s.t.} \quad & z \geq u+v, \vert x \vert \leq Mz, \\ & (u,z,-\theta) \in \mathcal{K}_{\text{exp}},\\ & (v,z,x-\theta) \in \mathcal{K}_{\text{exp}} \end{align*}}  \\
\bottomrule
\end{tabular}
\label{tab:perspfncomparison.mio}
\end{table}

\subsection{Perspective Cuts}
Another computationally useful application of the perspective reformulation technique has been to derive a class of cutting-planes for MIOs with logical constraints \citep{frangioni2006perspective}. To motivate our generalization of these cuts to low-rank problems, we now briefly summarize their main result.

Consider the following problem:
\begin{equation}
\begin{aligned}\label{prob:origminlo}
    \min_{\bm{z} \in  \mathcal{Z}} \min_{\bm{x} \in \mathbb{R}^n} \quad & \bm{c}^\top \bm{z}+f(\bm{x})+\sum_{i=1}^n \Omega_i(x_i)\\
    \text{s.t.} \quad & \bm{A}^i x_i \leq b_i z_i \quad \forall i \in [n],
\end{aligned}
\end{equation}
where $\{{\color{black}x_i: \ }\bm{A}^i x_i \leq 0\}=\{0\}$, which implies the set of feasible $\bm{x}$ is bounded, $\Omega_i(x_i)$ is a closed convex function, {\color{black} we take $\Omega_i(0)=0$ as in \cite{frangioni2006perspective} for simplicity}, and $f(\bm{x})$ is a convex function. Then, letting $\rho_i$ model the epigraph of $\Omega_i(x_i)+c_i z_i$ and $s_i$ be a subgradient of $\Omega_i$ at $\bar{x}_i$, i.e., $s_i \in \partial \Omega_i(\bar{x}_i)$, we have the following result \cite[][]{frangioni2006perspective, gunluk2010perspective}:
\begin{proposition}\label{prop:basicperspcuts}
The following cut 
\begin{align}\label{perspcut}
    \rho_i \geq (c_i+\Omega_i(\bar{x}_i))z_i+s_i(x_i-\bar{x}_i z_i)
\end{align}
is valid for the equivalent MINLO:
\begin{align*}
    \min_{\bm{z} \in \mathcal{Z}} \min_{\bm{x}, \bm{\rho} \in \mathbb{R}^n} \quad & f(\bm{x})+\sum_{i=1}^n \rho_i\\
    \text{s.t.} \quad & \bm{A}^i x_i \leq b_i z_i \quad \forall i \in [n],\nonumber\\
    & \rho_i \geq \Omega_i(x_i)+c_i z_i \quad \forall i \in [n].\nonumber
\end{align*}
\end{proposition}
\begin{remark}
In the special case where $\Omega_i(x_i)=x_i^2$, the cut reduces to:
\begin{align}
    \rho_i \geq 2 x_i \bar{x}_i-\bar{x}_i^2 z_i+c_i z_i \quad \forall \bar{x}_i.
\end{align}
\end{remark}
The class of cutting planes defined in Proposition \ref{prop:basicperspcuts} are commonly referred to as perspective cuts, because they define a linear lower approximation of the perspective function of $\Omega_i(x_i)$, $g_{\Omega_i}(x_i, z_i)$. Consequently, Proposition \ref{prop:basicperspcuts} implies that a perspective reformulation of \eqref{prob:origminlo} is equivalent to adding all (infinitely many) perspective cuts \eqref{perspcut}. This may be helpful where the original problem is nonlinear, as a sequence of linear MIOs can be easier to solve than one nonlinear MIO {\color{black}\citep[see][for a comparison]{frangioni2009computational}}.


\section{The Matrix Perspective Function and Its Applications}\label{sec:perspfns}
In this section, we generalize the perspective function from vectors to matrices, and invoke the matrix perspective function to propose a new technique for generating strong yet efficient relaxations of a diverse family of low-rank problems, which we call the Matrix Perspective Reformulation Technique (MPRT). Selected background on matrix analysis \citep[see][for a general theory]{bhatia2013matrix} and semidefinite optimization \cite[see][for a general theory]{wolkowicz2012handbook} which we use throughout this section can be found in Appendix \ref{sec:A.background}. 

\subsection{A Matrix Perspective Function}
To generalize the ideas from the previous section to low-rank constraints, we require a more expressive transform than the perspective transform, which introduces a single (scalar) additional degree of freedom and cannot control the eigenvalues of a matrix. Therefore, we invoke a generalization from quantum mechanics—the matrix perspective function defined in {\color{black}\cite{ebadian2011perspectives, effros2014non}}{, \color{black}building upon the work of \cite{effros2009matrix}}{\color{black}; see also \cite{marechal2001convexity, marechal2005functional1, marechal2005functional2, dacorogna2008role} for a related generalization of perspective functions to perspective functionals}.

\begin{definition}\label{defn:matrixconv}
For a matrix-valued function $f: \mathcal{X} \rightarrow \mathcal{S}^n_+$ where $\mathcal{X} \subseteq \mathcal{S}^n$ is a convex set, the {\color{black}matrix} perspective function of $f$, $g_f$, is defined as
\begin{align*}
    g_f(\bm{X},\bm{Y}) = \begin{cases} \bm{Y}^\frac{1}{2} f\left(\bm{Y}^{-\frac{1}{2}}\bm{X}\bm{Y}^{-\frac{1}{2}}\right)\bm{Y}^{\frac{1}{2}} & \mbox{ if } \bm{Y}^{-\frac{1}{2}}\bm{X}\bm{Y}^{-\frac{1}{2}} \in \mathcal{X}, \bm{Y} \succ \bm{0}, \\
    \infty & \mbox{ otherwise. }
    \end{cases}
\end{align*}
\end{definition}

\begin{remark}\blue If $\bm{X}$ and $\bm{Y}$ commute and $f$ is analytic, then Definition \ref{defn:matrixconv} simplifies into $\bm{Y} f\left(\bm{Y}^{-1}\bm{X}\right)$, which is the analog of the usual definition of the perspective function originally stated in \cite{effros2009matrix}. Definition \ref{defn:matrixconv}, however, generalizes this definition to the case where $\bm{X}$ and $\bm{Y}$ do not commute by ensuring that $\bm{Y}^{-\frac{1}{2}}\bm{X}\bm{Y}^{-\frac{1}{2}}$ is nonetheless symmetric, in a manner reminiscent of the development of interior point methods \citep[see, e.g.,][]{alizadeh1995interior}. In particular, if $\bm{Y}$ is a projection matrix such that $\bm{X}=\bm{Y}\bm{X}$–as occurs for the exact formulations of the low-rank problems we consider in this paper–then it is safe to assume that $\bm{X}, \bm{Y}$ commute. However, when $\bm{Y}$ is not a projection matrix, this cannot be assumed in general.
\end{remark}


The matrix perspective function generalizes the definition of the perspective transformation to matrix-valued functions and satisfies analogous properties:
\begin{proposition}\label{prop:genperspproperties}
Let $f$ be a matrix-valued function and $g_f$ its matrix perspective function. Then:
\begin{enumerate}[(a)]
    \item $f$ is matrix convex, {\color{black}i.e., 
\begin{align}
    t f(\bm{X})+(1-t) f(\bm{W}) \succeq f(t \bm{X}+(1-t)\bm{W}) \quad \forall \bm{X}, \bm{W} \in \mathcal{S}^n, \ t \in [0,1],
\end{align}} if and only if $g_f$ is matrix convex in $(\bm{X}, \bm{Y})$. 
    \item $g_f$ is a positive homogeneous function, i.e., for any $\mu {\blue >} 0$ we have
    \begin{align}
    g_f(\mu \bm{X}, \mu\bm{Y})=\mu g_f(\bm{X}, \bm{Y}).
    \end{align}
    \item Let $\bm{Y} \succ \bm{0}$ be a positive definite matrix. Then, letting the epigraph of $f$ be denoted by \begin{align}\mathrm{epi}(f):=\{(\bm{X}, \bm{\theta}): \bm{X} \in \mathrm{dom}(f), f(\bm{X}) \preceq \bm{\theta}\},\end{align} we have $(\bm{X}, \bm{Y}, \bm{\theta}) \in \mathrm{epi}(g_f)$ if and only if $(\bm{Y}^{-\frac{1}{2}}\bm{X}\bm{Y}^{-\frac{1}{2}}, \bm{Y}^{-\frac{1}{2}}\bm{\theta}\bm{Y}^{-\frac{1}{2}}) \in \mathrm{epi}(f)$.
\end{enumerate} 
\end{proposition}
\begin{proof}
We prove the claims successively:
\begin{enumerate}[(a)]
    \item This is precisely the main result of \citet[Theorem 2.2]{ebadian2011perspectives}.
    \item {\blue For $\mu > 0$,} 
    $g_f(\mu \bm{X}, \mu \bm{Y})=\mu \bm{Y}^{\frac{1}{2}} f\left((\mu\bm{Y})^{-\frac{1}{2}}\mu\bm{X}(\mu\bm{Y})^{-\frac{1}{2}}\right)\bm{Y}^{\frac{1}{2}}=\mu g_f(\bm{X}, \bm{Y})$. 
    \item By generalizing the main result in \citep[][Chapter 3.2.6]{boyd2004convex}, for any $\bm{Y} \succ \bm{0}$ we have that
    \begin{align*}
        (\bm{X}, \bm{Y}, \bm{\theta}) \in \mathrm{epi}(g_f) & \iff \quad \bm{Y}^\frac{1}{2}f(\bm{Y}^{-\frac{1}{2}}\bm{X}\bm{Y}^{-\frac{1}{2}})\bm{Y}^\frac{1}{2}\preceq \bm{\theta},\\
        & \iff \quad  f(\bm{Y}^{-\frac{1}{2}}\bm{X}\bm{Y}^{-\frac{1}{2}}) \preceq \bm{Y}^{-\frac{1}{2}}\bm{\theta}\bm{Y}^{-\frac{1}{2}},\\
        & \iff \quad (\bm{Y}^{-\frac{1}{2}}\bm{X}\bm{Y}^{-\frac{1}{2}}, \bm{Y}^{-\frac{1}{2}}\bm{\theta}\bm{Y}^{-\frac{1}{2}}) \in \mathrm{epi}(f). \quad \qed
    \end{align*}
\end{enumerate}
\end{proof}

{\color{black}
We now specialize our attention to matrix-valued functions defined by a scalar convex function, as suggested in the introduction.
}

{\blue 
\subsection{Matrix Perspectives of Operator Functions}
From any function $\omega : \mathbb{R} \rightarrow \mathbb{R}$, we can define its extension to the set of symmetric matrices, $f_\omega : \mathcal{S}^n \rightarrow \mathcal{S}^n$ as  
\begin{align}
    f_\omega(\bm{X}) = \bm{U} \operatorname{Diag}(\omega(\lambda_1^x),\dots, \omega(\lambda_n^x)) \bm{U}^\top,
\end{align}
where $\bm{X} = \bm{U} \operatorname{Diag}(\lambda_1^x,\dots, \lambda_n^x) \bm{U}^\top$ is an eigendecomposition of $\bm{X}$. Functions of this form are called \textit{operator functions} \citep[see][for a general theory]{bhatia2013matrix}. In particular, one can show that $f_\omega(\bm{X})$ is well-defined (does not depend explicitly on the eigenbasis of $\bm{X}$, $\bm{U}$). 
Among other examples, taking $\omega(x) = \exp(x)$ (resp. $\log(x)$) provides a matrix generalization of the exponential (resp. logarithm) function; see Appendix \ref{ssec:A.explog}.

Central to our analysis is that we can explicitly characterize the closure of the matrix perspective of $f_\omega$ under some assumptions on $\omega$, i.e., define by continuity $g_{f_\omega}(\bm{X},\bm{Y})$ for rank-deficient matrices $\bm{Y}$:
\begin{proposition}\label{prop:operatorperspective} Consider a function $\omega : \mathbb{R} \rightarrow \mathbb{R}$ satisfying Assumption \ref{ass:coercive}. Then, the closure of the matrix perspective of $f_\omega$ is, for any $\bm{X} \in \mathcal{S}^n$, $\bm{Y} \in \mathcal{S}_+^n$, 
\begin{align*}
    g_{f_\omega}(\bm{X},\bm{Y}) = \begin{cases} \bm{Y}^\frac{1}{2} f_\omega(\bm{Y}^{-\frac{1}{2}}\bm{X}\bm{Y}^{-\frac{1}{2}})\bm{Y}^{\frac{1}{2}} & \mbox{ if } \operatorname{Span}(\bm{X}) \subseteq \operatorname{Span}(\bm{Y}), \bm{Y} \succeq \bm{0}, \\
    \infty & \mbox{ otherwise, }
    \end{cases}
\end{align*}
where $\bm{Y}^{-\frac{1}{2}}$ denotes the pseudo-inverse of the square root of $\bm{Y}$.
\end{proposition}
\begin{remark} Note that in the expression of $g_{f_\omega}$ above, the matrix $\bm{Y}^{-\frac{1}{2}}\bm{X}\bm{Y}^{-\frac{1}{2}}$ is unambiguously defined if and only if $\operatorname{Span}(\bm{X}) \subseteq \operatorname{Span}(\bm{Y})$ (otherwise, its value depends on how we define the pseudo-inverse of $\bm{Y}^{\frac{1}{2}}$ outside of its range). Accordingly, in the remainder of the paper, we omit the condition $\operatorname{Span}(\bm{X}) \subseteq \operatorname{Span}(\bm{Y})$ whenever the analytic expression for $g_{f_\omega}$ explicitly involves $\bm{Y}^{-\frac{1}{2}}\bm{X}\bm{Y}^{-\frac{1}{2}}$.
\end{remark}
The proof of Proposition \ref{prop:operatorperspective} is deferred to Appendix \ref{ssec:A.proof.opepersp}. In the appendix, we also present an immediate extension where additional constraints, $\bm{X} \in \mathcal{X}$, are imposed on the argument of $f_\omega$. As in our prior work \citet{bertsimas2020mixed}, we reformulate the rank constraints in \eqref{prob:lrsdo} by introducing a projection matrix $\bm{Y}$ to encode for the span of $\bm{X}$. Naturally, $\bm{Y}$ should be rank-deficient. Hence, Proposition \ref{prop:operatorperspective} ensures that having $\operatorname{tr}(g_{f_\omega}(\bm{X},\bm{Y})) < \infty$ is a sufficient condition for $\bm{Y}$ to indeed control $\operatorname{Span}(\bm{X})$.

To gain intuition on how the matrix perspective function transforms $\bm{X}$ and $\bm{Y}$, we now provide
an interesting connection between the matrix perspective of $f_\omega$ and the perspective of $\omega$ in the case where $\bm{X}$ and  $\bm{Y}$ commute. 
\begin{proposition}\label{prop:persp.commute} Consider two matrices $\bm{X} \in \mathcal{S}^n, \bm{Y} \in \mathcal{S}^n_+$ that commute and such that $\mathrm{Span}(\bm{X}) \subseteq \mathrm{Span}(\bm{Y})$. Hence, there exists an orthogonal matrix $\bm{U}$ which jointly diagonalizes $\bm{X}$ and $\bm{Y}$. Let $\lambda_1^x, \dots, \lambda_n^x$ and  $\lambda_1^y, \dots, \lambda_n^y$ denote the eigenvalues of $\bm{X}$ and $\bm{Y}$ respectively, ordered according to this basis $\bm{U}$. Consider an operator function $f_\omega$ with $\omega$ satisfying Assumption \ref{ass:coercive}. Then, 
we have that:
\begin{align*}
    g_{f_\omega} (\bm{X}, \bm{Y}) = \bm{U} \operatorname{Diag}\left( g_\omega(\lambda_1^x, \lambda_1^y), \dots, g_\omega(\lambda_n^x, \lambda_n^y) \right) \bm{U}^\top
\end{align*}
\end{proposition}
\begin{proof} By simultaneously diagonalizing $\bm{X}$ and $\bm{Y}$, we get 
\begin{align*}
    \bm{Y}^{-\frac{1}{2}}\bm{X}\bm{Y}^{-\frac{1}{2}} &=\bm{U}\mathrm{Diag}\left( \lambda_1^x/\lambda_1^y, \dots, \lambda_n^x/\lambda_n^y \right)\bm{U}^\top, \\
    f_\omega\left( \bm{Y}^{-\frac{1}{2}}\bm{X}\bm{Y}^{-\frac{1}{2}} \right) &=\bm{U}\mathrm{Diag}\left( \omega(\lambda_1^x/\lambda_1^y), \dots, \omega(\lambda_n^x/\lambda_n^y) \right)\bm{U}^\top, \\
    \bm{Y}^{\frac{1}{2}} f_\omega\left( \bm{Y}^{-\frac{1}{2}}\bm{X}\bm{Y}^{-\frac{1}{2}} \right) \bm{Y}^{\frac{1}{2}} &=\bm{U}\mathrm{Diag}\left( \lambda_1^y \omega(\lambda_1^x/\lambda_1^y), \dots, \lambda_n^y \omega(\lambda_n^x/\lambda_n^y) \right)\bm{U}^\top. \quad \qed
\end{align*}
\end{proof}

Note that if $\bm{Y}$ is a projection matrix such that $\mathrm{Span}(\bm{X}) \subseteq \mathrm{Span}(\bm{Y})$ then we necessarily have that $\bm{X}=\bm{Y}\bm{X}=\bm{X}\bm{Y}$ and the assumptions of Proposition \ref{prop:persp.commute} hold.

In the general case where $\bm{X}$ and $\bm{Y}$ do not commute, we cannot simultaneously diagonalize them. However, we can still project $\bm{Y}$ onto the space of matrices that commute with $\bm{X}$. We show in Appendix \ref{ssec:A.persp.noncomm} that this is a trace preserving operation that can only reduce the value of $\operatorname{tr}\left( g_{f_\omega}(\bm{X}, \cdot) \right)$.
}

\subsection{The Matrix Perspective Reformulation Technique}\label{ssec:reformtech}
Definition \ref{defn:matrixconv} and Proposition \ref{prop:operatorperspective} supply the necessary language to lay out our Matrix Perspective Reformulation Technique (MPRT). Therefore, we now state the technique; details regarding its implementation will become clearer throughout the paper.

Let us revisit Problem
\eqref{prob:lrsdo}, and assume that the term $\Omega(\bm{X})$ satisfies the following properties:

\begin{assumption}
\label{assumption:seperability} 
$\Omega(\bm{X})= \operatorname{tr}\left( f_\omega(\bm{X}) \right)$, where $\omega$
is a {\color{black}function satisfying Assumption \ref{ass:coercive} and whose associated operator function, $f_\omega$, is matrix convex}.
\end{assumption}

{\blue Assumption \ref{assumption:seperability} implies that the regularizer can be rewritten as operating on the eigenvalues of $\bm{X}$, $\lambda_i(\bm{X})$, directly: $\Omega(\bm{X})= \sum_{i \in [n]} \omega(\lambda_i(\bm{X}))$.}
As we discuss in the next section, a broad class of functions satisfy this property. 
{\blue For ease of notation, we refer to $f_\omega$ as $f$ in the remainder of the paper (and accordingly denote by $g_f$ its matrix perspective function).}

After letting an orthogonal projection matrix $\bm{Y}$ model the rank of $\bm{X}$—as per \cite{bertsimas2020mixed}—Problem \eqref{prob:lrsdo} admits the equivalent mixed-projection reformulation:
\begin{align}\label{prob:lrsdo2}
    \min_{\bm{Y} \in \mathcal{Y}^k_n}\min_{\bm{X} \in \mathcal{S}^n} \quad & \langle \bm{C}, \bm{X} \rangle+\mu \cdot \mathrm{tr}(\bm{Y})+\mathrm{tr}(f(\bm{X})) \\
    \text{s.t.} \quad & \langle \bm{A}_i, \bm{X}\rangle=b_i \quad \forall i \in [m], \ \bm{X}=\bm{Y}\bm{X}, \ \bm{X} \in \mathcal{K}, \nonumber
\end{align}
where $\bm{Y} \in \mathcal{Y}^k_n$ is the set of $n \times n$ orthogonal projection matrices with trace at most $k$:
\begin{align*}
\mathcal{Y}^k_n := \left\lbrace \bm{Y} \in \mathcal{S}_+^n : \bm{Y}^2 = \bm{Y},\ \mathrm{tr}(\bm{Y}) \leq k \right\rbrace. 
\end{align*}
Note that for $k \in \mathbb{N}$, the convex hull of $\mathcal{Y}^k_n$ is given by $ \mathrm{Conv}(\mathcal{Y}^k_n)=\{\bm{Y} \in \mathcal{S}^n_{+}: \bm{Y} \preceq \mathbb{I}, \mathrm{tr}(\bm{Y}) \leq k\}$, which is a well-studied object in its own right \cite{overton1992sum,overton1993optimality, lewis1996convex,pataki1998rank}.

Since $\bm{Y}$ is an orthogonal projection matrix, imposing the nonlinear constraint $\bm{X}=\bm{Y}\bm{X}$ and introducing the term $\Omega(\bm{X})=\mathrm{tr}(f(\bm{X}))$ in the objective is equivalent to introducing the following term in the objective: $$\mathrm{tr}(g_f(\bm{X}, \bm{Y}))+(n-\mathrm{tr}(\bm{Y}))\omega(0),$$ where $g_f$ is the matrix perspective of $f$, and thus Problem \eqref{prob:lrsdo2} is equivalent to:
\begin{align}\label{prob:lrsdo3}
    \min_{\bm{Y} \in \mathcal{Y}^k_n}\min_{\bm{X} \in \mathcal{S}^n} \quad & \langle \bm{C}, \bm{X} \rangle+\mu \cdot \mathrm{tr}(\bm{Y})+\mathrm{tr}(g_{f}(\bm{X}, \bm{Y}))+(n-\mathrm{tr}(\bm{Y}))\omega(0) \\
    \text{s.t.} \quad & \langle \bm{A}_i, \bm{X}\rangle=b_i \quad \forall i \in [m], \ \bm{X} \in \mathcal{K}, \nonumber
\end{align}
Let us formally state and verify the equivalence between Problems \eqref{prob:lrsdo2}-\eqref{prob:lrsdo3} via:
\begin{theorem}\label{lemma:equivalence}
Problems \eqref{prob:lrsdo2}-\eqref{prob:lrsdo3} attain the same optimal objective value.
\end{theorem}

\begin{proof}
It suffices to show that for any feasible solution to \eqref{prob:lrsdo2} we can construct a feasible solution to \eqref{prob:lrsdo3} with an equal or lower cost, and vice versa:
\begin{itemize}
    \item Let $(\bm{X}, \bm{Y})$ be a feasible solution to \eqref{prob:lrsdo2}. {\blue Since $\bm{X} = \bm{Y} \bm{X} \in \mathcal{S}^n$, $\bm{X}$ and $\bm{Y}$ commute. Hence, by Proposition \ref{prop:persp.commute}, we have (using the same notation as in Proposition \ref{prop:persp.commute}): }
    \begin{align*}
    \operatorname{tr}\left( g_f(\bm{X}, \bm{Y}) \right) = \sum_{i \in [n]} g_\omega\left(\lambda^x_i,\lambda^y_i \right) = \sum_{i \in [n]} 1\{\lambda_i^y>0\}\omega(\lambda^x_i),
    \end{align*} 
    where $1\{\lambda_i^y>0\}$ is an indicator function which denotes whether the $i$th eigenvalue of $\bm{Y}$ {\blue(which is either 0 or 1)} is {\color{black}strictly positive}. {\blue Moreover, since $\bm{X} = \bm{Y}\bm{X}$, $\lambda_i^y = 0 \implies \lambda_i^x = 0$ and 
    \begin{align}
      \operatorname{tr}\left(f(\bm{X})\right) &= \sum_{i\in[n]} \omega(\lambda_i^x) = \operatorname{tr}\left( g_f(\bm{X}, \bm{Y}) \right) + \sum_{i \in [n]} 1\{\lambda_i^y=0\}\omega(0) \notag \\ &= \operatorname{tr}\left( g_f(\bm{X}, \bm{Y}) \right) + (n-\mathrm{tr}(\bm{Y}))\omega(0). \label{eqn:traceq}
    \end{align}
    }
    This establishes that $(\bm{X}, \bm{Y})$ is feasible in \eqref{prob:lrsdo3} with the same cost.
    \item Let $(\bm{X}, \bm{Y})$ be a feasible solution to \eqref{prob:lrsdo3}. Then, it follows that $\bm{X} \in \mathrm{Span}(\bm{Y})$, which implies that $\bm{X}=\bm{Y}\bm{X}$ since $\bm{Y}$ is a projection matrix. Therefore, {\blue\eqref{eqn:traceq}} holds, which establishes that $(\bm{X}, \bm{Y})$ is feasible in \eqref{prob:lrsdo2} with the same cost. \quad \qed
\end{itemize}
\end{proof}

Eventually, relaxing $\bm{Y} \in \mathcal{Y}^k_n$ in Problem \eqref{prob:lrsdo3} supplies as strong—and sometimes significantly stronger—relaxations than by any other technique we are aware of, as we explore in Section \ref{sec:convhulls}. 
{\blue 
\begin{remark}\label{rem:commuting} Note that, based on the proof of Theorem \ref{lemma:equivalence}, we could replace $g_f(\bm{X}, \bm{Y})$ in \eqref{prob:lrsdo3} by any function $\Tilde{g}(\bm{X}, \bm{Y})$ such that $g_f(\bm{X}, \bm{Y}) = \Tilde{g}(\bm{X}, \bm{Y})$ for $\bm{X}, \bm{Y}$ that commute, with no impact on the objective value. However, it might impact tractability if  $\Tilde{g}(\bm{X}, \bm{Y})$ is not convex in $(\bm{X}, \bm{Y})$.
\end{remark} }
\begin{remark}
{\blue Under Assumption \ref{assumption:seperability}, the regularization term $\Omega(\bm{X})$ penalizes all eigenvalues of $f_{\omega}(\bm{X})$ equally.} The MPRT can be extended to {\blue a wider class of regularization functions that penalize the largest eigenvalues more heavily}, 
at the price of (a significant amount of) additional notation. For brevity, we lay out this extension in Appendix \ref{sec:mprtextension}. 
\end{remark}

{\color{black}
Theorem \ref{lemma:equivalence} only uses the fact that $f$ is an operator function with $\omega$ satisfying Assumption \ref{ass:coercive}, not the fact that $f$ is matrix convex. In other words, \eqref{prob:lrsdo3} is always an equivalent reformulation of \eqref{prob:lrsdo2}.
An interesting question is to identify the set of necessary conditions for the objective of \eqref{prob:lrsdo3} to be convex in $(\bm{X},\bm{Y})$--$f$ being matrix convex is clearly sufficient. 
The objective in \eqref{prob:lrsdo3} is convex only as long as $\operatorname{tr}\left( g_f \right)$ is. Interestingly, this is not equivalent to the convexity of  $\mathrm{tr}(f)$. 
See Appendix \ref{append:counterexample} for a counter-example. 
It is, however, an open question whether a weaker notion than matrix convexity could ensure the joint convexity of $\mathrm{tr}(g_f)$.
It would also be interesting to investigate the benefits and the tractability of non-convex penalties (either by having $f$ not matrix convex or $\omega$ non-convex), given the successes of non-convex penalty functions in sparse regression problems \citep{zhang2010nearly,fan2001variable}. 
}

\subsection{Convex Hulls of Low-Rank Sets and the MPRT}
We now show that, for a general class of low-rank sets, applying the MPRT is equivalent to taking the convex hull of the set. This is significant, because we are not aware of any general-purpose techniques for taking convex hulls of low-rank sets. Formally, we have the following result:
\begin{theorem}\label{thm:lowrankset} {\blue Consider an operator function $f=f_\omega$ satisfying Assumption \ref{assumption:seperability}.} Let
\begin{align}
\mathcal{T}=\left\{\bm{X} \in \mathcal{S}^n: \mathrm{tr}(f(\bm{X}))+\mu \cdot \mathrm{Rank}(\bm{X}) \leq t, \mathrm{Rank}(\bm{X}) \leq k \right\}
\end{align}
be a set where $t \in \mathbb{R},k \in \mathbb{N}$ are fixed. 
Then, {\color{black}an extended formulation of }the convex hull of $\mathcal{T}$ is given by:
\begin{align}
\mathcal{T}^c=&\left\{(\bm{X}, \bm{Y}) \in \mathcal{S}^n \times \mathrm{Conv}(\mathcal{Y}^k_n): \mathrm{tr}(g_f(\bm{X, \bm{Y}}))+\mu\cdot \mathrm{tr}(\bm{Y})+(n-\mathrm{tr}(\bm{Y}))\omega(0) \leq t \right\}.
\end{align}
Where $\mathrm{Conv}(\mathcal{Y}^k_n)=\{\bm{Y} \in \mathcal{S}^n_+: \bm{Y} \preceq \mathbb{I}, \mathrm{tr}(\bm{Y}) \leq k \}$ is the convex hull of trace-$k$ projection matrices, and $g_f$ is the matrix perspective function of $f$.
\end{theorem}

\begin{proof}
We prove the two directions sequentially:
\begin{itemize}

\item $\mathrm{Conv}\left(\mathcal{T}\right) \subseteq \mathcal{T}^c$: let $\bm{X} \in \mathcal{T}$. Then, since the rank of $\bm{X}$ is at most $k$, there exists some $\bm{Y} \in \mathcal{Y}^k_n$ such that $\bm{X}=\bm{Y}\bm{X}$ and $\mathrm{tr}(\bm{Y})=\mathrm{Rank}(\bm{X})$. Moreover, by the same argument as in the proof of Theorem \ref{lemma:equivalence}, it follows that {\blue \eqref{eqn:traceq} holds and} $\mathrm{tr}(g_f(\bm{X}, \bm{Y}))+\mu \cdot \mathrm{tr}(\bm{Y})+(n-\mathrm{tr}(\bm{Y}))\omega(0) \leq t$, which confirms that $(\bm{X}, \bm{Y}) \in \mathcal{T}^c$. Since $\mathcal{T}^c$ is a convex set, we therefore have $\mathrm{Conv}\left(\mathcal{T}\right) \subseteq \mathcal{T}^c$.

\item $\mathcal{T}^c \subseteq \mathrm{Conv}\left(\mathcal{T}\right)$: let $(\bm{X}, \bm{Y}) \in \mathcal{T}^c$. {\blue Our proof uses Proposition \ref{prop:persp.commute}, which requires $\bm{X}$ and $\bm{Y}$ to commute. Let $\mathcal{X}$ denote the set of matrices that commute with $\bm{X}$: $\mathcal{X} := \{ \bm{M} \ : \ \bm{X M} = \bm{M X} \}$. Denote $\bm{Y}_{|\mathcal{X}}$ the projection of $\bm{Y}$ onto $\mathcal{X}$.}
{\blue By Lemma \ref{lemma:traceineq.proj}, we have that $\bm{Y}_{|\mathcal{X}} \in \operatorname{Conv}(\mathcal{Y}_n^k)$, and 
$\operatorname{tr}\left( g_{f}(\bm{X}, \bm{Y}_{|\mathcal{X}}) \right) \leq  \operatorname{tr}\left( g_{f}(\bm{X}, \bm{Y}) \right) < \infty$ so $(\bm{X}, \bm{Y}_{|\mathcal{X}}) \in \mathcal{T}^c$ as well. Hence, without loss of generality, by renaming $\bm{Y} \leftarrow \bm{Y}_{|\mathcal{X}}$, we can assume that $\bm{X}$ and $\bm{Y}$ commute.}
Then, it follows from 
Proposition \ref{prop:persp.commute} that {\blue the vectors of eigenvalues of $\bm{X}$ and $\bm{Y}$ (ordered according to a shared eigenbasis $\bm{U}$),} 
$\left(\bm{\lambda}(\bm{X}), \bm{\lambda}(\bm Y)\right)$ belong to the set
\begin{align*}
\left\{(\bm{x}, \bm{y}) \in \mathbb{R}^n \times [0,1]^n: \sum_i y_i \leq k,  \sum_{i=1}^n y_i \omega\left(\tfrac{x_i}{y_i}\right)+\mu \sum_i y_i +(n-\sum_i y_i)\omega(0) \leq t \right\},
\end{align*}
which, by \citep[][Lemma 6]{gunluk2010perspective}, is the convex hull of 
\begin{align*}
\mathcal{U}^c:=\left\{(\bm{x}, \bm{y}) \in \mathbb{R}^n \times \{0,1\}^n: \sum_i y_i \leq k,  \sum_{i=1}^n \omega\left({x_i}\right)+\mu \sum_i y_i \leq t, x_i = 0 \ \text{if} \ y_i = 0 \ \forall i \in [n]\right\}.
\end{align*}
Let us decompose $(\bm{\lambda}(\bm{X}), \bm{\lambda}(\bm Y))$ into $\bm{\lambda}(\bm{X}) = \sum_{k} \alpha_k \bm{x}^{(k)}$, $\lambda(\bm{Y}) = \sum_{k} \alpha_k \bm{y}^{(k)}$, with $\alpha_k \geq 0$, $\sum_k \alpha_k = 1$, and $(\bm{x}^{(k)}, \bm{y}^{(k)}) \in \mathcal{U}^c$. By definition, $$\bm{T}^{(k)} := \bm{U} \text{Diag}(\bm{x}^{(k)}) \bm{U}^\top \in \mathcal{T}$$ and $\bm{X} = \sum_k \alpha_k \bm{T}^{(k)}$.
Therefore, we have that $\bm{X} \in \mathrm{Conv}(\mathcal{T})$, as required.\quad \qed
\end{itemize}
\end{proof}

\begin{remark}
Since linear optimization problems over convex sets admit extremal optima, Theorem \ref{thm:lowrankset} demonstrates that \textit{unconstrained} low-rank problems with spectral objectives can be recast as linear semidefinite problems, where the rank constraint is dropped without loss of optimality. This suggests that work on hidden convexity in low-rank optimization, i.e., deriving conditions under which low-rank linear optimization problems admit exact relaxations where the rank constraint is omitted \citep[see, e.g.,][]{pataki1998rank, wang2019tightness, bertsimas2020solving}, could be extended to incorporate spectral functions.
\end{remark}

\subsection{Examples of the Matrix Perspective Function}\label{sec:examplesoffunctions}
Theorem \ref{thm:lowrankset} demonstrates that, for spectral functions under low-rank constraints, taking the matrix perspective is equivalent to taking the convex hull. To highlight the utility of Theorems \ref{lemma:equivalence}-\ref{thm:lowrankset}, we therefore supply the perspective functions of some spectral regularization functions which frequently arise in the low-rank matrix literature, and summarize them in Table \ref{tab:perspfncomparison}. We also discuss how these functions and their perspectives can be efficiently optimized over. {\color{black}Note that all functions introduced in this section are either matrix convex or the trace of a matrix convex function, and thus supply valid convex relaxations when used as regularizers for the MPRT.}

\paragraph{Spectral constraint:}
{\blue Let $\omega(x) = 0$ if $|x| \leq M$, $+\infty$ otherwise. Then,}
\begin{align*}
    f(\bm{X})=\begin{cases} {\blue \bm{0}} & \text{if} \quad \Vert \bm{X}\Vert_\sigma \leq M, \\ +\infty & \text{otherwise}, \end{cases}
\end{align*} 
for $\bm{X} \in \mathcal{S}^n${\color{black}, where $\Vert \cdot \Vert_\sigma$ denotes the spectral norm, i.e., the largest eigenvalue in absolute magnitude of $\bm{X}$. Observe that the condition $\Vert \bm{X}\Vert_\sigma \leq M$ can be expressed via semidefinite constraints$ -M\mathbb{I} \preceq \bm{X}\preceq M\mathbb{I}$}. The perspective function $g_f$ {\blue can then be expressed as}
\begin{align*}
    g_f(\bm{X}, \bm{Y})
    =\begin{cases} {\blue \bm{0}} & \text{if} \quad -M \bm{Y} \preceq \bm{X} \preceq M \bm{Y}, \\ +\infty & \text{otherwise}. \end{cases}
\end{align*} 
{\blue If $\bm{X}$ and $\bm{Y}$ commute, $g_f(\bm{X},\bm{Y})$} requires that $\vert \lambda_j(\bm{X}) \vert \leq M \lambda_j(\bm{Y}) \ \forall j \in [n]$--the spectral analog of a big-$M$ constraint. This constraint can be modeled using two semidefinite cones, and thus handled by semidefinite solvers.

\paragraph{Convex quadratic:} {\blue For $\omega(x) = x^2$,} $f(\bm{X})=\bm{X}^\top \bm{X}$. Then, the perspective function $g_f$ is
\begin{align*}
    g_f(\bm{X},\bm{Y}) = \begin{cases} \bm{X}^\top \bm{Y}^
    \dag \bm{X} & \mbox{ if } 
    \bm{Y} \succeq \bm{0}, \\
    +\infty & \mbox{ otherwise.}
    \end{cases}
\end{align*}
Observe that this function's epigraph is semidefinite-representable. Indeed, by the Schur complement lemma \citep[][Equation 2.41]{boyd1994linear}, minimizing the trace of $g_f(\bm{X}, \bm{Y})$ is equivalent to solving
\begin{align*}
    \min_{\bm{\theta} \in \mathcal{S}^n, \bm{Y} \in \mathcal{S}^n, \bm{X} \in \mathcal{S}^n}\quad \mathrm{tr}(\bm{\theta}) \quad \text{s.t.} \quad \begin{pmatrix} \bm{\theta} & \bm{X}\\ \bm{X}^\top & \bm{Y}\end{pmatrix} \succeq \bm{0}.
\end{align*}

Interestingly, this perspective function allows us to rewrite the 
rank-$k$ SVD problem
\begin{align*}
    \min_{\bm{X} \in \mathbb{R}^{n \times m}} \quad \Vert \bm{X}-\bm{A}\Vert_F^2: \ \mathrm{Rank}(\bm{X})\leq k
\end{align*}
as a linear optimization problem over the set of orthogonal projection matrices, which implies that the orthogonal projection constraint can be relaxed to its convex hull without loss of optimality (since some extremal solution will be optimal for the relaxation). This is significant, because while rank-$k$ SVD is commonly thought of as a non-convex problem which ``surprisingly'' admits a closed-form solution, the MPRT shows that it actually admits an \textit{exact} convex reformulation:
\begin{align*}
    \min_{\bm{X}, \bm{Y}, \bm{\theta}} \quad \frac{1}{2}\mathrm{tr}(\bm{\theta})-\langle \bm{A}, \bm{X}\rangle+\frac{1}{2}\Vert \bm{A}\Vert_F^2 \ \text{s.t.} \ \bm{Y} \preceq \mathbb{I}, \ \mathrm{tr}(\bm{Y}) \leq k, \begin{pmatrix} \bm{\theta} & \bm{X}\\ \bm{X}^\top & \bm{Y}\end{pmatrix} \succeq \bm{0}. 
\end{align*}
{\blue Note that, in the above formulation, we extended our results for symmetric matrices to rectangular matrices $\bm{X} \in \mathbb{R}^{n \times m}$ without justification. We rigorously derive this extension for $f(\bm{X}) = \bm{X}^\top \bm{X}$ in Appendix \ref{sec:A.rectangular} and defer the study of the general case to future research. }
\paragraph{Spectral plus convex quadratic:}
Let 
\begin{align*}
    f(\bm{X})=\begin{cases} \bm{X}^\top \bm{X} & \text{if} \quad \Vert \bm{X}\Vert_\sigma \leq M,\\ +\infty & \text{otherwise}, \end{cases}
\end{align*} 
for $\bm{X} \in \mathcal{S}^n$. Then, the perspective function $g_f$ is
\begin{align*}
    g_f(\bm{X}, \bm{Y})=\begin{cases} \bm{X}^\top \bm{Y}^\dag \bm{X} & \text{if} \quad -M \bm{Y} \preceq \bm{X} \preceq M \bm{Y}, \\ +\infty & \text{otherwise}. \end{cases}
\end{align*} 
This can be interpreted as the spectral analog of combining a big-$M$ and a ridge penalty.

\paragraph{Convex quadratic over completely positive cone:} {\blue Consider the following optimization problem}
\begin{align*}
    \blue \min_{\bm{X} \in \mathcal{S}^n} \: \bm{X}^\top \bm{X} \mbox{ s.t. } \bm{X} \in \mathcal{C}^n_+, 
\end{align*}
where $\mathcal{C}^n_+=\{\bm{X}: \bm{X}=\bm{U}\bm{U}^\top, \bm{U} \in \mathbb{R}^{n \times n}_+\} \subseteq \mathcal{S}^n_+$ denotes the completely positive cone. 
Then, {\blue by denoting $f(\bm{X}) = \bm{X}^\top \bm{X}$ and $g_f$ its} perspective function 
{\blue we obtain a valid relaxation by} minimizing $\mathrm{tr}(g_f)$, {\blue which, }
by the Schur complement lemma \citep[see][Equation 2.41]{boyd1994linear}, 
can be reformulated as
\begin{align*}
    \min_{\bm{\theta} \in \mathcal{S}^n, \bm{Y} \in \mathcal{S}^n, \bm{X} \in \mathcal{S}^n}\quad \mathrm{tr}(\bm{\theta}) \quad \text{s.t.} \quad \begin{pmatrix} \bm{\theta} & \bm{X}\\ \bm{X}^\top & \bm{Y}\end{pmatrix} \in \mathcal{S}^{2n}_+, \bm{X} \in \mathcal{C}^n_+.
\end{align*}
Unfortunately, this formulation cannot be tractably optimized over, since separating over the completely positive cone is NP-hard. However, by relaxing the completely positive cone to the doubly non-negative cone—$\mathcal{S}^n_+ \cap \mathbb{R}^{n \times n}_{+}$—we obtain a tractable and near-exact relaxation. Indeed, as we shall see in our numerical experiments, combining this relaxation with a state-of-the-art heuristic supplies certifiably near-optimal solutions in both theory and practice.

{\blue Note that we could have obtained an alternative relaxation by instead considering the perspective of}
\begin{align*}
    f(\bm{X})=\begin{cases} \bm{X}^\top \bm{X} & \text{if} \quad \bm{X} \in \mathcal{C}^n_+,\\ +\infty & \text{otherwise}. \end{cases}
\end{align*} 

\begin{remark}
One can obtain a nearly identical formulation over the copositive cone \citep[c.f.][]{burer2009copositive}.
\end{remark}

\paragraph{Power:} Let\footnote{Note that $f(\bm{X})$ and its perspective are concave functions; hence we model their hypographs, not epigraphs.} $f(\bm{X})=\bm{X}^\alpha$ for $\alpha \in[0, 1]$ and $\bm{X} \in \mathcal{S}^n_+$. The matrix perspective function is\footnote{We only consider the PSD case for notational convenience. However, the symmetric case follows in much the same manner, after splitting $\bm{X}=\bm{X}_{+}-\bm{X}_{-}: \bm{X}_{+}, \bm{X}_{-} \succeq \bm{0}, \langle \bm{X}_{+}, \bm{X}_{-}\rangle=0$ and replacing $\bm{X}$ with $\bm{X}_{+}+\bm{X}_{-}$.}
\begin{align*}
    g_f(\bm{X},\bm{Y}) = \begin{cases} \bm{Y}^\frac{1-\alpha}{2}\bm{X}^\alpha\bm{Y}^\frac{1-\alpha}{2} & \mbox{ if } {\color{black}\bm{Y}^\frac{-1}{2} \bm{X}\bm{Y}^\frac{-1}{2}} \in \mathcal{S}^n_+, \bm{Y} \succeq \bm{0}, \\
    +\infty & \mbox{ otherwise.}
    \end{cases}
\end{align*}

\begin{remark}[Matrix Power Cone]
This function's epigraph, the matrix power cone, i.e.,
\begin{align*}
    \mathcal{K}^{\text{pow}, \alpha}_{\text{mat}}=\{(\bm{X}_1, \bm{X}_2, \bm{X}_3) \in \mathcal{S}^n_+\times \mathcal{S}^n_{+} \times \mathcal{S}^n: \bm{X}_2^{\frac{1-\alpha}{2}}\bm{X}_1^\alpha \bm{X}_2^{\frac{1-\alpha}{2}} \succeq 
    \bm{X}_{3,+}+\bm{X}_{3,-}\}
\end{align*}
is a closed convex cone which is semidefinite representable for any rational $\alpha$ \cite{fawzi2017lieb}. 
Consequently, it is a tractable object which successfully models the matrix power function (and its perspective) and we shall make repeated use of it when we apply the MPRT to several important low-rank problems in Section \ref{sec:examplesoffunctions}.
\end{remark}

\paragraph{Logarithm:}
Let $f(\bm{X})=-\log(\bm{X})$ be the matrix logarithm function. 
We have that
\begin{align*}
    g_f(\bm{X}, \bm{Y})=\begin{cases}
    - \bm{Y}^\frac{1}{2}\log\left(\bm{Y}^{-\frac{1}{2}}\bm{X}\bm{Y}^{-\frac{1}{2}}\right)\bm{Y}^\frac{1}{2} & \mbox{ if } \bm{X}, \bm{Y} \succ \bm{0}, \\
    +\infty & \mbox{ otherwise.}
    \end{cases}
\end{align*}
{\blue Observe that when $\bm{X}$ and $\bm{Y}$ commute, $g_f(\bm{X},\bm{Y})$ can be rewritten as $\bm{Y}(\log(\bm{Y}) - \log(\bm{X}))$, which} 
is the quantum relative entropy function \cite[see][for a general theory]{fawzi2019semidefinite}. We remark that the domain of $\log(\bm{X})$ requires that $\bm{X}$ is full-rank, which at a first glance makes the use of this function problematic for low-rank optimization. Accordingly, we consider the $\epsilon-$logarithm function, i.e., $\log_{\epsilon}(\bm{X})=\log(\bm{X} + \epsilon \mathbb{I})$ 
for $\epsilon>0$, 
as advocated by \citet{fazel2003log} in a different context. Note that background on the matrix exponential and logarithm functions can be found in Appendix \ref{sec:A.background}.

Observe that $\mathrm{tr}(\log(\bm{X}))=\log\det(\bm{X})$ while $\mathrm{tr}(g_f)=\mathrm{tr}(\bm{X}(\log(\bm{X})-\log(\bm{Y}))$. Thus, the matrix logarithm and its trace verify the concavity of the logdet function—which has numerous applications in low-rank problems \cite{fazel2003log} and interior point methods \cite{renegar2001mathematical} among others—while the perspective of the matrix logarithm provides an elementary proof of the convexity of the quantum relative entropy: a task for which perspective-free proofs are technically demanding \cite{effros2009matrix}. 

\paragraph{Von Neumann entropy:}
Let {\color{black} $f(\bm{X})=\bm{X}\log(\bm{X})$} denote the von Neumann quantum entropy of a density matrix $\bm{X}$. Then, its perspective function is $\blue g_f(\bm{X}, \bm{Y})= \bm{X}\bm{Y}^{-\frac{1}{2}}\log(\bm{Y}^{-\frac{1}{2}}\bm{X}\bm{Y}^{-\frac{1}{2}})\bm{Y}^\frac{1}{2}$. {\blue When $\bm{X}$ and $\bm{Y}$ commute, this perspective can be equivalently written as}
\begin{align*}
    g_f(\bm{X}, \bm{Y})=\begin{cases}
    \bm{X}^\frac{1}{2}\log(\bm{Y}^{-\frac{1}{2}}\bm{X}\bm{Y}^{-\frac{1}{2}})\bm{X}^\frac{1}{2} & \mbox{ if } \bm{X}, \bm{Y} \succ \bm{0}, \\
    +\infty & \mbox{ otherwise.}
    \end{cases}
\end{align*}
which is referred to as the Umegaski relative entropy or the matrix Kullback-Leibler divergence in the literature. {\blue Note that various generalizations of the relative entropy for matrices have been proposed in the quantum physics literature \citep{hiai1991proper}. However, these different definitions agree on the set of commuting matrices, hence can be used interchangeably for optimization purposes (see Remark \ref{rem:commuting}).}
\begin{remark}[Quantum relative entropy cone]
Note the epigraph of $g_f$, namely,
\begin{align*}
    \mathcal{K}^{\text{op, rel}}_{\text{mat}}=\{(\bm{X}_1, \bm{X}_2, \bm{X}_3) \in \mathcal{S}^n \times \mathcal{S}^n_{++} \times \mathcal{S}^n_{++}: \bm{X}_1 \succeq -\bm{X}_2^{\frac{1}{2}}\log(\bm{X}_2^{-\frac{1}{2}}\bm{X}_3\bm{X}_2^{-\frac{1}{2}})\bm{X}_2^{\frac{1}{2}}\},
\end{align*}
is a 
convex cone which can be approximated using semidefinite cones and optimized over using either the \verb|Matlab| package \verb|CVXQuad| (see \cite{fawzi2019semidefinite}), or optimized over directly using an interior point method for asymmetric cones \cite{karimi2019domain}\footnote{Specifically, if we are interested in quantum relative entropy problems where we minimize the trace of $\bm{X}_1$, as occurs in the context of the MPRT, we may achieve this using the domain-driven solver developed by \cite{karimi2019domain}. However, we are not aware of any IPMs which can currently optimize over the full quantum relative entropy cone.}.
Consequently, this is a tractable object which models the matrix logarithm and Von Neumann entropy (and their perspectives).
\end{remark}

Finally, Table \ref{tab:perspfncomparison} relates the matrix perspectives discussed above with their scalar analogs. 

\begin{table}[h!]
\centering\footnotesize
\caption{Analogy between perspectives of scalars and perspectives of matrix convex functions. }
\begin{tabular}{@{}l l l l l l l @{}} \toprule
 & \multicolumn{3}{c@{\hspace{0mm}}}{Perspective of function} &\multicolumn{3}{c@{\hspace{0mm}}}{Matrix perspective of function} \\
 \cmidrule(l){2-4} \cmidrule(l){5-7}  Type & $f(x): \mathbb{R} \rightarrow \mathbb{R}$ & $g_f(\bm{x}, t)$ & Ref. & $f$ & $g_f$ & Ref. \\\midrule
Quadratic & $x^2$ & $x^2/t$ & \citep{ben2001lectures} & $\bm{X}^\top \bm{X}$ & $\bm{X}^\top \bm{Y}^\dag \bm{X}$ & \citep{bertsimas2020mixed}\\
Power & $-x^\alpha: 0 < \alpha < 1$ & $-x^\alpha t^{1-\alpha}$ & \cite{boyd2004convex} & $-\bm{X}^\alpha$ & $-\bm{Y}^\frac{1-\alpha}{2}\bm{X}^\alpha \bm{Y}^\frac{1-\alpha}{2}$  & Prop. \ref{prop:operatorperspective}\\
  Log & $-\log(x)$ & $-t\log(\frac{x}{t})$ & \cite{boyd2004convex} & $-\log(\bm{X})$ & $-\bm{Y}^\frac{1}{2}\log\left(\bm{X}^{-\frac{1}{2}}\bm{Y}\bm{X}^{-\frac{1}{2}}\right)\bm{Y}^\frac{1}{2}$ & \cite{fawzi2019semidefinite}\\
 Entropy & $x\log(x)$ & $x \log(\frac{x}{t})$ & \cite{boyd2004convex} &\color{black} $\bm{X}\log(\bm{X})$ & $\bm{X}^\frac{1}{2}\log(\bm{Y}^{-\frac{1}{2}}\bm{X}\bm{Y}^{-\frac{1}{2}})\bm{X}^\frac{1}{2}$ & \cite{lieb1973proof, effros2009matrix}\\
\bottomrule
\end{tabular}
\label{tab:perspfncomparison}
\end{table}

\subsection{{\color{black}Matrix} Perspective Cuts}
We now generalize the perspective cuts of \cite{frangioni2006perspective, gunluk2010perspective} from vectors to matrices and cardinality to rank constraints. Let us reconsider the previously defined mixed-projection optimization problem:
\begin{align*}
    \min_{\bm{Y} \in \mathcal{Y}^k_n} \min_{\bm{X} \in \mathcal{S}^n_+} \quad & \langle \bm{C}, \bm{X}\rangle+\mu \cdot \mathrm{tr}(\bm{Y})+\mathrm{tr}(f(\bm{X}))\ \text{s.t.} \ \langle \bm{A}_i, \bm{X}\rangle=b_i \quad \forall i \in [m], \ \bm{X}=\bm{Y}\bm{X}, \bm{X}\in \mathcal{K},
\end{align*}
{\color{black}where similarly to \cite{frangioni2006perspective} we assume that $f(\bm{0})=\bm{0}$ to simplify the cut derivation procedure.} Letting $\bm{\theta}$ model the epigraph of {\color{black}$f$} via $\bm{\theta} \succeq f(\bm{X})$ and $\bm{S}$ be a subgradient of {\color{black}$f$} at $\bar{\bm{X}}$, we have:
\begin{align}\label{eqn:matrixperspective}
    \bm{\theta} \succeq f(\bar{\bm{X}})\bm{Y}+\bm{S}^{\top}(\bm{X}-\bar{\bm{X}}\bm{Y}),
\end{align}
which if $f(\bm{X})=\bm{X}^2$ —as discussed previously—reduces to 
\begin{align*}
    \bm{\theta}^i \succeq \bar{\bm{X}}(2\bm{X}-\bar{\bm{X}}\bm{Y}),
\end{align*}
which is precisely the analog of perspective cuts in the vector case. Note however that these cuts require semidefinite constraints to impose, which suggests they may not be as practically useful. For instance, our prior work \cite{bertsimas2020mixed}'s outer-approximation scheme for low-rank problems has a non-convex QCQOP master problem, which can only be currently solved using \verb|Gurobi|, while \verb|Gurobi| currently does not support semidefinite constraints. 

We remark however that the inner product of Equation \eqref{eqn:matrixperspective} with an arbitrary PSD matrix supplies a valid linear inequality. Two interesting cases of this observation arise when we take the inner product of the cut with either a rank-one matrix or the identity matrix.

Taking an inner product with the identity matrix supplies the inequality:
\begin{align}
    \mathrm{tr}(\bm{\theta}) \geq \langle f(\bar{\bm{X}}), \bm{Y}\rangle+\langle\bm{S}, \bm{X}-\bar{\bm{X}}\bm{Y}\rangle \quad \forall \bm{Y} \in \mathcal{Y}^k_n.
\end{align}
Moreover, by analogy to \citep[Section 3.4]{bertsimas2019unified}, if we ``project out'' the $\bm{X}$ variables by decomposing the problem into a master problem in $\bm{Y}$ and subproblems in $\bm{X}$ then this cut becomes the Generalized Benders Decomposition cuts derived in our prior work \citep[Equation {\color{black}(17)}]{bertsimas2020mixed}.

Alternatively, taking the inner product of the cut with a rank-one matrix $\bm{b}\bm{b}^\top$ gives:
\begin{align*}
    {\color{black}\bm{b}^\top \bm{\theta}\bm{b}\geq \bm{b}^\top \left(f(\bar{\bm{X}})\bm{Y}+\bm{S}^{\top}(\bm{X}-\bar{\bm{X}}\bm{Y})\right) \bm{b}.}
\end{align*}

A further improvement is actually possible: rather than requiring that the semidefinite inequality is non-negative with respect to one rank-one matrix, we can require that it is simultaneously non-negative in the directions $\bm{v}^1$ and $\bm{v}^2$. This supplies the second-order cone \citep[][Eqn. (8)]{permenter2018partial} cut:
\begin{align*}
      \begin{pmatrix}
          \bm{v}^1\\ \bm{v}^2
      \end{pmatrix}^\top \left( \bm{\theta}-f(\bar{\bm{X}})\bm{Y}-\bm{S}^{ \top}(\bm{X}-\bar{\bm{X}}\bm{Y})\right) \begin{pmatrix}
          \bm{v}^1 \\ \bm{v}^2
      \end{pmatrix}^\top\succeq \begin{pmatrix} 0 & 0\\ 0 & 0\end{pmatrix}.
\end{align*}
    
The analysis in this section suggests that applying a perspective cut decomposition scheme out-of-the-box may be impractical, but leaves the door open to {\color{black}adaptations} of the scheme which account for the projection matrix structure.

\section{Examples and Perspective Relaxations} \label{sec:convhulls}
In this section, we apply the MRPT to several important low-rank problems, in addition to the previously discussed reduced-rank regression problem (Section \ref{sec:motivexample}). We also recall Theorem \ref{thm:lowrankset} to demonstrate that applying the MPRT to spectral functions which feature in these problems actually gives the convex hull of relevant substructures.

\subsection{Matrix Completion} 
Given a sample $(A_{i,j} : (i, j) \in \mathcal{I} \subseteq [n] \times [n])$  of a
matrix $\bm{A} \in \mathcal{S}^n_+$, the matrix completion problem is to reconstruct the entire matrix, by assuming $\bm{A}$ is approximately low-rank \cite{candes2009exact}. Letting $\mu, \gamma >0$ be penalty multipliers, this problem admits the formulation:
\begin{align}
    \min_{\bm{X} \in \mathcal{S}^n_+} \quad & \sum_{(i,j) \in \mathcal{I}}(X_{i,j}-A_{i,j})^2+\frac{1}{2\gamma}\Vert \bm{X}\Vert_F^2+\mu \cdot \mathrm{Rank}(\bm{X}).
\end{align}

Applying the MPRT to the $\Vert \bm{X}\Vert_F^2=\mathrm{tr}(\bm{X}^\top \bm{X})$ term demonstrates that this problem is equivalent to the mixed-projection problem:
\begin{align*}
    \min_{\bm{X}, \bm{\theta} \in \mathcal{S}^n_+, \bm{Y} \in \mathcal{Y}^n_n} \quad & \sum_{(i,j) \in \mathcal{I}}(X_{i,j}-A_{i,j})^2+\frac{1}{2\gamma}\mathrm{tr}(\bm{\theta})+\mu \cdot \mathrm{tr}(\bm{Y})\quad \text{s.t.} \quad \begin{pmatrix} \bm{Y} & \bm{X}\\ \bm{X} & \bm{\theta} \end{pmatrix} \succeq \bm{0},
\end{align*}
and relaxing $\bm{Y} \in \mathcal{Y}^n_n$ to $\bm{Y} \in \mathrm{Conv}(\mathcal{Y}^n_n) =\{\bm{Y} \in \mathcal{S}^n: \bm{0} \preceq \bm{Y} \preceq \mathbb{I}\}$ supplies a valid relaxation. We now argue that this relaxation is often high-quality, by demonstrating that the MPRT supplies the convex envelope of $t\geq \frac{1}{2\gamma}\Vert \bm{X}\Vert_F^2+\mu \cdot\mathrm{Rank}(\bm{X})$, via the following corollary to Theorem \ref{thm:lowrankset}:
\begin{corollary}\label{lemma:initlemma}
\begin{align*}
\text{Let} \quad \mathcal{S}=\left\{(\bm{Y}, \bm{X}, \bm{\theta}) \in \mathcal{Y}^k_n \times \mathcal{S}^n_+ \times \mathcal{S}^n: \bm{\theta} \succeq \bm{X}^\top \bm{X}, u \bm{Y} \succeq \bm{X}\succeq \ell \bm{Y} \right\}
\end{align*}
be a set where $\ell, u \in \mathbb{R}_+$. Then, this set's convex hull is given by:
\begin{align*}
\mathcal{S}^c=\left\{(\bm{Y}, \bm{X}, \bm{\theta}) \in \mathcal{S}^n_+ \times \mathcal{S}^n_+ \times \mathcal{S}^n:\bm{Y} \preceq \mathbb{I}, \mathrm{tr}(\bm{Y}) \leq k, u \bm{Y} \succeq \bm{X}\succeq \ell \bm{Y}, \begin{pmatrix} \bm{Y} & \bm{X} \\ \bm{X}^\top & \bm{\theta}\end{pmatrix}\succeq \bm{0} \right\}.
\end{align*}
\end{corollary}

\subsection{Tensor Completion}
A central problem in machine learning is to reconstruct a $d$-tensor $\xt$ given a subsample of its entries $(A_{i_1, \ldots i_d}: (i_1, \ldots i_d) \in \mathcal{I} \subseteq [n_1] \times [n_2] \times \ldots \times [n_d])$, by assuming that the tensor is low-rank. Since even evaluating the rank of a tensor is NP-hard \cite{kolda2009tensor}, a popular approach for solving this problem is to minimize the reconstruction error while constraining the ranks of different unfoldings of the tensor \citep[see, e.g.,][]{gandy2011tensor}. After imposing Frobenius norm regularization and letting $\Vert \cdot \Vert_{HS}=\sqrt{\sum_{i_{1}=1}^{n_{1}} \ldots \sum_{{i_d}=1}^{n_d} X_{i_1, \ldots, i_d}^2}$ denote the (second-order cone representable) Hilbert-Schmidt norm of a tensor, this leads to optimization problems of the form:
\begin{align}\label{prob:modenunfold_tensor}
    \min_{\mathscr{X} \in \mathbb{R}^{n_1 \times \ldots \times n_d}} \quad {\color{black}\sum_{({i_1, \ldots i_d}) \in \mathcal{I}}\left(\mathscr{A}_{i_1, \ldots i_d}-\mathscr{X}_{i_1, \ldots i_d}\right)^2}+\sum_{i=1}^n \Vert \mathscr{X}_{(i)}\Vert_F^2 \quad \text{s.t.}\quad \mathrm{Rank}(\mathscr{X}_{(i)}) \leq k \quad \forall i \in [n].
\end{align}
Similarly to low-rank matrix completion, it is tempting to apply the MRPT to model the $\bm{X}_{(i)}^\top \bm{X}_{(i)}$ term for each mode-$n$ unfolding. We now demonstrate this supplies {\color{black}a tight approximation of }the convex hull of the sum of the regularizers, via the following lemma (proof omitted, follows in the spirit of \citep[Lemma 4]{gunluk2010perspective}):
\begin{lemma}\label{lemma:sumofsquaresoflowrankmatrices}
\begin{align*}
\text{Let} \quad \mathcal{Q}=\left\{(\rho, \bm{Y}_1, \ldots, \bm{Y}_m, \bm{X}_1, \ldots, \bm{X}_m, \bm{\theta}_1, \ldots, \bm{\theta}_m): \rho \geq \sum_{i=1}^m q_i \mathrm{tr}(\bm{\theta}_i), (\bm{X}_i, \bm{Y}_i, \bm{\theta}_i) \in \mathcal{S}^i \ \forall i \in [m] \right\}    
\end{align*}
be a set where $l_i, u_i, q_i \in \mathbb{R}^n_+ \ \forall i \in [m]$, and $\mathcal{S}_i$ is a set of the same form as $\mathcal{S}$, but $l,u$ are replaced by $l_i, u_i$. Then, {\color{black}an extended formulation of }this set's convex hull is given by:
\begin{align*}
\mathcal{Q}^c=\left\{(\rho, \bm{Y}_1, \ldots, \bm{Y}_m, \bm{X}_1, \ldots, \bm{X}_m, \bm{\theta}_1, \ldots, \bm{\theta}_m): \rho \geq \sum_{i=1}^m q_i \mathrm{tr}(\bm{\theta}_i), (\bm{X}_i, \bm{Y}_i, \bm{\theta}_i) \in \mathcal{S}^c_i \ \forall i \in [m] \right\}.
\end{align*}
\end{lemma}

Lemma \ref{lemma:sumofsquaresoflowrankmatrices} suggests that the MPRT may improve algorithms which aim to recover tensors of low slice rank. For instance, in low-rank tensor problems where \eqref{prob:modenunfold_tensor} admits multiple local solutions, solving the convex relaxation coming from $\mathcal{Q}^c$ 
and greedily rounding may give a high-quality initial point for an alternating minimization method such as the method of \cite{farias2019learning}, and indeed allow such a strategy to return better solutions than if it were initialized at a random point. 

{\color{black}Note however that Lemma \ref{lemma:sumofsquaresoflowrankmatrices} does not necessarily give the convex hull of the sum of the regularizers, since the regularization terms involve different slices of the same tensor and thus interact; see also \cite{romera2013new} for a related proof that the tensor trace norm does not give the convex envelope of the sum of ranks of slices}.

\subsection{Low-Rank Factor Analysis}
An important problem in statistics, psychometrics and economics is to decompose a covariance matrix $\bm{\Sigma} \in \mathcal{S}^n_+$ into a low-rank matrix $\bm{X} \in \mathcal{S}^n_+$ plus a diagonal matrix $\bm{\Phi} \in \mathcal{S}^n_+$, as explored by \citet{bertsimas2017certifiably} and references therein. This corresponds to {\color{black}solving}:
\begin{align}
    \min_{\bm{X}, \bm{\Phi} \in \mathcal{S}^n_+} \ \Vert \bm{\Sigma}-\bm{\Phi}-\bm{X}\Vert_q^q \ \text{s.t.}\ \mathrm{Rank}(\bm{X}) \leq k, \ \Phi_{i,j}=0, \forall i, j \in [n]: i \neq j, \ \Vert \bm{X}\Vert_\sigma \leq M
\end{align}
where $q\geq 1$, $\Vert \bm{X}
\Vert_q=
\left(\sum_{i=1}^n \lambda_i(\bm{X})^q\right)^\frac{1}{q}$ denotes the matrix $q$-norm, and we constrain the spectral norm of $\bm{X}$ via a big-$M$ constraint for the sake of tractability.

This problem's objective involves minimizing $\mathrm{tr}\left(\bm{\Sigma}-\bm{\Phi}-\bm{X}\right)^q$, and it is not immediately obvious how to either apply the technique in the presence of the $\bm{\Phi}$ variables or alternatively seperate out the $\bm{\Phi}$ term and apply the MPRT to an appropriate ($\bm{\Phi}$-free) substructure. To proceed, let us therefore first consider its scalar analog, obtaining the convex closure of the following set:
\begin{align*}
    \mathcal{T}=\{(x, y, z, t) \in \mathbb{R} \times \mathbb{R} \times \{0, 1\} \times \mathbb{R}^+: t \geq \vert x+y-d \vert^q, \vert x\vert \leq M, x=0 \ \mbox{if} \ z=0\},
\end{align*}
where $d \in \mathbb{R}$ and $q \geq 1$ are fixed constants, and we require that $\vert x\vert \leq M$ for the sake of tractability. We obtain the convex closure via the following proposition (proof deferred to Appendix \ref{append:proofs}):
\begin{proposition}\label{prop:powerconecl}
The convex closure of the set $\mathcal{T}$, $\mathcal{T}^c$, is given by:
\begin{align*}
    \mathcal{T}^c=\bigg\{(x, y, z, t) \in \mathbb{R} \times \mathbb{R} \times [0, 1] \times \mathbb{R}^+: \exists \beta \geq 0: \ t \geq \frac{\vert y-\beta-d(1-z)\vert^q}{(1-z)^{q-1}}+\frac{\vert x+\beta-d z \vert^q}{z^{q-1}}, \ \vert x \vert \leq Mz\bigg\}.
\end{align*}
\end{proposition}

\begin{remark}
To check that this set is indeed a valid convex relaxation, observe that if $z=0$ then $x=0$ and $x=-\beta \implies \beta=0$ and $t \geq \vert y-d\vert^q$, while if $z=1$ then $y=\beta$ and $t \geq \vert x+y-d\vert^q$.
\end{remark}

Observe that $\mathcal{T}^c$ can be modeled using two power cones and one inequality constraint. 

Proposition \ref{prop:powerconecl} suggests that we can obtain high-quality convex relaxations for low-rank factor analysis problems via a judicious use of the matrix power cone. Namely, introduce an epigraph matrix $\bm{\theta}$ to model the eigenvalues of $(\bm{\Sigma}-\bm{\Phi}-\bm{X})^q$ and an orthogonal projection matrix $\bm{Y}_2$ to model the span of $\bm{X}$. This then leads to the following matrix power cone representable relaxation:
\begin{align*}
   \min_{\bm{X}, \bm{\Phi}, \bm{\theta}, \bm{Y}_1, \bm{Y}_2 \in \mathcal{S}^n_+, \bm{\beta} \in \mathcal{S}^n} \quad & \mathrm{tr}(\bm{\theta})\\
  \text{s.t.} \quad & \bm{\theta} \succeq \bm{Y}_1^{\frac{1-q}{2}}(\bm{Y}_1^\frac{1}{2}\bm{\Sigma}\bm{Y}_1^\frac{1}{2}-\bm{\beta}-\bm{\Phi})\bm{Y}_1^{\frac{1-q}{2}}+\bm{Y}_2^{\frac{1-q}{2}}(\bm{Y}_2^\frac{1}{2}\bm{\Sigma}\bm{Y}_2^\frac{1}{2}+\bm{\beta}-\bm{X})\bm{Y}_2^{\frac{1-q}{2}},\nonumber\\
   & \bm{Y}_1+\bm{Y}_2=\mathbb{I}, \mathrm{tr}(\bm{Y}) \leq k, \Phi_{i,j}=0, \forall i, j \in [n]: i \neq j, \nonumber\\
   & \bm{\Phi} \preceq \bm{X}, \bm{X}\preceq M\bm{Y}_2, -\bm{X}\preceq M\bm{Y}_2. \nonumber
\end{align*}

\subsection{Optimal Experimental Design}


Letting $\bm{A} \in \mathbb{R}^{n \times m}$ where $m \geq n$ be a matrix of linear measurements of the form $y_i=\bm{a}_i^\top \bm{\beta}+\epsilon_i$ from an experimental setting, the D-optimal experimental design problem (a.k.a. the sensor selection problem) is to pick $k \leq m$ of these experiments in order to make the most accurate estimate of $\bm{\beta}$ possible, by solving \citep[see][for a modern approach]{joshi2008sensor, singh2018approximation}:
\begin{align}\label{eqn:doptdesign1}
    \max_{\bm{z} \in \{0, 1\}^n: \bm{e}^\top \bm{z} \leq k} \quad & \log \det_{\epsilon}\left(\sum_{i \in [n]}z_i \bm{a}_i \bm{a}_i^\top \right),
\end{align}
where we define $\log \det_{\epsilon}(\bm{X})= \blue \log\det(\bm{X}+\epsilon \mathbb{I})$ for $\epsilon>0$ to be the pseudo log-determinant of a rank-deficient PSD matrix, which can be thought of as imposing an uninformative prior of importance $\epsilon$ on the experimental design process. Since $\log\det(\bm{X})=\mathrm{tr}(\log(\bm{X}))$, 
a valid convex relaxation is given by:
\begin{align*}
    \max_{\bm{z} \in [0, 1]^n, \bm{\theta} \in \mathcal{S}^n_+}\quad \mathrm{tr}(\bm{\theta}) \quad \text{s.t.} \quad \log\left(\bm{A}\mathrm{Diag}(\bm{z})\bm{A}^\top+\epsilon\mathbb{I}\right)\succeq \bm{\theta},
\end{align*}
which can be modeled using the quantum relative entropy cone, via $(-\bm{\theta}, \mathbb{I}, \bm{A}\mathrm{Diag}(\bm{z})\bm{A}^\top+\epsilon\mathbb{I}) \in \mathcal{K}^{\text{rel, op}}_{\text{mat}}$. This is equivalent to perhaps the most common relaxation of D-optimal design, as proposed by \citet[Eqn. 7.2.6]{boyd2004convex}. By formulating in terms of the quantum relative entropy cone, the identity term suggests this relaxation leaves something ``on the table''.

In this direction, let us apply the MPRT. Observe that $\bm{X}:=\sum_{i \in [n]} z_i \bm{a}_i \bm{a}_i^\top$ is a rank-$k$ matrix and thus at an optimal solution to the original problem there is some orthogonal projection matrix $\bm{Y}$ such that $\bm{X}=\bm{Y}\bm{X}$. Therefore, we can take the perspective function of $f(\bm{X})=\log(\bm{X}{\color{black}+\epsilon\mathbb{I}})$, and thereby obtain the following valid—and potentially much tighter when $k<n$—convex relaxation:
\begin{align}
    \max_{\bm{z} \in [0, 1]^n, \bm{\theta}, \bm{Y} \in \mathcal{S}^n_+}\quad &  \mathrm{tr}(\bm{\theta})+(n-\mathrm{tr}(\bm{Y}))\log(\epsilon) \label{eqn:quantentrdopt}\\ 
     \text{s.t.} \quad & \bm{Y}^\frac{1}{2}\log\left({\blue \bm{Y}^{-\frac{1}{2}}\bm{A}\mathrm{Diag}(\bm{z})\bm{A}^\top \bm{Y}^{-\frac{1}{2}} +\epsilon\mathbb{I}} \right)\bm{Y}^\frac{1}{2}\succeq \bm{\theta}, \bm{Y} \preceq \mathbb{I}, \mathrm{tr}(\bm{Y}) \leq k,\nonumber
\end{align}
which can be modeled via the quantum relative entropy cone: 
{\blue $(-\bm{\theta}, \bm{Y}, \bm{A}\mathrm{Diag}(\bm{z})\bm{A}^\top+\epsilon\bm{Y}) \in \mathcal{K}^{\text{rel, op}}_{\text{mat}}$}. We now argue that this relaxation is high-quality, by demonstrating that the MPRT supplies the convex envelope of $t \geq -\log\det_\epsilon(\bm{X})$ under a low-rank constraint, via the following corollary to Theorem \ref{thm:lowrankset}:
\begin{corollary}\label{lemma:initlemma2}
\begin{align*}
\text{Let} \quad \mathcal{S}=\left\{\bm{X} \in \mathcal{S}^n_+ : t \geq -\log\det_\epsilon(\bm{X}), \mathrm{Rank}(\bm{X}) \leq k \right\}
\end{align*}
be a set where $\epsilon, k,t$ are fixed. Then, this set's convex hull is:
\begin{align*}
\mathcal{S}^c=\bigg\{(\bm{Y}, \bm{X}) \in \mathcal{S}^n_+ \times \mathcal{S}^n_+: & \bm{0} \preceq \bm{Y} \preceq \mathbb{I}, \mathrm{tr}(\bm{Y}) \leq k, \\
& t \geq - \mathrm{tr}(\bm{Y}^\frac{1}{2}\log_{\epsilon}(\bm{Y}^{-\frac{1}{2}}\bm{X}\bm{Y}^{-\frac{1}{2}}) \bm{Y}^\frac{1}{2})-(n-\mathrm{tr}(\bm{Y}))\log(\epsilon)\bigg\}.\nonumber
\end{align*}
\end{corollary}

\begin{remark}
Observe that \eqref{eqn:quantentrdopt}'s relaxation is not useful in the over-determined regime where $k \geq n$, since setting $\bm{Y}=\mathbb{I}$ recovers \eqref{eqn:doptdesign1}'s Boolean relaxation, which is considerably cheaper to optimize over. Accordingly, we only consider the under-determined regime in our experiments.
\end{remark}

\subsection{Non-Negative Matrix Optimization}\label{ssec:NNMF}
Many important problems in combinatorial optimization, statistics and computer vision \citep[see, e.g.,][]{burer2009copositive} reduce to optimizing over the space of low-rank matrices with non-negative factors. An important special case is when we would like to find the low-rank {\color{black}completely positive} matrix $\bm{X}$ which best approximates (in a least-squares sense) a given matrix $\bm{A} \in \mathcal{S}^n_+$, i.e., perform non-negative principal component analysis. 
Formally, we have the problem:
\begin{align}\label{eqn:completelypospca}
        \min_{\bm{X} \in \mathcal{C}^n_+: \mathrm{Rank}(\bm{X}) \leq k} \quad & \Vert \bm{X}-\bm{A}\Vert_F^2,
\end{align}
{\color{black}where $\mathcal{C}^n_+:=\{\bm{U}\bm{U}^\top: \bm{U} \in \mathbb{R}^{n \times n}_{+}\}$ denotes the cone of $n \times n$ completely positive matrices.}

Applying the MPRT to the strongly convex $\frac{1}{2}\Vert \bm{X}\Vert_F^2$ term in the objective therefore yields the following completely positive program:
\begin{align}\label{eqn:dnn}
        \min_{\bm{X} \in \mathcal{C}^n_+, \bm{Y}, \bm{\theta} \in \mathcal{S}^n} \quad & \frac{1}{2}\mathrm{tr}(\bm{\theta})-\langle \bm{X}, \bm{A}\rangle+\frac{1}{2}\Vert \bm{A}\Vert_F^2\ \text{s.t.}\ \bm{Y} \preceq \mathbb{I}, \ \mathrm{tr}(\bm{Y})\leq k,\ \begin{pmatrix} \bm{Y} & \bm{X}\\ \bm{X}^\top & \bm{\theta}\end{pmatrix} \in S^{2n}_+.
\end{align}
Interestingly, since \eqref{eqn:dnn}'s reformulation has a linear objective, some extreme point in its relaxation is optimal, which means we can relax the requirement that $\bm{Y}$ is a projection matrix without loss of optimality and the computational complexity of the problem is entirely concentrated in the completely positive cone. Unfortunately however, completely positive optimization itself is intractable. Nonetheless, it can be approximated by replacing the completely positive cone with the doubly non-negative cone, $\mathcal{S}^n_+ \cap \mathbb{R}^{n \times n}_{+}$. Namely, we instead solve
\begin{align}\label{eqn:dnn_relax}
        \min_{\bm{X} \in \mathcal{S}^n_+ \cap \mathbb{R}^{n \times n}_{+}, \bm{Y}, \bm{\theta}  \in \mathcal{S}^n} \ & \frac{1}{2}\mathrm{tr}(\bm{\theta})-\langle \bm{X}, \bm{A}\rangle+\frac{1}{2}\Vert \bm{A}\Vert_F^2\ \text{s.t.} \ \begin{pmatrix} \bm{Y} & \bm{X}\\ \bm{X}^\top & \bm{\theta}\end{pmatrix} \in S^{2n}_+,\ \bm{Y} \preceq \mathbb{I}, \ \mathrm{tr}(\bm{Y}) \leq k.
\end{align}

Unfortunately, rounding a solution to \eqref{eqn:dnn_relax} to obtain a completely positive $\bm{X}$ is non-trivial. Indeed, according to \citet{ge2010doubly}, there is currently no effective mechanism for rounding doubly non-negative programs. Nonetheless, as we shall see in our numerical results, there are already highly effective heuristic methods for completely positive matrix factorization, and combining our relaxation with such a procedure offers certificates of near optimality in a tractable fashion.

\begin{remark}
If $\bm{X}=\bm{D}\bm{\Pi}$ is a monomial matrix, i.e., decomposable as the product of a diagonal matrix $\bm{D}$ and a permutation matrix $\bm{\Pi}$, as occurs in binary optimization problems such as $k$-means clustering problems among others \citep[c.f.][]{peng2007approximating}, then it follows that $(\bm{X}^\top \bm{X})^\dag \geq \bm{0}$ \citep[see][]{plemmons1972generalized} and thus $\bm{Y}:=\bm{X}(\bm{X}^\top \bm{X})^\dag\bm{X}^\top$ is elementwise non-negative. 
In this case, the doubly non-negative relaxation \eqref{eqn:dnn_relax} should be strengthened by requiring that $\bm{Y} \geq \bm{0}$. 
\end{remark}

\section{Numerical Results}\label{sec:numres}
In this section, we evaluate the algorithmic strategies derived in the previous section, implemented in \verb|Julia| 1.5 using \verb|JuMP.jl| $0.21.6$ and \verb|Mosek| $9.1$ to solve the conic problems considered here. Except where indicated otherwise, all experiments were performed on a Intel Xeon E5---2690 v4 2.6GHz CPU core using 32 GB RAM. To bridge the gap between theory and practice, we have made our code freely available on \verb|Github| at \verb|github.com/ryancorywright/MatrixPerspectiveSoftware|.

\subsection{Reduced Rank Regression}
In this section, we compare our convex relaxations for reduced rank regression developed in the introduction and laid out in \eqref{eqn:rrr_persp}-\eqref{eqn:rrr_dcl}—which we refer to as ``Persp'' and ``DCL'' respectively—against the nuclear norm estimator proposed by \cite{negahban2011estimation} (``NN''), who solve
\begin{align}\label{eqn:rrr_nn}
    \min_{\bm{\beta} \in \mathbb{R}^{p \times n}} \quad \frac{1}{2m}\Vert \bm{Y}-\bm{X}\bm{\beta}\Vert_F^2+\frac{1}{2\gamma}\Vert \bm{\beta}\Vert_F^2+\mu \Vert\bm{\beta}\Vert_*.
\end{align}
 
Similarly to \cite{negahban2011estimation}, we attempt to recover rank$-k_{true}$ estimators $\bm{\beta}_{\text{true}}=\bm{U}\bm{V}^\top$, where each entry of $\bm{U} \in \mathbb{R}^{p \times k_{true}}, \bm{V} \in \mathbb{R}^{n \times k_{true}}$ is i.i.d. standard Gaussian $\mathcal{N}(0,1)$, the matrix $\bm{X} \in \mathbb{R}^{m \times p}$ contains i.i.d. standard Gaussian $\mathcal{N}(0,1)$ entries, $\bm{Y}=\bm{X}\bm{\beta}+\bm{E}$, and $E_{i,j} \sim \mathcal{N}(0, \sigma)$ injects a small amount of i.i.d. noise. We set $n=p=50, k=10$, $\gamma =10^6$, $\sigma=0.05$ and vary $m$. To ensure a fair comparison, we cross-validate $\mu$ for both of our relaxations and \cite{negahban2011estimation}'s approach so as to minimize the MSE on a validation set. For each $m$, we evaluate $20$ different values of $\mu$ which are distributed uniformly in logspace between $10^{-4}$ and $10^{4}$ across $50$ random instances for our convex relaxations and report on $100$ different random instances with the ``best'' $\mu$ for each method and each $p$. 
 
\paragraph{Rank recovery and statistical accuracy:} Figures \ref{sfig:redrank.acc}-\ref{sfig:redrank.mse} report 
the relative accuracy ($\Vert \bm{\beta}_{\text{est}}-\bm{\beta}_{\text{true}}\Vert_F/\Vert \bm{\beta}_{\text{true}}\Vert_F$), the rank (i.e., number of singular values of $\bm{\beta}_{\text{est}}$ which exceed $10^{-4}$), and the out-of-sample MSE\footnote{Evaluated on $m=1000$ new observations of $\bm{X}_j, \bm{Y}_k$ generated from the same distribution.} $\Vert \bm{X}_{\text{new}}\bm{\beta}_{\text{est}}-\bm{y}_{\text{new}}\Vert_F^2$ (normalized by the out-of-sample MSE of the ground truth $\Vert \bm{X}_{\text{new}}\bm{\beta}_{\text{true}}-\bm{y}_{\text{new}}\Vert_F^2$). Results are averaged over $100$ random instances per value of $m$. We observe that—even though we did not supply the true rank of the optimal solution in our formulation—Problem \eqref{eqn:rrr_dcl}'s relaxation returns solutions of the correct rank ($k_{true}=10$) and better MSE/accuracy, while our more ``naive'' perspective relaxation \eqref{eqn:rrr_persp} and the nuclear norm approach \eqref{eqn:rrr_nn} return solutions of a higher rank and lower accuracy. This suggests that \eqref{eqn:rrr_dcl}'s formulation should be considered as a more accurate estimator for reduced rank problems, and empirically confirms that the MPRT can lead to significant improvements in statistical accuracy.

\paragraph{Scalability w.r.t. $m$:} 
Figure \ref{sfig:redrank.time} reports the average time for \verb|Mosek| to converge\footnote{We model the convex quadratic $\Vert \bm{X}\bm{\beta}-\bm{Y}\Vert_F^2$ using a rotated second order cone for formulations \eqref{eqn:rrr_persp} and \eqref{eqn:rrr_nn} (the quadratic term doesn't appear directly in \eqref{eqn:rrr_dcl}), model the nuclear norm term in \eqref{eqn:rrr_nn} by introducing matrices $\bm{U}, \bm{V}$ such that $\begin{pmatrix} \bm{U} & \bm{\beta}\\\bm{\beta}^\top & \bm{V}\end{pmatrix} \succeq \bm{0}$ and minimizing $\mathrm{tr}(\bm{U})+\mathrm{tr}(\bm{V})$, use default Mosek parameters for all approaches.} to an optimal solution (over $100$ random instances per $m$). Surprisingly, although \eqref{eqn:rrr_dcl} is a stronger relaxation than \eqref{eqn:rrr_persp}, it is one to two orders of magnitude faster than \eqref{eqn:rrr_persp} and \eqref{eqn:rrr_nn}'s formulations. The relative scalability of \eqref{eqn:rrr_dcl}'s formulation as $m$—the number of observation— increases can be explained by the fact that \eqref{eqn:rrr_dcl} considers a linear inner product of the Gram matrix $\bm{X}^\top \bm{X}$ with a semidefinite matrix $\bm{B}$ (the size of which does not vary with $m$) while Problems \eqref{eqn:rrr_persp} and \eqref{eqn:rrr_nn} have a quadratic inner product $\langle \bm{\beta}\bm{\beta}^\top, \bm{X}^\top \bm{X}\rangle$ which must be modeled using a rotated second-order cone constraint (the size of which depends on $m$), since modern conic solvers such as \verb|Mosek| do not allow quadratic objective terms and semidefinite constraints to be simultaneously present (if they did, we believe all three formulations would scale similarly). 

\begin{figure}[h!]\centering
        \begin{subfigure}[t]{.45\linewidth}
            \includegraphics[width=\textwidth]{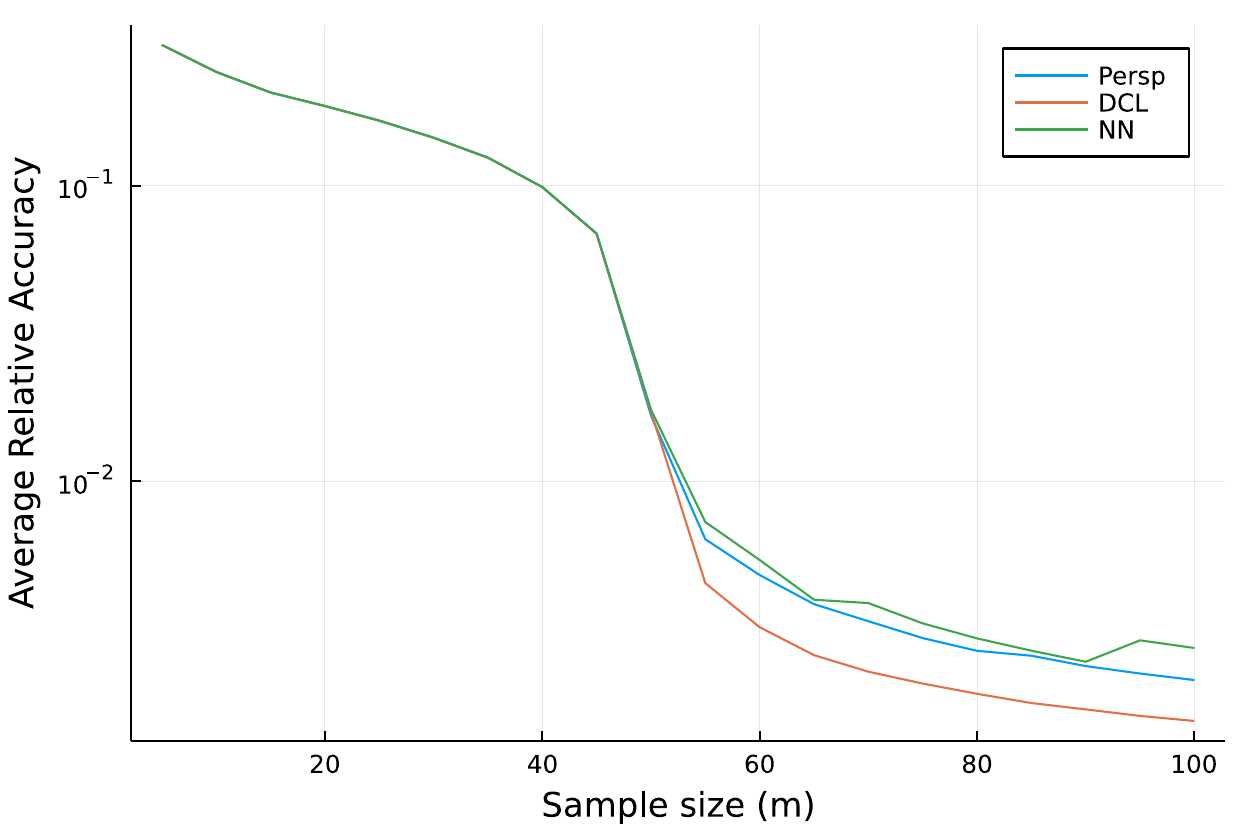}
            \subcaption{Accuracy} \label{sfig:redrank.acc}
    \end{subfigure}
    \begin{subfigure}[t]{.45\linewidth}
    \centering
            \includegraphics[width=\textwidth]{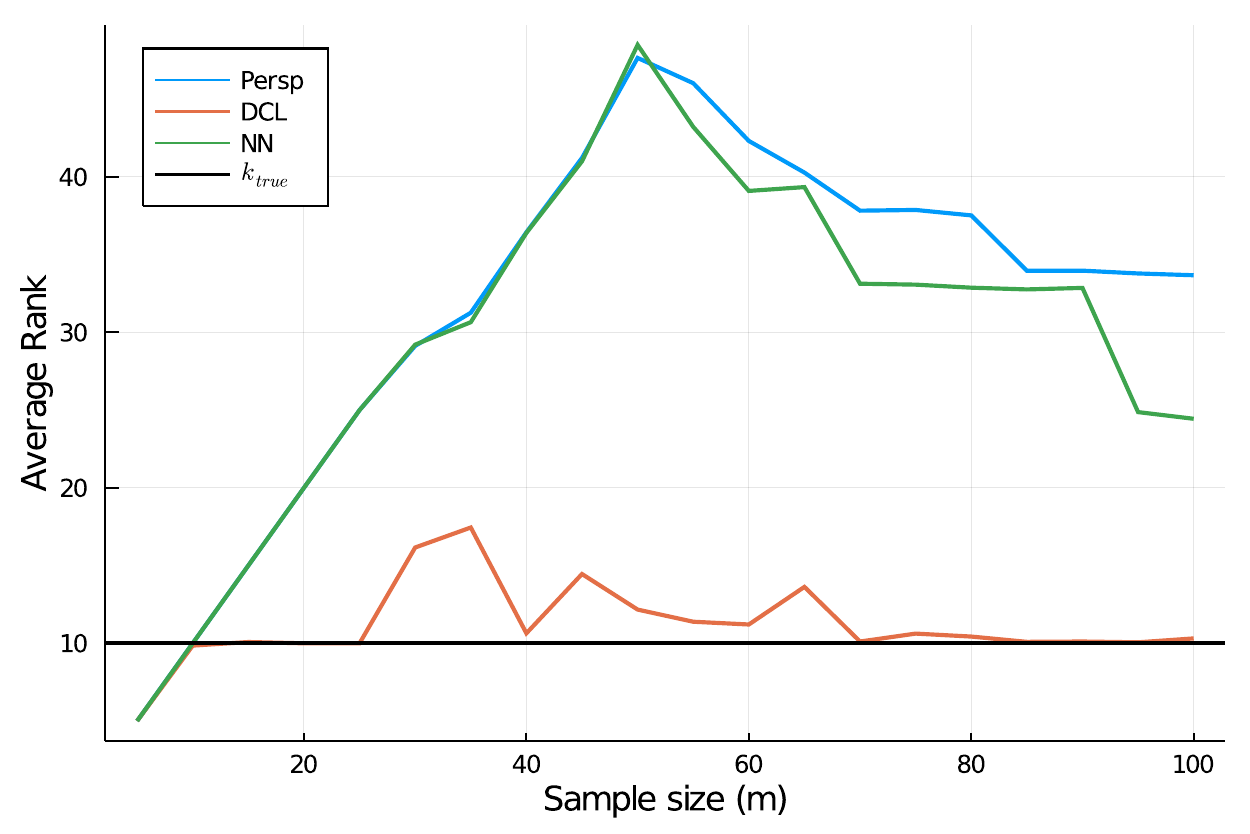}
            \subcaption{Rank} \label{sfig:redrank.rank}
    \end{subfigure} \\
    \begin{subfigure}[t]{.45\linewidth}
            \includegraphics[width=\textwidth]{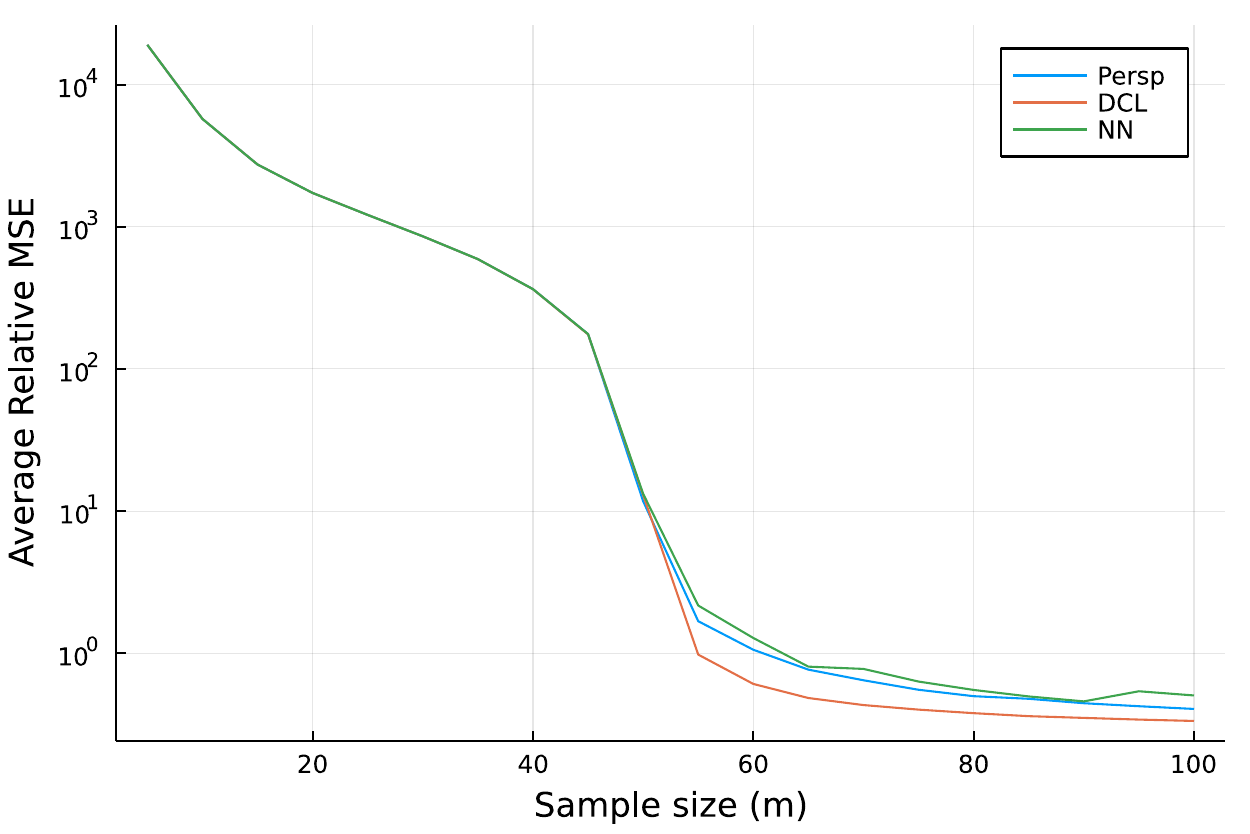}
            \subcaption{Relative MSE} \label{sfig:redrank.mse}
    \end{subfigure}
    \begin{subfigure}[t]{.45\linewidth}
    \centering
            \includegraphics[width=\textwidth]{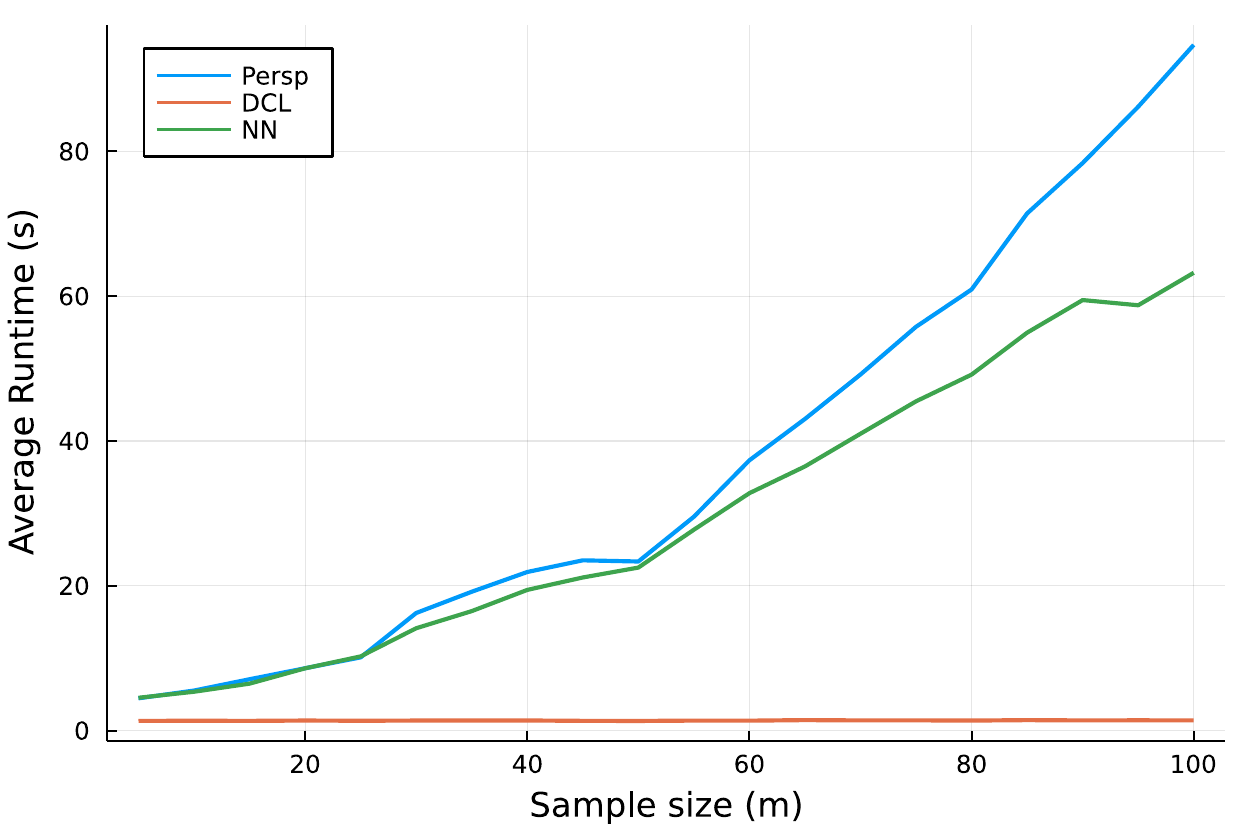}
            \subcaption{Runtime} \label{sfig:redrank.time}
    \end{subfigure}
   \caption{Comparative performance, as the number of samples $m$ increases, of formulations \eqref{eqn:rrr_persp} (Persp, in blue), \eqref{eqn:rrr_dcl} (DCL, in orange) and \eqref{eqn:rrr_nn} (NN, in green), averaged over $100$ synthetic reduced rank regression instances where $n=p=50$, $k_{true}=10$. The hyperparameter $\mu$ was first cross-validated for all approaches separately.}
   \label{fig:sensitivitytom}
\end{figure}

\paragraph{Scalability w.r.t $p$:} Next, we evaluate the scalability of all three approaches in terms of their solve times and peak memory usage (measured using the \verb|slurm| command \verb|MaxRSS|), as $n=p$ increases. Fig. \ref{fig:sensitivityton} depicts the average time to converge to an optimal solution (a) and peak memory consumption (b) by each method as we vary $n=p$ with $m=n$, $k=10$, $\gamma=10^6$, each $\mu$ fixed to the average cross-validated value found in the previous experiment, a peak memory budget of $120$GB, a runtime budget of $12$ hours, and otherwise the same experimental setup as previously (averaged over $20$ random instances per $n$). We observe \eqref{eqn:rrr_dcl}'s relaxation is dramatically more scalable than the other two approaches considered, and can solve problems of nearly twice the size ($4$ times as many variables), and solves problems of a similar size in substantially less time and with substantially less peak memory consumption ($40$s vs. $1000$s when $n=100$). All in all, the proposed relaxation \eqref{eqn:rrr_dcl} seems to be the best method of the three considered.

\begin{figure}[h!]\centering
    \begin{subfigure}[t]{.45\linewidth}
    \centering
           \includegraphics[width=\textwidth]{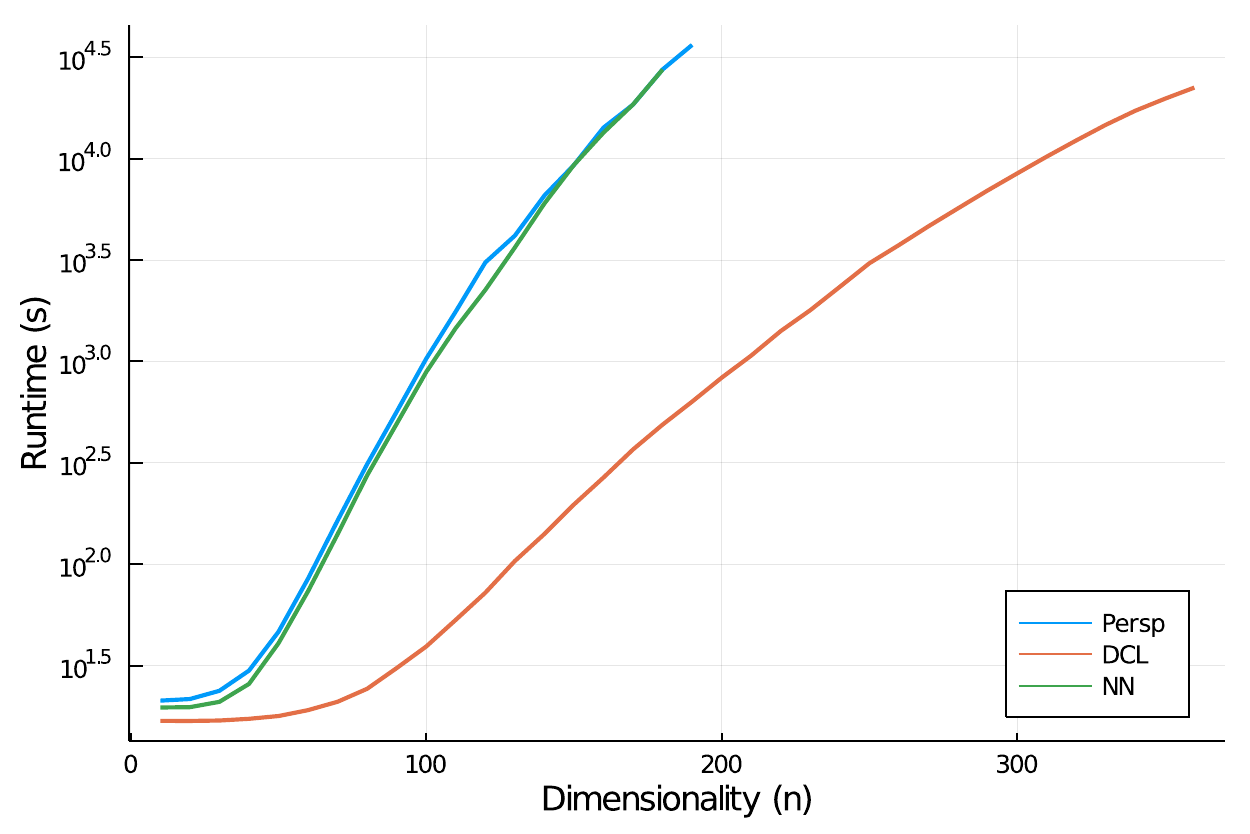} \subcaption{Runtime} 
    \end{subfigure}
        \begin{subfigure}[t]{.45\linewidth}
           \includegraphics[width=\textwidth]{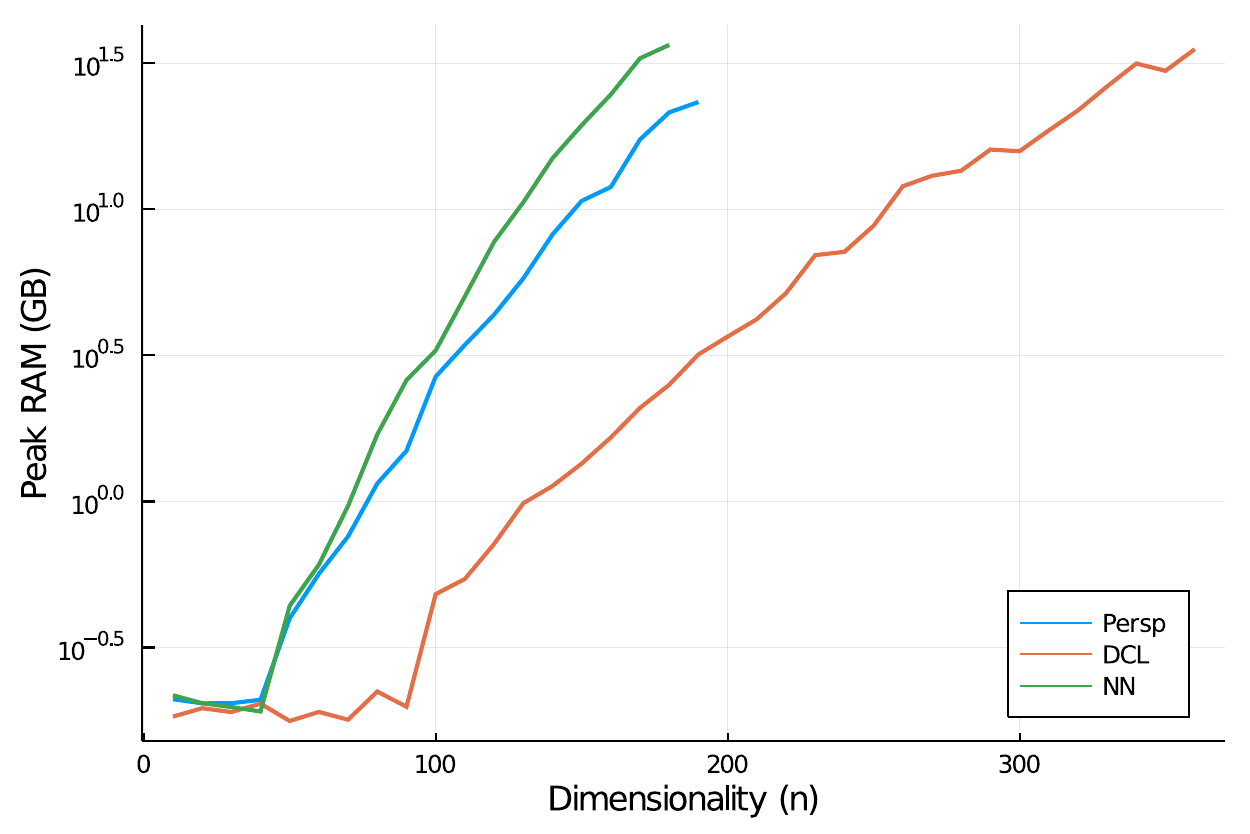} \subcaption{Peak Memory} 
    \end{subfigure}
   \caption{Average time to compute an optimal solution (left panel) and peak memory usage (right panel) vs. dimensionality $n=p$ for Problems \eqref{eqn:rrr_persp} (Persp, in blue), \eqref{eqn:rrr_dcl} (DCL. in orange) and \eqref{eqn:rrr_nn} (NN, in green) over $20$ synthetic reduced rank regression instances where $k_{true}=10$. }
   \label{fig:sensitivityton}
\end{figure}

\subsection{Non-Negative Matrix Factorization}
In this section, we benchmark the quality of our dual bound for non-negative matrix factorization laid out in Section \ref{ssec:NNMF} by using the non-linear reformulation strategy proposed by \cite{burer2003nonlinear} (alternating least squares or ALS) to obtain upper bounds. Namely, we obtain upper bounds by solving for local minima of the problem
\begin{align}\label{eqn:completelypospca2}
        \min_{\bm{U} \in \mathbb{R}^{n \times k}_{+}} \quad & \Vert \bm{U}\bm{U}^\top-\bm{A}\Vert_F^2.
\end{align}
In our implementation of ALS, we obtain a local minimum by introducing a dummy variable $\bm{V}$ which equals $\bm{U}$ at optimality and alternating between solving the following two problems
\begin{align}\label{eqn:completelypospca1}
        \bm{U}_{t+1}=\arg\min_{\bm{U} \in \mathbb{R}^{n \times k}_{+}} \quad & \Vert \bm{U}\bm{V}_t^\top-\bm{A}\Vert_F^2+\rho_t \Vert \bm{U}-\bm{V}_t\Vert_F^2,\\
        \bm{V}_{t+1}=\arg\min_{\bm{V} \in \mathbb{R}^{n \times k}_{+}} \quad & \Vert \bm{U}_t\bm{V}^\top-\bm{A}\Vert_F^2+\rho_t \Vert \bm{U}_t-\bm{V}\Vert_F^2,
\end{align}
where we set $\rho_t=\min(10^{-4}\times 2^{t-1}, 10^5)$ at the $t$th iteration in order that the final matrix is positive semidefinite, as advocated in \citep[Section 5.2.3]{bertsekas1999nonlinear} (we cap $\rho_t$ to avoid numerical instability). We iterate over solving these two problems from a random initialization point $\bm{V}_{0}$—where each $V_{0,i,j}$ is i.i.d. standard uniform—until either the objective value between iterations does not change by $10^{-4}$ or we exceed the maximum number of allowable iterations, which we set to $100$.

To generate problem instances, 
we let $\bm{A}=\bm{U}\bm{U}^\top+\bm{E}$ where $\bm{U} \in \mathbb{R}^{n \times k_{true}}$, each $U_{i,j}$ is uniform on $[0,1]$, $E_{i,j} \sim \mathcal{N}(0, 0.0125 k_{true})$, and set $A_{i,j}=0$ if $A_{i,j}<0$. We set $n=50, k_{true}=10$. We use the ALS heuristic to compute a feasible solution $\bm{X}$ and an upper-bound on the problem's objective value. By comparing it with the lower bound derived from our MPRT, we can assess the sub-optimality of the heuristic solution, which previously lacked optimality guarantees. 

Figure \ref{fig:sensitivitytok} depicts the average relative in-sample MSE of the heuristic ($\Vert \bm{X}-\bm{A}\Vert_F/\Vert \bm{A}\Vert_F$) 
and the relative bound gap—(UB-LB)/UB— as we vary the target rank, averaged over $100$ random synthetic instances. 
We observe that the method is most accurate and has the lowest MSE when $k$ is set to $k_{true}=10$, which confirms that the method can recover solutions of the correct rank. In addition, by combining the solution from OLS with our lower-bound, we can compute a duality gap and assert that the heuristic solution is $0\%-3\%$-optimal, with the gap peaking at $k = k_{true}$ and stabilizing as $k \rightarrow n$. This echoes similar findings in $k$-means clustering and alternating current optimal power flow problems, where the SDO relaxation need not be near-tight in theory but nonetheless is nearly exact in practice \cite{peng2007approximating, lavaei2011zero}. Further, this suggests our convex relaxation may be a powerful weapon for providing gaps for heuristics for non-negative matrix factorization, and particularly detecting when they are performing well or can be further improved.

\begin{figure}[h!]\centering
    \begin{subfigure}[t]{.45\linewidth}
\centering
\includegraphics[width=\textwidth]{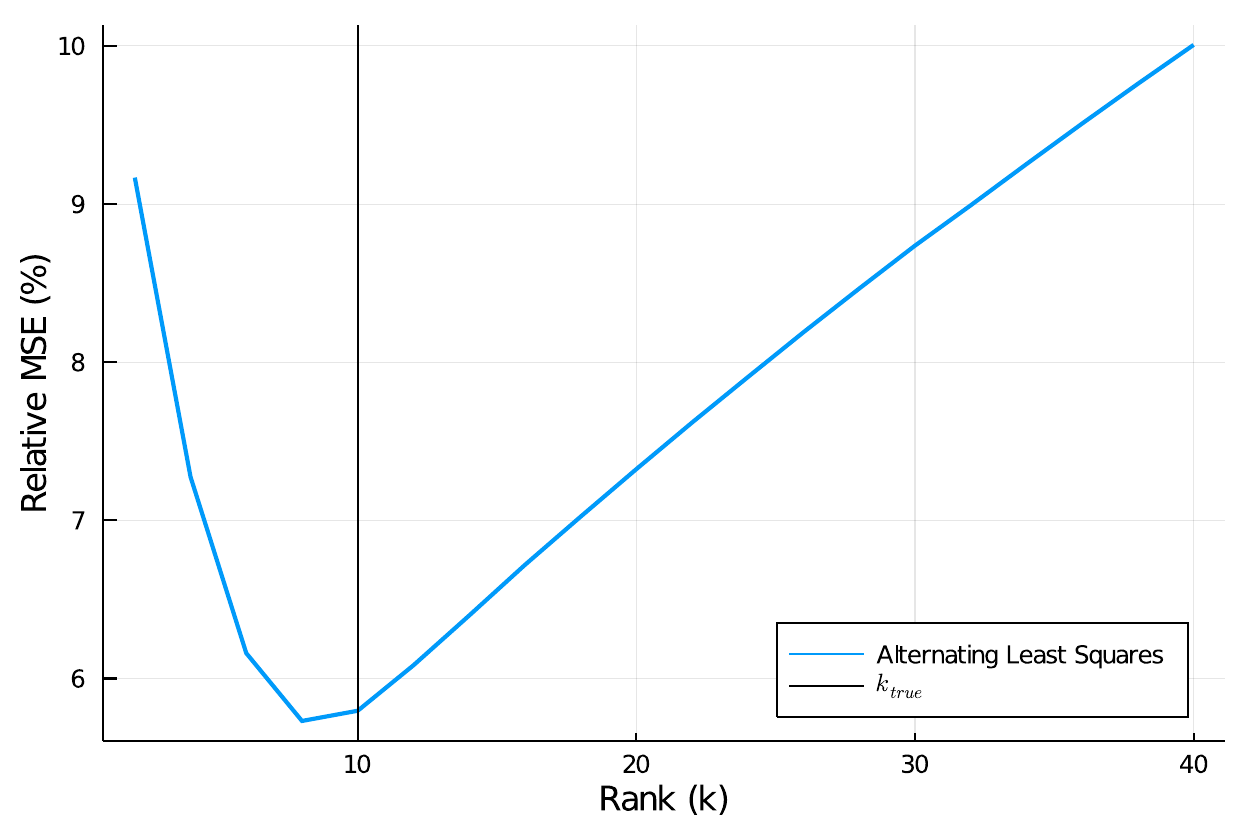}
\subcaption{Relative MSE}
\end{subfigure}
\begin{subfigure}[t]{.45\linewidth}
\centering
\includegraphics[width=\textwidth]{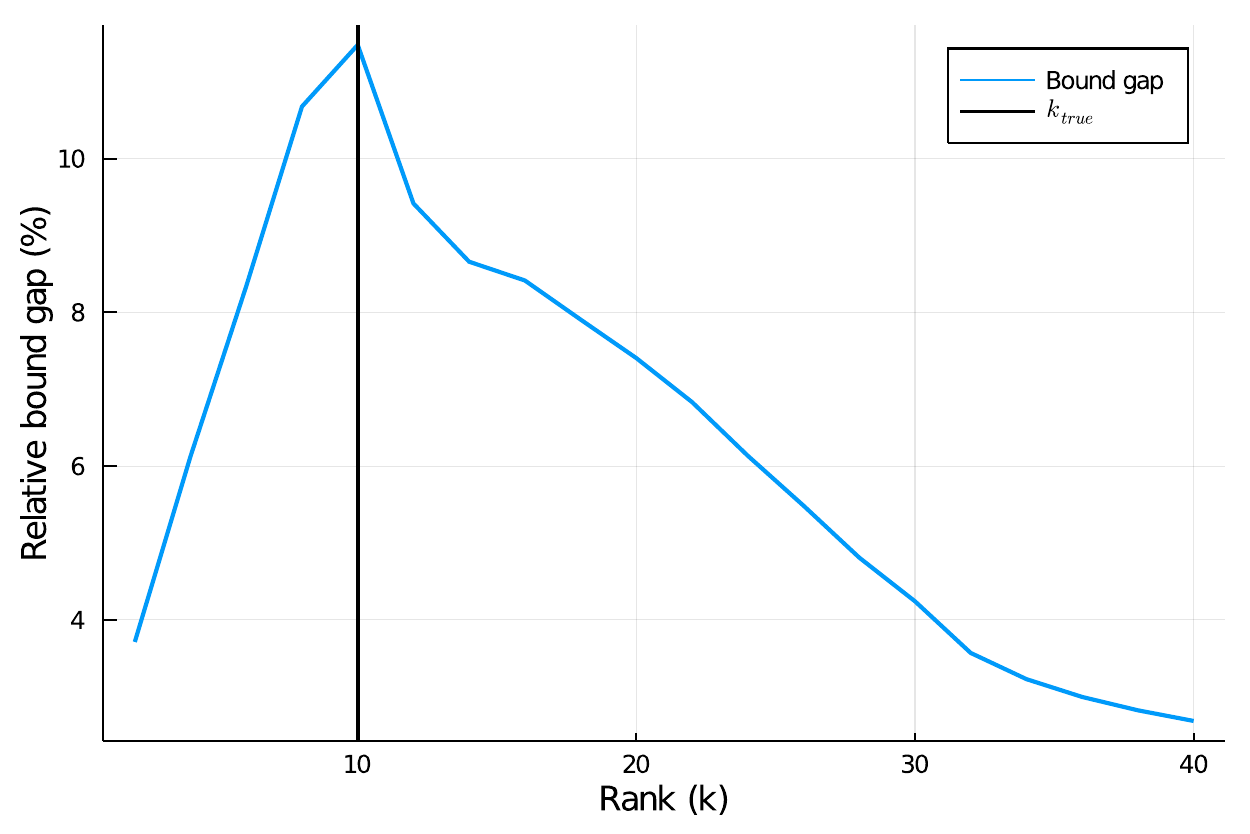}
\subcaption{Bound gap}
\end{subfigure}
\caption{Average relative MSE and duality gap vs. target rank $k$ using the ALS heuristic (UB) and the MPRT relaxation (LB). Results are averaged  over $100$ synthetic completely positive matrix factorization instances where $n=50$, $k_{true}=10$.}
\label{fig:sensitivitytok}
\end{figure}

Figure \ref{fig:nnmf.time.tok} reports the time needed to compute both the upper bound and a lower bound solution as we vary the target rank. 
\begin{figure}[h!]\centering
\includegraphics[width=.5\textwidth]{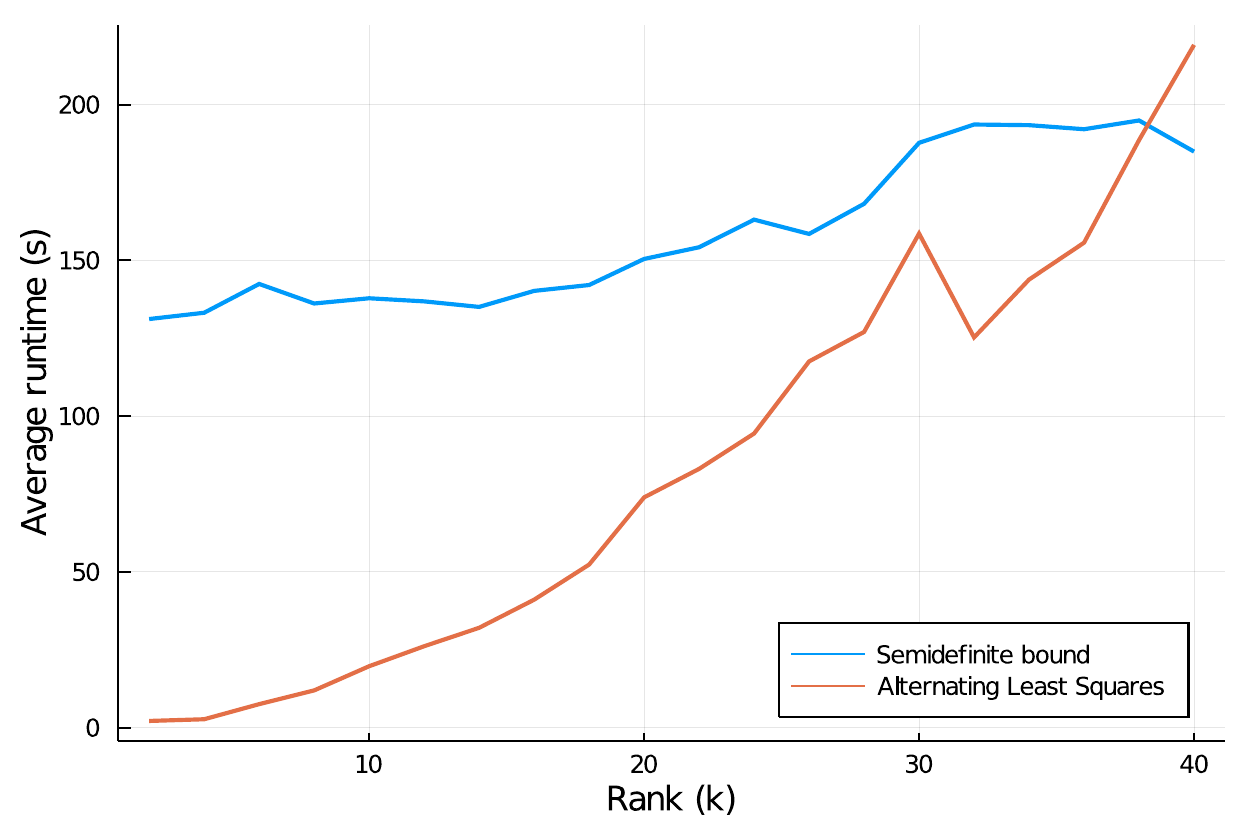}
   \caption{Computational time to compute a feasible solution (ALS) and solve the relaxation (Semidefinite bound) vs. target rank $k$, averaged over $100$ synthetic completely positive matrix factorization instances where $n=50$, $k_{true}=10$.}
   \label{fig:nnmf.time.tok}
\end{figure}

\subsection{Optimal Experimental Design}
In this section, we benchmark our dual bound for D-optimal experimental design \eqref{eqn:quantentrdopt} against the convex relaxation \eqref{eqn:doptdesign1} and a greedy submodular maximization approach, in terms of both bound quality and the ability of all three approaches to generate high-quality feasible solutions. We round both relaxations to generate feasible solutions greedily, by setting the $k$ largest $z_i$'s in a continuous relaxation to $1$, while for the submodular maximization approach we iteratively set the $j$th index of $\bm{z}$ to $1$, where $\mathcal{S}$ is initially an empty set and we iteratively take
\begin{align*}
\mathcal{S} \gets \mathcal{S} \cup \{j\}: j \in \arg\max_{i \in [n]\backslash \mathcal{S}}\left\{\log\det_\epsilon \left(\sum_{l \in \mathcal{S}}z_l \bm{a}_l\bm{a}_l^\top+\bm{a}_i\bm{a}_i^\top \right)\right\}.
\end{align*}
Interestingly, the greedy rounding approach enjoys rigorous approximation guarantees \citep[see][]{joshi2008sensor,singh2018approximation}, while the submodular maximization approach also enjoys strong guarantees \citep[see][]{nemhauser1978analysis}.

We benchmark all methods in terms of their performance on synthetic $D$-optimal experimental design problems, where we let $\bm{A} \in \mathbb{R}^{n \times m}$ be a matrix with i.i.d. $\mathcal{N}(0, \frac{1}{\sqrt{n}})$ entries. We set $n=20, m=10, \epsilon=10^{-6}$ and vary $k <m$ over $20$ random instances. Table \ref{tab:comparison} depicts the average relative bound gap, objective values, and runtimes for all $3$ methods (we use the lower bound from \eqref{eqn:doptdesign1}'s relaxation to compute the submodular bound gap). Note that all results for this experiment were generated on a standard Macbook pro laptop with a $2.9$GHZ $6$-core Intel i$9$ CPU using $16$GB DDR4 RAM, CVX version $1.22$, Matlab R$2021$a, and \verb|Mosek| $9.1$. Moreover, we optimize over \eqref{eqn:quantentrdopt}'s relaxation using the \verb|CVXQuad| package developed by \cite{fawzi2019semidefinite}. 

\begin{table}[h]
\centering\footnotesize
\caption{{Average runtime in seconds and relative bound gap per approach, over $20$ random instances where $n=10, m=20$.}} 
\begin{tabular}{@{}l r r r r r r@{}} \toprule
 & \multicolumn{2}{c@{\hspace{0mm}}}{Problem \eqref{eqn:doptdesign1}+round} &  \multicolumn{2}{c@{\hspace{0mm}}}{Submodular} &  \multicolumn{2}{c@{\hspace{0mm}}}{Problem \eqref{eqn:quantentrdopt}+round} \\
\cmidrule(l){2-3} \cmidrule(l){4-5} \cmidrule(l){6-7} $k$ & Time(s) & Gap ($\%$) & Time(s) & Gap ($\%$) & Time(s) & Gap ($\%$)\\\midrule
$1$ & $0.52$& $88.8$& $0.00$& $88.9$& $347.0$& $0.00$ \\
$2$ & $0.63$& $93.7$& $0.00$& $93.7$& $338.5$& $0.01$ \\
$3$ & $0.59$& $97.1$& $0.00$& $97.0$& $320.8$& $0.06$ \\
$4$ & $0.63$& $100.2$& $0.00$& $100.2$& $338.7$& $0.18$ \\
$5$ & $0.53$& $103.8$& $0.00$& $103.9$& $331.1$& $0.37$ \\
$6$ & $0.53$& $109.0$& $0.00$& $109.0$& $287.5$& $1.40$ \\
$7$ & $0.55$& $117.7$& $0.00$& $117.7$& $255.1$& $2.39$ \\
$8$ & $0.60$& $136.9$& $0.00$& $138.5$& $236.1$& $5.25$ \\
$9$ & $0.54$& $260.9$& $0.00$& $287.5$& $235.9$& $28.43$ \\
\bottomrule
\end{tabular}
\label{tab:comparison}
\end{table}

\paragraph{Relaxation quality:} We observe that \eqref{eqn:quantentrdopt}'s relaxation is dramatically stronger than \eqref{eqn:doptdesign1}, offering bound gaps on the order of $0\%-3\%$ when $k \leq 7$, rather than gaps of $90\%$ or more. This confirms the efficacy of the MPRT, and demonstrates the value of taking low-rank constraints into account when designing convex relaxations, even when not obviously present.

\paragraph{Scalability:} We observe that \eqref{eqn:quantentrdopt}'s relaxation is around two orders of magnitude slower than the other proposed approaches, largely because semidefinite approximations of quantum relative entropy are expensive, but is still tractable for moderate sizes. We believe, however, that the relaxation would scale significantly better if it were optimized over using an interior point method for non-symmetric cones \citep[see, e.g.,][]{skajaa2015homogeneous, karimi2019domain}, or an alternating minimization approach \citep[see][]{faybusovich2020self}. As such, \eqref{eqn:quantentrdopt}'s relaxation is potentially useful at moderate problem sizes with off-the-shelf software, or at larger problem sizes with problem-specific techniques such as alternating minimization.

\section{Conclusion}
In this paper, we introduced the Matrix Perspective Reformulation Technique (MPRT), a new technique for deriving tractable and often high-quality relaxations of a wide variety of low-rank problems. We also invoked the technique to derive the convex hulls of some frequently-studied low-rank sets, and provided examples where the technique proves useful in practice. This is significant and potentially useful to the community, because substantial progress on producing tractable upper bounds for low-rank problems has been made over the past decade, but until now almost no progress on tractable lower bounds has followed.

Future work could take three directions: (1) automatically detecting structures where the MPRT could be applied, as is already done for perspective reformulations in the MIO case by \verb|CPLEX| and \verb|Gurobi|, (2) developing scalable semidefinite-free techniques for solving the semidefinite relaxations proposed in this paper, and (3) combining the ideas in this paper and in our prior work \cite{bertsimas2020mixed} with custom branching strategies to solve low-rank problems to optimality at scale.


{\color{black}
\paragraph{Acknowledgments:}
We are very grateful to two anonymous referees for useful and constructive comments. In particular, we would like to thank reviewer $\#1$ for a very helpful refinement of our definition of the matrix perspective function, and reviewer $\#2$ for suggesting the name \textit{matrix perspective function} and supplying some new references on perspective operator functions. 
}
{
\bibliographystyle{abbrvnat} 

\begin{thebibliography}{77}
\providecommand{\natexlab}[1]{#1}
\providecommand{\url}[1]{\texttt{#1}}
\expandafter\ifx\csname urlstyle\endcsname\relax
  \providecommand{\doi}[1]{doi: #1}\else
  \providecommand{\doi}{doi: \begingroup \urlstyle{rm}\Url}\fi

\bibitem[Akt{\"u}rk et~al.(2009)Akt{\"u}rk, Atamt{\"u}rk, and
  G{\"u}rel]{akturk2009strong}
M.~S. Akt{\"u}rk, A.~Atamt{\"u}rk, and S.~G{\"u}rel.
\newblock A strong conic quadratic reformulation for machine-job assignment
  with controllable processing times.
\newblock \emph{Operations Research Letters}, 37\penalty0 (3):\penalty0
  187--191, 2009.

\bibitem[Alizadeh(1995)]{alizadeh1995interior}
F.~Alizadeh.
\newblock Interior point methods in semidefinite programming with applications
  to combinatorial optimization.
\newblock \emph{{SIAM} Journal on Optimization}, 5\penalty0 (1):\penalty0
  13--51, 1995.

\bibitem[Atamt{\"u}rk and Gomez(2019)]{atamturk2019rank}
A.~Atamt{\"u}rk and A.~Gomez.
\newblock Rank-one convexification for sparse regression.
\newblock \emph{arXiv:1901.10334}, 2019.

\bibitem[Ben-Tal and Nemirovski(2001)]{ben2001lectures}
A.~Ben-Tal and A.~Nemirovski.
\newblock \emph{Lectures on modern convex optimization: Analysis, algorithms,
  and engineering applications}, volume~2.
\newblock {SIAM} Philadelphia, PA, 2001.

\bibitem[Bertsekas(2016)]{bertsekas1999nonlinear}
D.~P. Bertsekas.
\newblock \emph{Nonlinear programming}.
\newblock Athena Scientific Belmont MA, 3rd edition, 2016.

\bibitem[Bertsimas and Van~Parys(2020)]{bertsimas2017sparse}
D.~Bertsimas and B.~Van~Parys.
\newblock Sparse high-dimensional regression: Exact scalable algorithms and
  phase transitions.
\newblock \emph{The Annals of Statistics}, 48\penalty0 (1):\penalty0 300--323,
  2020.

\bibitem[Bertsimas et~al.(2016)Bertsimas, King, and
  Mazumder]{bertsimas2016best}
D.~Bertsimas, A.~King, and R.~Mazumder.
\newblock Best subset selection via a modern optimization lens.
\newblock \emph{The Annals of Statistics}, pages 813--852, 2016.

\bibitem[Bertsimas et~al.(2017)Bertsimas, Copenhaver, and
  Mazumder]{bertsimas2017certifiably}
D.~Bertsimas, M.~S. Copenhaver, and R.~Mazumder.
\newblock Certifiably optimal low rank factor analysis.
\newblock \emph{Journal of Machine Learning Research}, 18\penalty0
  (1):\penalty0 907--959, 2017.

\bibitem[Bertsimas et~al.(2020)Bertsimas, Pauphilet, and
  Van~Parys]{bertsimas2019sparse}
D.~Bertsimas, J.~Pauphilet, and B.~Van~Parys.
\newblock Sparse regression: Scalable algorithms and empirical performance.
\newblock \emph{Statistical Science}, 35\penalty0 (4):\penalty0 555--578, 2020.

\bibitem[Bertsimas et~al.(2021{\natexlab{a}})Bertsimas, Cory-Wright, and
  Pauphilet]{bertsimas2019unified}
D.~Bertsimas, R.~Cory-Wright, and J.~Pauphilet.
\newblock A unified approach to mixed-integer optimization problems with
  logical constraints.
\newblock \emph{SIAM Journal on Optimization}, 31\penalty0 (3):\penalty0
  2340--2367, 2021{\natexlab{a}}.

\bibitem[Bertsimas et~al.(2021{\natexlab{b}})Bertsimas, Cory-Wright, and
  Pauphilet]{bertsimas2020mixed}
D.~Bertsimas, R.~Cory-Wright, and J.~Pauphilet.
\newblock Mixed-projection conic optimization: A new paradigm for modeling rank
  constraints.
\newblock \emph{Operations Research, Articles in Advance}, 2021{\natexlab{b}}.

\bibitem[Bertsimas et~al.(2022)Bertsimas, Cory-Wright, and
  Pauphilet]{bertsimas2020solving}
D.~Bertsimas, R.~Cory-Wright, and J.~Pauphilet.
\newblock Solving large-scale sparse {PCA} to certifiable (near) optimality.
\newblock \emph{Journal of Machine Learning Research}, 23\penalty0
  (13):\penalty0 1--35, 2022.

\bibitem[Bhatia(2013)]{bhatia2013matrix}
R.~Bhatia.
\newblock \emph{Matrix analysis}, volume 169.
\newblock Springer Science \& Business Media New York, 2013.

\bibitem[Bienstock(2010)]{bienstock2010eigenvalue}
D.~Bienstock.
\newblock Eigenvalue techniques for convex objective, nonconvex optimization
  problems.
\newblock In \emph{International Conference on Integer Programming and
  Combinatorial Optimization}, pages 29--42. Springer, 2010.

\bibitem[Boyd and Vandenberghe(2004)]{boyd2004convex}
S.~Boyd and L.~Vandenberghe.
\newblock \emph{Convex Optimization}.
\newblock Cambridge University Press, Cambridge, UK, 2004.

\bibitem[Boyd et~al.(1994)Boyd, El~Ghaoui, Feron, and
  Balakrishnan]{boyd1994linear}
S.~Boyd, L.~El~Ghaoui, E.~Feron, and V.~Balakrishnan.
\newblock \emph{Linear matrix inequalities in system and control theory},
  volume~15.
\newblock Studies in Applied Mathematics, Society for Industrial and Applied
  Mathematics, Philadelphia, PA, 1994.

\bibitem[Burer(2009)]{burer2009copositive}
S.~Burer.
\newblock On the copositive representation of binary and continuous nonconvex
  quadratic programs.
\newblock \emph{Mathematical Programming}, 120\penalty0 (2):\penalty0 479--495,
  2009.

\bibitem[Burer and Monteiro(2003)]{burer2003nonlinear}
S.~Burer and R.~D. Monteiro.
\newblock A nonlinear programming algorithm for solving semidefinite programs
  via low-rank factorization.
\newblock \emph{Mathematical Programming}, 95\penalty0 (2):\penalty0 329--357,
  2003.

\bibitem[Cand{\`e}s and Recht(2009)]{candes2009exact}
E.~J. Cand{\`e}s and B.~Recht.
\newblock Exact matrix completion via convex optimization.
\newblock \emph{Foundations of Computational Mathematics}, 9\penalty0
  (6):\penalty0 717, 2009.

\bibitem[Carlen(2010)]{carlen2010trace}
E.~Carlen.
\newblock Trace inequalities and quantum entropy: an introductory course.
\newblock \emph{Entropy and the quantum}, 529:\penalty0 73--140, 2010.

\bibitem[Ceria and Soares(1999)]{ceria1999convex}
S.~Ceria and J.~Soares.
\newblock Convex programming for disjunctive convex optimization.
\newblock \emph{Mathematical Programming}, 86\penalty0 (3):\penalty0 595--614,
  1999.

\bibitem[Chares(2009)]{chares2009cones}
R.~Chares.
\newblock \emph{Cones and interior-point algorithms for structured convex
  optimization involving powers and exponentials}.
\newblock PhD thesis, UCL-Universit{\'e} Catholique de Louvain, 2009.

\bibitem[Combettes(2018)]{combettes2018perspective}
P.~L. Combettes.
\newblock Perspective functions: Properties, constructions, and examples.
\newblock \emph{Set-Valued and Variational Analysis}, 26\penalty0 (2):\penalty0
  247--264, 2018.

\bibitem[Dacorogna and Mar{\'e}chal(2008)]{dacorogna2008role}
B.~Dacorogna and P.~Mar{\'e}chal.
\newblock The role of perspective functions in convexity, polyconvexity,
  rank-one convexity and separate convexity.
\newblock \emph{Journal of Convex Analysis}, 15\penalty0 (2):\penalty0
  271--284, 2008.

\bibitem[Dong et~al.(2015)Dong, Chen, and Linderoth]{dong2015regularization}
H.~Dong, K.~Chen, and J.~Linderoth.
\newblock Regularization vs. relaxation: A conic optimization perspective of
  statistical variable selection.
\newblock \emph{arXiv:1510.06083}, 2015.

\bibitem[Ebadian et~al.(2011)Ebadian, Nikoufar, and
  Gordji]{ebadian2011perspectives}
A.~Ebadian, I.~Nikoufar, and M.~E. Gordji.
\newblock Perspectives of matrix convex functions.
\newblock \emph{Proceedings of the National Academy of Sciences}, 108\penalty0
  (18):\penalty0 7313--7314, 2011.

\bibitem[Effros and Hansen(2014)]{effros2014non}
E.~Effros and F.~Hansen.
\newblock Non-commutative perspectives.
\newblock \emph{Annals of Functional Analysis}, 5\penalty0 (2):\penalty0
  74--79, 2014.

\bibitem[Effros(2009)]{effros2009matrix}
E.~G. Effros.
\newblock A matrix convexity approach to some celebrated quantum inequalities.
\newblock \emph{Proceedings of the National Academy of Sciences}, 106\penalty0
  (4):\penalty0 1006--1008, 2009.

\bibitem[Fan and Li(2001)]{fan2001variable}
J.~Fan and R.~Li.
\newblock Variable selection via nonconcave penalized likelihood and its oracle
  properties.
\newblock \emph{Journal of the American statistical Association}, 96\penalty0
  (456):\penalty0 1348--1360, 2001.

\bibitem[Farias and Li(2019)]{farias2019learning}
V.~F. Farias and A.~A. Li.
\newblock Learning preferences with side information.
\newblock \emph{Management Science}, 65\penalty0 (7):\penalty0 3131--3149,
  2019.

\bibitem[Fawzi and Saunderson(2017)]{fawzi2017lieb}
H.~Fawzi and J.~Saunderson.
\newblock Lieb's concavity theorem, matrix geometric means, and semidefinite
  optimization.
\newblock \emph{Linear Algebra and its Applications}, 513:\penalty0 240--263,
  2017.

\bibitem[Fawzi et~al.(2019)Fawzi, Saunderson, and
  Parrilo]{fawzi2019semidefinite}
H.~Fawzi, J.~Saunderson, and P.~A. Parrilo.
\newblock Semidefinite approximations of the matrix logarithm.
\newblock \emph{Foundations of Computational Mathematics}, 19\penalty0
  (2):\penalty0 259--296, 2019.

\bibitem[Faybusovich and Zhou(2020)]{faybusovich2020self}
L.~Faybusovich and C.~Zhou.
\newblock Self-concordance and matrix monotonicity with applications to quantum
  entanglement problems.
\newblock \emph{Applied Mathematics and Computation}, 375:\penalty0 125071,
  2020.

\bibitem[Fazel et~al.(2003)Fazel, Hindi, and Boyd]{fazel2003log}
M.~Fazel, H.~Hindi, and S.~P. Boyd.
\newblock Log-det heuristic for matrix rank minimization with applications to
  {H}ankel and {E}uclidean distance matrices.
\newblock In \emph{Proceedings of the 2003 American Control Conference, 2003.},
  volume~3, pages 2156--2162. IEEE, 2003.

\bibitem[Fischetti et~al.(2016)Fischetti, Ljubi{\'c}, and
  Sinnl]{fischetti2016redesigning}
M.~Fischetti, I.~Ljubi{\'c}, and M.~Sinnl.
\newblock Redesigning {B}enders decomposition for large-scale facility
  location.
\newblock \emph{Management Science}, 63\penalty0 (7):\penalty0 2146--2162,
  2016.

\bibitem[Frangioni and Gentile(2006)]{frangioni2006perspective}
A.~Frangioni and C.~Gentile.
\newblock Perspective cuts for a class of convex 0--1 mixed integer programs.
\newblock \emph{Mathematical Programming}, 106\penalty0 (2):\penalty0 225--236,
  2006.

\bibitem[Frangioni and Gentile(2009)]{frangioni2009computational}
A.~Frangioni and C.~Gentile.
\newblock A computational comparison of reformulations of the perspective
  relaxation: {SOCP} vs. cutting planes.
\newblock \emph{Operations Research Letters}, 37\penalty0 (3):\penalty0
  206--210, 2009.

\bibitem[Frangioni et~al.(2020)Frangioni, Gentile, and
  Hungerford]{frangioni2020decompositions}
A.~Frangioni, C.~Gentile, and J.~Hungerford.
\newblock Decompositions of semidefinite matrices and the perspective
  reformulation of nonseparable quadratic programs.
\newblock \emph{Mathematics of Operations Research}, 45\penalty0 (1):\penalty0
  15--33, 2020.

\bibitem[Gandy et~al.(2011)Gandy, Recht, and Yamada]{gandy2011tensor}
S.~Gandy, B.~Recht, and I.~Yamada.
\newblock Tensor completion and low-n-rank tensor recovery via convex
  optimization.
\newblock \emph{Inverse Problems}, 27\penalty0 (2):\penalty0 025010, 2011.

\bibitem[Ge and Ye(2010)]{ge2010doubly}
D.~Ge and Y.~Ye.
\newblock On doubly positive semidefinite programming relaxations.
\newblock \emph{Optimization Online}, 2010.

\bibitem[G{\"u}nl{\"u}k and Linderoth(2010)]{gunluk2010perspective}
O.~G{\"u}nl{\"u}k and J.~Linderoth.
\newblock Perspective reformulations of mixed integer nonlinear programs with
  indicator variables.
\newblock \emph{Mathematical Programming}, 124\penalty0 (1-2):\penalty0
  183--205, 2010.

\bibitem[Han et~al.(2020)Han, G{\'o}mez, and Atamt{\"u}rk]{han20202x2}
S.~Han, A.~G{\'o}mez, and A.~Atamt{\"u}rk.
\newblock 2x2 convexifications for convex quadratic optimization with indicator
  variables.
\newblock \emph{arXiv:2004.07448}, 2020.

\bibitem[Hazimeh et~al.(2021)Hazimeh, Mazumder, and Saab]{hazimeh2020sparse}
H.~Hazimeh, R.~Mazumder, and A.~Saab.
\newblock Sparse regression at scale: Branch-and-bound rooted in first-order
  optimization.
\newblock \emph{Mathematical Programming, articles in advance}, pages 1--42,
  2021.

\bibitem[Hiai and Petz(1991)]{hiai1991proper}
F.~Hiai and D.~Petz.
\newblock The proper formula for relative entropy and its asymptotics in
  quantum probability.
\newblock \emph{Communications in Mathematical Physics}, 143\penalty0
  (1):\penalty0 99--114, 1991.

\bibitem[Hiriart-Urruty and Lemar{\'e}chal(2013)]{hiriart2013convex}
J.-B. Hiriart-Urruty and C.~Lemar{\'e}chal.
\newblock \emph{Convex analysis and minimization algorithms I: Fundamentals},
  volume 305.
\newblock Springer Science \& Business Media Berlin, 2013.

\bibitem[Horn and Johnson(1985)]{johnson1985matrix}
R.~A. Horn and C.~R. Johnson.
\newblock \emph{Matrix analysis}.
\newblock Cambridge {U}niversity {P}ress, {N}ew {Y}ork, 1985.

\bibitem[Joshi and Boyd(2008)]{joshi2008sensor}
S.~Joshi and S.~Boyd.
\newblock Sensor selection via convex optimization.
\newblock \emph{IEEE Transactions on Signal Processing}, 57\penalty0
  (2):\penalty0 451--462, 2008.

\bibitem[Karimi and Tun{\c{c}}el(2019)]{karimi2019domain}
M.~Karimi and L.~Tun{\c{c}}el.
\newblock Domain-driven solver ({DDS}): a {MATLAB}-based software package for
  convex optimization problems in domain-driven form.
\newblock \emph{arXiv preprint arXiv:1908.03075}, 2019.

\bibitem[Kolda and Bader(2009)]{kolda2009tensor}
T.~G. Kolda and B.~W. Bader.
\newblock Tensor decompositions and applications.
\newblock \emph{SIAM Review}, 51\penalty0 (3):\penalty0 455--500, 2009.

\bibitem[Lavaei and Low(2011)]{lavaei2011zero}
J.~Lavaei and S.~H. Low.
\newblock Zero duality gap in optimal power flow problem.
\newblock \emph{IEEE Transactions on Power Systems}, 27\penalty0 (1):\penalty0
  92--107, 2011.

\bibitem[Lewis(1996)]{lewis1996convex}
A.~S. Lewis.
\newblock Convex analysis on the {H}ermitian matrices.
\newblock \emph{SIAM Journal on Optimization}, 6\penalty0 (1):\penalty0
  164--177, 1996.

\bibitem[Lieb and Ruskai(1973)]{lieb1973proof}
E.~H. Lieb and M.~B. Ruskai.
\newblock Proof of the strong subadditivity of quantum-mechanical entropy. with
  an appendix by {B}. {S}imon.
\newblock \emph{Journal of Mathematical Physics}, 14:\penalty0 1938--1941,
  1973.

\bibitem[Mar{\'e}chal(2001)]{marechal2001convexity}
P.~Mar{\'e}chal.
\newblock On the convexity of the multiplicative potential and penalty
  functions and related topics.
\newblock \emph{Mathematical Programming}, 89\penalty0 (3):\penalty0 505--516,
  2001.

\bibitem[Mar{\'e}chal(2005{\natexlab{a}})]{marechal2005functional1}
P.~Mar{\'e}chal.
\newblock On a functional operation generating convex functions, part 1:
  duality.
\newblock \emph{Journal of Optimization Theory and Applications}, 126\penalty0
  (1):\penalty0 175--189, 2005{\natexlab{a}}.

\bibitem[Mar{\'e}chal(2005{\natexlab{b}})]{marechal2005functional2}
P.~Mar{\'e}chal.
\newblock On a functional operation generating convex functions, part 2:
  algebraic properties.
\newblock \emph{Journal of Optimization Theory and Applications}, 126\penalty0
  (2):\penalty0 357--366, 2005{\natexlab{b}}.

\bibitem[Negahban and Wainwright(2011)]{negahban2011estimation}
S.~Negahban and M.~J. Wainwright.
\newblock Estimation of (near) low-rank matrices with noise and
  high-dimensional scaling.
\newblock \emph{The Annals of Statistics}, pages 1069--1097, 2011.

\bibitem[Nemhauser et~al.(1978)Nemhauser, Wolsey, and
  Fisher]{nemhauser1978analysis}
G.~L. Nemhauser, L.~A. Wolsey, and M.~L. Fisher.
\newblock An analysis of approximations for maximizing submodular set
  functions—i.
\newblock \emph{Mathematical Programming}, 14\penalty0 (1):\penalty0 265--294,
  1978.

\bibitem[Nguyen et~al.(2019)Nguyen, Kim, and Shim]{nguyen2019low}
L.~T. Nguyen, J.~Kim, and B.~Shim.
\newblock Low-rank matrix completion: A contemporary survey.
\newblock \emph{IEEE Access}, 7:\penalty0 94215--94237, 2019.

\bibitem[Overton and Womersley(1992)]{overton1992sum}
M.~L. Overton and R.~S. Womersley.
\newblock On the sum of the largest eigenvalues of a symmetric matrix.
\newblock \emph{SIAM Journal on Matrix Analysis and Applications}, 13\penalty0
  (1):\penalty0 41--45, 1992.

\bibitem[Overton and Womersley(1993)]{overton1993optimality}
M.~L. Overton and R.~S. Womersley.
\newblock Optimality conditions and duality theory for minimizing sums of the
  largest eigenvalues of symmetric matrices.
\newblock \emph{Mathematical Programming}, 62\penalty0 (1-3):\penalty0
  321--357, 1993.

\bibitem[Pataki(1998)]{pataki1998rank}
G.~Pataki.
\newblock On the rank of extreme matrices in semidefinite programs and the
  multiplicity of optimal eigenvalues.
\newblock \emph{Mathematics of Operations Research}, 23\penalty0 (2):\penalty0
  339--358, 1998.

\bibitem[Peng and Wei(2007)]{peng2007approximating}
J.~Peng and Y.~Wei.
\newblock Approximating {K}-means-type clustering via semidefinite programming.
\newblock \emph{SIAM Journal on Optimization}, 18\penalty0 (1):\penalty0
  186--205, 2007.

\bibitem[Permenter and Parrilo(2018)]{permenter2018partial}
F.~Permenter and P.~Parrilo.
\newblock Partial facial reduction: simplified, equivalent {SDP}s via
  approximations of the {PSD} cone.
\newblock \emph{Mathematical Programming}, 171\penalty0 (1-2):\penalty0 1--54,
  2018.

\bibitem[Pilanci et~al.(2015)Pilanci, Wainwright, and
  El~Ghaoui]{pilanci2015sparse}
M.~Pilanci, M.~J. Wainwright, and L.~El~Ghaoui.
\newblock Sparse learning via {B}oolean relaxations.
\newblock \emph{Mathematical Programming}, 151\penalty0 (1):\penalty0 63--87,
  2015.

\bibitem[Plemmons and Cline(1972)]{plemmons1972generalized}
R.~Plemmons and R.~Cline.
\newblock The generalized inverse of a nonnegative matrix.
\newblock \emph{Proceedings of the American Mathematical Society}, pages
  46--50, 1972.

\bibitem[Renegar(2001)]{renegar2001mathematical}
J.~Renegar.
\newblock \emph{A mathematical view of interior-point methods in convex
  optimization}, volume~3.
\newblock Society for Industrial and Applied Mathematics, 2001.

\bibitem[Rockafellar(1970)]{rockafellar1970convex}
R.~T. Rockafellar.
\newblock \emph{Convex analysis}.
\newblock Number~28. Princeton university press, 1970.

\bibitem[Romera-Paredes and Pontil(2013)]{romera2013new}
B.~Romera-Paredes and M.~Pontil.
\newblock A new convex relaxation for tensor completion.
\newblock \emph{arXiv preprint arXiv:1307.4653}, 2013.

\bibitem[Singh and Xie(2020)]{singh2018approximation}
M.~Singh and W.~Xie.
\newblock Approximation algorithms for {D}-optimal design.
\newblock \emph{Mathematics of Operations Research}, 45:\penalty0 1193--1620,
  2020.

\bibitem[Skajaa and Ye(2015)]{skajaa2015homogeneous}
A.~Skajaa and Y.~Ye.
\newblock A homogeneous interior-point algorithm for nonsymmetric convex conic
  optimization.
\newblock \emph{Mathematical Programming}, 150\penalty0 (2):\penalty0 391--422,
  2015.

\bibitem[Stubbs(1996)]{stubbs1998branch}
R.~A. Stubbs.
\newblock \emph{Branch-and-cut methods for mixed 0-1 convex programming}.
\newblock PhD thesis, Northwestern University, 1996.

\bibitem[Stubbs and Mehrotra(1999)]{stubbs1999branch}
R.~A. Stubbs and S.~Mehrotra.
\newblock A branch-and-cut method for 0-1 mixed convex programming.
\newblock \emph{Mathematical Programming}, 86\penalty0 (3):\penalty0 515--532,
  1999.

\bibitem[Wang and K{\i}l{\i}n{\c{c}}-Karzan(2021)]{wang2019tightness}
A.~L. Wang and F.~K{\i}l{\i}n{\c{c}}-Karzan.
\newblock On the tightness of {SDP} relaxations of {QCQP}s.
\newblock \emph{Mathematical Programming, Articles in Advance}, pages 1--41,
  2021.

\bibitem[Wolkowicz et~al.(2012)Wolkowicz, Saigal, and
  Vandenberghe]{wolkowicz2012handbook}
H.~Wolkowicz, R.~Saigal, and L.~Vandenberghe.
\newblock \emph{Handbook of semidefinite programming: theory, algorithms, and
  applications}, volume~27.
\newblock Springer Science \& Business Media, 2012.

\bibitem[Xie and Deng(2020)]{xie2018ccp}
W.~Xie and X.~Deng.
\newblock Scalable algorithms for the sparse ridge regression.
\newblock \emph{SIAM Journal on Optimization}, 30\penalty0 (4):\penalty0
  3359--3386, 2020.

\bibitem[Zhang(2010)]{zhang2010nearly}
C.-H. Zhang.
\newblock Nearly unbiased variable selection under minimax concave penalty.
\newblock \emph{The Annals of Statistics}, 38\penalty0 (2):\penalty0 894--942,
  2010.

\bibitem[Zheng et~al.(2014)Zheng, Sun, and Li]{zheng2014improving}
X.~Zheng, X.~Sun, and D.~Li.
\newblock Improving the performance of {MIQP} solvers for quadratic programs
  with cardinality and minimum threshold constraints: A semidefinite program
  approach.
\newblock \emph{INFORMS Journal on Computing}, 26\penalty0 (4):\penalty0
  690--703, 2014.

\end{thebibliography}

}
\begin{appendices}

\section{\blue Background on Operator Functions}\label{sec:A.background}
{\blue In this work, we make repeated use of operator functions, i.e., functions defined from the spectral decomposition of a matrix. Namely, for any function $\omega : \mathbb{R} \rightarrow \mathbb{R}$, its corresponding operator function $f_\omega : \mathcal{S}^n \rightarrow \mathcal{S}^n$ is defined as
\begin{align*}
    f_\omega(\bm{X}) = \bm{U} \operatorname{Diag}(\omega(\lambda_1^x),\dots, \omega(\lambda_n^x)) \bm{U}^\top
\end{align*}
where $\bm{X} = \bm{U} \operatorname{Diag}(\lambda_1^x,\dots, \lambda_n^x) \bm{U}^\top$ is an eigendecomposition of $\bm{X}$. In this appendix, we present some common examples and useful properties of operator functions.}

\subsection{\blue Examples: Matrix exponential and logarithm} \label{ssec:A.explog}
For self-consistency of the paper, we now define {\blue the matrix exponential and logarithm} functions and summarize their properties. These results are well known and can be found in modern matrix analysis textbooks \citep[see, e.g.,][]{bhatia2013matrix} 
\begin{definition}[Matrix exponential]
Let $\bm{X} \in \mathcal{S}^n$ be a symmetric matrix with eigendecomposition $\bm{X}=\bm{U}\bm{\Lambda}\bm{U}^\top$. Letting $\exp(\bm{\Lambda})=\mathrm{diag}(e^{\lambda_1}, e^{\lambda_2}, \ldots, e^{\lambda_n})$, we define
    $\exp(\bm{X}):=\bm{U}\exp(\bm{\Lambda})\bm{U}^\top$.
\end{definition}
\begin{proposition}
The matrix exponential, $\exp: \mathcal{S}^n \rightarrow \mathcal{S}^n_+$, satisfies the following properties:
\begin{itemize}
    \item Power series expansion: $\exp(\bm{X})=\mathbb{I}+\sum_{i=1}^\infty \frac{1}{i!}\bm{X}^i$.
    \item Trace monotonicity: $\bm{X} \preceq \bm{Y} \implies \mathrm{tr}(\exp(\bm{X}))\leq \mathrm{tr}(\exp(\bm{Y}))$.
    \item Golden-Thompson-inequality: $\mathrm{tr}(\exp(\bm{X}+\bm{Y})) \leq \mathrm{tr}(\exp(\bm{X}))+\mathrm{tr}(\exp(\bm{Y}))$.
\end{itemize}
\end{proposition}
\begin{remark}
The matrix exponential is not monotone: $\bm{X} \preceq \bm{Y} \centernot\implies \exp(\bm{X}) \preceq \exp(\bm{Y})$ \citep[Ch.V]{bhatia2013matrix}.
\end{remark}

\begin{definition}[Matrix logarithm]
Let $\bm{X} \in \mathcal{S}^n$ be a symmetric matrix with eigendecomposition $\bm{X}=\bm{U}\bm{\Lambda}\bm{U}^\top$. Letting $\log(\bm{\Lambda})=\mathrm{diag}(\log(\lambda_1), \log(\lambda_2), \ldots, \log(\lambda_n))$, we have
    $\log(\bm{X}):=\bm{U}\log(\bm{\Lambda})\bm{U}^\top$.
\end{definition}
\begin{proposition}
The matrix logarithm, $\log(\bm{X}): \mathcal{S}^n_{++} \rightarrow \mathcal{S}^n$, satisfies the following properties:
\begin{itemize}
    \item Operator monotonicity: $\bm{X} \preceq \bm{Y} \implies \log(\bm{X}) \preceq  \log(\bm{Y})$.
    \item Functional inversion: $\log(\exp(\bm{X}))=\bm{X} \quad \forall \bm{X} \in \mathcal{S}^n$.
    \item Jacobi formula I: $\mathrm{tr}(\log(\bm{X}))=\log\det(\bm{X})$.
    \item Jacobi formula II: $\exp\left(\frac{1}{n}\mathrm{tr}\log(\bm{X})\right)=\det(\bm{X})^\frac{1}{n}$.
\end{itemize}
\end{proposition}

\subsection{\blue Properties of operator functions}
{\blue  Among other properties, one can show that the trace of operator functions is invariant under an orthogonal rotation, i.e., $\mathrm{tr}(f_\omega(\bm{X}))=\mathrm{tr}(f_\omega(\bm{U}^\top \bm{X}\bm{U}))$ for any orthogonal rotation $\bm{U}$. Also, if $\omega$ is analytical, then $f_\omega$ is also analytical with the same Taylor expansion. 

In our analysis (in particular the proof of Proposition \ref{prop:operatorperspective}), we will use this simple bound on $\bm{v}^\top f_\omega(\bm{A})\bm{v}$ in the case where $\omega$ is convex:
\begin{lemma} \label{lemma:peierls} Consider a convex function $\omega : \mathbb{R} \rightarrow \mathbb{R}$ and a symmetric matrix $\bm{A} \in \mathcal{S}^n$. Consider 
a unit vector $\bm{v}$. Then,  
\begin{align*}
    \bm{v}^\top f_\omega(\bm{A}) \bm{v} \geq\omega \left( \bm{v}^\top \bm{A} \bm{v} \right).
\end{align*}
\end{lemma}
\begin{proof} Consider a spectral decomposition of $\bm{A}$, $\bm{A} = \sum_{i=1}^n \lambda_i \bm{u}_i \bm{u}_i^\top$. Then, $f_\omega(\bm{A}) = \sum_{i=1}^n \omega(\lambda_i) \bm{u}_i\bm{u}_i^\top$ and
\begin{align*}
    \bm{v}^\top f_\omega(\bm{A}) \bm{v} = \sum_{i=1}^n \omega(\lambda_i) \bm{v}^\top \bm{u}_i \bm{u}_i^\top \bm{v}
     \geq \omega\left( \sum_{i=1}^n \lambda_i \bm{v}^\top \bm{u}_i \bm{u}_i^\top \bm{v} \right) =
    \omega \left( \bm{v}^\top \bm{A} \bm{v} \right),
\end{align*}
where the inequality comes from the convexity of $\omega$ since $\bm{v}^\top \bm{u}_i \bm{u}_i^\top \bm{v} = (\bm{u}_i^\top \bm{v})^2 \geq 0$ and $\sum_{i =1}^n \bm{v}^\top \bm{u}_i \bm{u}_i^\top \bm{v} = \bm{v}^\top \left( \sum_{i =1}^n  \bm{u}_i \bm{u}_i^\top \right) \bm{v} = \| \bm{v} \|^2 = 1$.
\end{proof}
}

\section{Omitted Proofs}\label{append:proofs}
In this section, we supply all omitted proofs, in the order the results were stated.

{\color{black}
\subsection{Proof of Proposition \ref{prop:operatorperspective}}\label{ssec:A.proof.opepersp}
\begin{proof} Fix $\bm{X} \in \mathcal{S}^n$. For $\bm{Y} \succ \bm{0}$, the perspective of $f_\omega$ is well-defined according to Definition \ref{defn:matrixconv}. Now, consider an arbitrary $\bm{Y} \succeq \bm{0}$ and define $\bm{P}$ as the orthogonal projection onto the kernel of $\bm{Y}$, which is orthogonal to $\operatorname{Span}(\bm{Y})$. Then, $\bm{Y}_\varepsilon := \bm{Y} + \varepsilon \bm{P}$ for $\varepsilon > 0$ is invertible. The closure of the matrix perspective of $f_\omega$ is defined by continuity as the limit of $\bm{M}_\varepsilon := \bm{Y}_\varepsilon^{\frac{1}{2}} f_\omega\left( \bm{Y}_\varepsilon^{-\frac{1}{2}}\bm{X}\bm{Y}_\varepsilon^{-\frac{1}{2}} \right) \bm{Y}_\varepsilon^{\frac{1}{2}}$ for $\varepsilon \to 0$.

Since the ranges of $\bm{Y}$ and $\bm{P}$ are orthogonal ($\bm{Y} \bm{P} = \bm{P} \bm{Y} = \bm{0}$), we have $\bm{Y}_\varepsilon^{-\frac{1}{2}} = \bm{Y}^{-\frac{1}{2}} + {\varepsilon}^{-\frac{1}{2}} \bm{P}$, and
\begin{align*}
    \bm{Y}_\varepsilon^{-\frac{1}{2}}\bm{X}\bm{Y}_\varepsilon^{-\frac{1}{2}} &=
    \bm{Y}^{-\frac{1}{2}}\bm{X}\bm{Y}^{-\frac{1}{2}} + {\varepsilon}^{-\frac{1}{2}} \bm{P} \bm{X}\bm{Y}^{-\frac{1}{2}} +
    {\varepsilon}^{-\frac{1}{2}} \bm{Y}^{-\frac{1}{2}}\bm{X} \bm{P} + 
    {\varepsilon}^{-1} \bm{P} \bm{X} \bm{P}.
\end{align*}
Note that $\displaystyle \lim_{\varepsilon \to 0} \bm{Y}_\varepsilon^{\frac{1}{2}} = \bm{Y}^{\frac{1}{2}}$ but $\displaystyle \lim_{\varepsilon \to 0} \bm{Y}_\varepsilon^{-\frac{1}{2}} \neq \bm{Y}^{-\frac{1}{2}}$. 
We now distinguish two cases. 

{\bf Case 1: } If $\operatorname{span}(\bm{X}) \subseteq \operatorname{span}(\bm{Y})$, $\bm{X} \bm{P} = \bm{P} \bm{X} = \bm{0}$ so 
\begin{align*}
    \bm{Y}_\varepsilon^{-\frac{1}{2}}\bm{X}\bm{Y}_\varepsilon^{-\frac{1}{2}} &=
    \bm{Y}^{-\frac{1}{2}}\bm{X}\bm{Y}^{-\frac{1}{2}}, \\
    \bm{M}_\varepsilon &= \bm{Y}_\varepsilon^{\frac{1}{2}} f_\omega\left(  \bm{Y}^{-\frac{1}{2}}\bm{X}\bm{Y}^{-\frac{1}{2}} \right) \bm{Y}_\varepsilon^{\frac{1}{2}} \quad 
    \to_{\varepsilon \to 0} \bm{Y}^{\frac{1}{2}} f_\omega\left(  \bm{Y}^{-\frac{1}{2}}\bm{X}\bm{Y}^{-\frac{1}{2}} \right) \bm{Y}^{\frac{1}{2}}.
\end{align*}

{\bf Case 2: } If $\operatorname{span}(\bm{X}) \not\subseteq \operatorname{span}(\bm{Y})$, consider an orthonormal basis of $\mathbb{R}^n$ such that $\bm{u}_1, \dots, \bm{u}_k$ is an eigenbasis of $\operatorname{Span}(\bm{Y})$ (with respective eigenvalues $\lambda^y_{1}, \dots \lambda^y_{k}$) and  $\bm{u}_{k+1}, \dots, \bm{u}_n$ is a basis of $\operatorname{Span}(\bm{Y})^\perp = \operatorname{Ker}(\bm{Y})$.  By assumption, $k < n$ and there exists $j > k$ such that $\bm{u}_j^\top \bm{X} \bm{u}_j \neq 0$. Without loss of generality, we shall assume $\bm{u}_n^\top \bm{X} \bm{u}_n \neq 0$. We show that the matrix $\bm{M}_\varepsilon$ goes to infinity as $\varepsilon \to 0$  by showing that $\bm{u}_n^\top \bm{M}_\varepsilon \bm{u}_n$ diverges. 

Since $\bm{Y}_\varepsilon^{\pm \frac{1}{2}} \bm{u}_n = \varepsilon^{\pm \frac{1}{2}} \bm{u}_n$,
we have
\begin{align*}
\bm{u}_n^\top \bm{M}_\varepsilon \bm{u}_n 
=  \varepsilon \ \bm{u}_n^\top f_\omega\left( \bm{Y}_\varepsilon^{-\frac{1}{2}}\bm{X}\bm{Y}_\varepsilon^{-\frac{1}{2}} \right) \bm{u}_n 
\quad \geq \varepsilon \ \omega \left( \bm{u}_n^\top \bm{Y}_\varepsilon^{-\frac{1}{2}}\bm{X}\bm{Y}_\varepsilon^{-\frac{1}{2}} \bm{u}_n \right) 
= \varepsilon \ \omega \left( \varepsilon^{-1} \bm{u}_n^\top \bm{X} \bm{u}_n \right),
\end{align*}
where the inequality follows from the convexity of $\omega$ and Lemma \ref{lemma:peierls}. By Assumption \ref{ass:coercive},
\begin{align*}
    \lim_{\varepsilon \to 0} \varepsilon  \omega \left( \varepsilon^{-1} \bm{u}_n^\top \bm{X}\bm{u}_n \right) =  \omega_{\infty}(\bm{u}_n^\top \bm{X}\bm{u}_n) = +\infty,
\end{align*}
because $\bm{u}_n^\top \bm{X}\bm{u}_n \neq 0$ and $\omega$ is coercive. \quad \qed
\end{proof}
}

{\blue We now provide a simple extension of Proposition \ref{prop:operatorperspective} that will prove useful later in our exposition. 
\begin{corollary} Consider a function $\omega : \mathbb{R} \rightarrow \mathbb{R}$ satisfying Assumption \ref{ass:coercive} and denote its associated operator function $f_\omega$. Consider a closed set $\mathcal{X} \subseteq \mathcal{S}^n$ and define 
\begin{align*}
    f(\bm{X}) = \begin{cases} f_\omega(\bm{X}) & \mbox{ if } \bm{X} \in \mathcal{X}, \\ +\infty & \mbox{ otherwise.} \end{cases}
\end{align*}
Then, the closure of the matrix perspective of $f$ is, for any $\bm{X} \in \mathcal{S}^n$, $\bm{Y} \in \mathcal{S}_+^n$, 
\begin{align*}
    g_{f}(\bm{X},\bm{Y}) = \begin{cases} \bm{Y}^\frac{1}{2} f_\omega(\bm{Y}^{-\frac{1}{2}}\bm{X}\bm{Y}^{-\frac{1}{2}})\bm{Y}^{\frac{1}{2}} & \mbox{ if } \operatorname{Span}(\bm{X}) \subseteq \operatorname{Span}(\bm{Y}), \bm{Y} \succeq \bm{0}, \bm{Y}^{-\frac{1}{2}}\bm{X}\bm{Y}^{-\frac{1}{2}} \in \mathcal{X}, \\
    \infty & \mbox{ otherwise, }
    \end{cases}
\end{align*}
where $\bm{Y}^{-\frac{1}{2}}$ denotes the pseudo-inverse of the square root of $\bm{Y}$.
\end{corollary}
\begin{proof} Fix $\bm{X} \in \mathcal{S}^n$ and $\bm{Y} \in \mathcal{S}^n_+$. From Proposition \ref{prop:operatorperspective}, we know that $g_f(\bm{X},\bm{Y}) = + \infty$ if $\operatorname{Span}(\bm{X}) \not\subseteq \operatorname{Span}(\bm{Y})$. Let us assume that $\operatorname{Span}(\bm{X}) \subseteq \operatorname{Span}(\bm{Y})$. Following the same construction as in the proof of Proposition  \ref{prop:operatorperspective}, we obtain a sequence $\bm{Y}_\varepsilon$ that converges to $\bm{Y}$ as $\varepsilon \rightarrow 0$ and such that $\bm{Y}_\varepsilon^{-\frac{1}{2}}\bm{X}\bm{Y}_\varepsilon^{-\frac{1}{2}} = \bm{Y}^{-\frac{1}{2}}\bm{X}\bm{Y}^{-\frac{1}{2}}$, which concludes the proof. \qed
\end{proof}
}
{\blue
\subsection{Perspective functions with non-commuting matrices} \label{ssec:A.persp.noncomm}
In contrast with Proposition \ref{prop:persp.commute}, in the general case where $\bm{X}$ and $\bm{Y}$ do not commute, we cannot simultaneously diagonalize them and connect $g_{f_\omega}$ with $g_\omega$. However, we can still project $\bm{Y}$ onto the space of matrices that commute with $\bm{X}$ and obtain the following result when $g_{f_\omega}$ is matrix convex: 
\begin{lemma} \label{lemma:traceineq.proj}
Let $\bm{X} \in \mathcal{S}^n$ and $\bm{Y} \in \mathcal{S}_+^n$ be matrices, and define $\mathcal{X} := \{ \bm{M} \ : \ \bm{MX} = \bm{XM} \}$ as the set of matrices which commute with $\bm{X}$. For any matrix $\bm{M}$, denote $\bm{M}_{|\mathcal{X}}$ the orthogonal projection of $\bm{M}$ onto $\mathcal{X}$. Then, since $\bm{M} \mapsto \bm{M}_{|\mathcal{X}}$ is a projection operator, we have that
\begin{align*}
    \bm{Y}_{|\mathcal{X}} \in \mathcal{S}^n_+,\ \mbox{ and } \ \operatorname{tr}\left( \bm{Y}_{|\mathcal{X}} \right) =\operatorname{tr}\left( \bm{Y} \right).
\end{align*}
Moreover, if $\bm{Y} \mapsto g_{f_\omega}(\bm{X},\bm{Y})$ is matrix convex, then we have
\begin{align*}
    \operatorname{tr}\left[ g_{f_\omega}(\bm{X}, \bm{Y}_{|\mathcal{X}}) \right] \leq  \operatorname{tr}\left[ g_{f_\omega}(\bm{X}, \bm{Y}) \right].
\end{align*}
\end{lemma}
\begin{proof}
First, let us observe that $\mathcal{X}$ is a closed subset of $\mathcal{S}^n$, contains the identity, and is closed under multiplication and transposition, also know as a Von Neumann subalgebra  \citep[see][Section 4 for a detailed treatment of projections onto subalgebras]{carlen2010trace}. 
The orthogonal projection of a semidefinite matrix onto $\mathcal{X}$ is also semidefinite and has the same trace \citep[Theorem. 4.13]{carlen2010trace}, so
\begin{align*}
    \operatorname{tr}\left( \bm{Y}_{|\mathcal{X}} \right) =\operatorname{tr}\left( \bm{Y} \right).
\end{align*}
Furthermore, since $\bm{Y} \mapsto g_{f_\omega}(\bm{X},\bm{Y})$ is matrix convex, \citet[Theorem 4.16]{carlen2010trace} yields
\begin{align*}
    g_{f_\omega}(\bm{X}, \bm{Y}_{|\mathcal{X}}) \preceq g_{f_\omega}(\bm{X}, \bm{Y})_{|\mathcal{X}}.
\end{align*}
Taking the trace on both sides and using that $\operatorname{tr}\left( g_{f_\omega}(\bm{X}, \bm{Y})_{|\mathcal{X}} \right) = \operatorname{tr}\left( g_{f_\omega}(\bm{X}, \bm{Y}) \right)$ concludes the proof. \qed
\end{proof}
In other words, taking the projection of $\bm{Y}$ onto the commutant of $\bm{X}$ is a trace preserving operation that can only reduce the value of $\operatorname{tr}\left( g_{f_\omega}(\bm{X}, \cdot) \right)$. In this paper, we invoke the projection onto $\mathcal{X}$ (a non-convex set) for theoretical purposes, not computational ones. So we are not interested in how to compute $\bm{Y}_{|\mathcal{X}}$ in practice. Note that, according to Proposition \ref{prop:genperspproperties}(a), Lemma \ref{lemma:traceineq.proj} holds if $f_\omega$ is matrix convex.
}
{\color{black}
\subsection{Counterexample to joint convexity of trace of matrix perspective of cube}\label{append:counterexample}
In this section, we demonstrate by counterexample that if $\omega$ is a convex and continuous function then, even though the trace of its matrix extension, $\mathrm{tr}(f_\omega)$, is convex \citep[c.f.][Theorem 2.10]{carlen2010trace}, the trace of its matrix perspective need not be convex. 

Specifically, let us consider $\omega(x)=x^3$. In this case, $\omega$ is convex on $\mathbb{R}_+$, $f_\omega$ is not matrix convex, but $\mathrm{tr}(f_\omega)$ is matrix convex. We have that 
\begin{align*}
    \mathrm{tr}(g_{f_\omega}(\bm{X}, \bm{Y}))=\mathrm{tr}\left(\bm{X}\bm{Y}^\dag\bm{X}\bm{Y}^\dag\bm{X}\right)
\end{align*}
for $\bm{X} \in \mathrm{Span}(\bm{Y}), \bm{X}, \bm{Y} \in \mathcal{S}^n_+$.
Let us now consider
\begin{align*}
    \bm{Y}_1=\begin{pmatrix} 0.160378 & 0.343004 \\ 0.343004 & 0.764592\end{pmatrix}, \quad \bm{Y}_2=\begin{pmatrix} 0.0859208 & 0.181976 \\ 0.181976 & 0.52666\end{pmatrix},\\
    \bm{X}_1=\begin{pmatrix}0.242865 & 0.543321\\ 0.543321 & 1.26604 \end{pmatrix}, \quad \bm{X}_2=\begin{pmatrix} 0.0595215 & 0.241702\\ 0.241702 & 1.0596\end{pmatrix}.
\end{align*}
Then, some elementary algebra reveals that 
\begin{align*}
    \mathrm{tr}\left[g_{f_\omega}\left(\tfrac{1}{2}\bm{X}_1+\tfrac{1}{2}\bm{X}_2, \tfrac{1}{2}\bm{Y}_1+\tfrac{1}{2}\bm{Y}_2\right)\right]=6.248327,
\end{align*}
while 
\begin{align*}
    \tfrac{1}{2}\mathrm{tr}\left[g_{f_\omega}\left(\bm{X}_1, \bm{Y}_1\right)\right]+\tfrac{1}{2}\mathrm{tr}\left[g_{f_\omega}\left(\bm{X}_2, \bm{Y}_2\right)\right]=6.23977,
\end{align*}
which verifies that $\mathrm{tr}(g_{f_\omega}(\bm{X}, \bm{Y}))$ is not midpoint convex in $(\bm{X}, \bm{Y})$, despite $\mathrm{tr}(f_\omega)$ being convex. 
}

\subsection{Proof of Proposition \ref{prop:powerconecl}}
\begin{proof}
We use the proof technique laid out in \citep[Section 3.1]{han20202x2}, namely writing $\mathcal{T}$ as the disjunction of two convex sets driven by whether $z$ is active and applying Fourier-Motzkin elimination. That is, we have $\mathcal{T}=\mathcal{T}^1 \cup \mathcal{T}^2$ where:
\begin{align*}
    & \mathcal{T}^1=\left\{(0,y_1,0,t_1): t_1\geq \vert y_1-d\vert^q \right\},\\
    & \mathcal{T}^2=\left\{(x_2,y_2,1,t_2): t_2\geq \vert x_2-y_2-d\vert^q, \vert x_2 \vert \leq M\right\}.
\end{align*}
Moreover, a point $(x,y,z,t)$ is in the convex hull $\mathcal{T}^c$ if and only if it can be written as a convex combination of points in $\mathcal{T}^1, \mathcal{T}^2$. Letting $\lambda_1, \lambda_2$ denote the weight of points in this system, we then have that $(x,y,z,t) \in \mathcal{T}^c$ if and only if the following system admits a solution:
\begin{equation}
\begin{aligned}
    &\lambda_1+\lambda_2=1, \\
    &x=\lambda_2 x_2 ,\\
    &y=\lambda_1 y_1+\lambda_2 y_2,\\
    &t=\lambda_1 t_1+\lambda_2 t_2,\\
    &z=\lambda_2,\\
    &t_1 \geq \vert y_1-d\vert^q,\\
    &t_2 \geq \vert x_2+y_2-d\vert^q,\\
    &\lambda_1, \lambda_2 \geq 0,\\
    & \vert x_2 \vert \leq M.
\end{aligned}
\end{equation}
For ease of computation, we now eliminate variables. First, one can substitute $t_1, t_2$ for their lower bounds in the definition of $t$ and replace $\lambda_2$ with $z$ to obtain
\begin{equation}
\begin{aligned}
    &\lambda_1+z=1, \\
    &x=z x_2 ,\\
    &y=\lambda_1 y_1+z y_2,\\
    &t\geq \lambda_1 \vert y_1-d\vert^q+z \vert x_2+y_2-d\vert^q,\\
   & \lambda_1, z \geq 0,\\
   & \vert x_2 \vert \leq M.
\end{aligned}
\end{equation}
Next, we substitute $x/z$ for $x_2$ and $(y-z y_2)/\lambda_1$ for $y_1$ to obtain
\begin{equation}
\begin{aligned}
   & \lambda_1+z=1, \     \lambda_1, z \geq 0, \ \vert x \vert \leq M z\\
   &  t\geq \frac{1}{\lambda_1^{q-1}}\vert y- y_2 z-d(1-z)\vert^q+\frac{1}{z^{q-1}}\vert x+ y_2 z-dz\vert^q.\
\end{aligned}
\end{equation}
Finally, we let $z y_2$ be the free variable $\beta$ and set $\lambda_1=1-z$ to obtain the required convex set.\quad \qed
\end{proof}

\section{{\color{black}Generalizing} the Matrix Perspective Reformulation Technique {\color{black}to Functions} }\label{sec:mprtextension}
We now demonstrate the MPRT can be extended to incorporate a {\color{black}different} separability of eigenvalues assumption, at the price of (a possibly significant amount of) additional notations. {\blue For any symmetric matrix $\bm{X}$, let us denote $\lambda_i^{\downarrow}(\bm{X})$ the $i$th largest eigenvalue of $\bm{X}$. }
Before proceeding any further, we recall the following result, due to \citep[Example 18.c]{ben2001lectures}, which provides a semidefinite representation of the sum of the $k$ largest eigenvalues:
\begin{lemma}[Representability of sums of largest eigenvalues]\label{lemma:sdrk}
Let $S_k(\bm{X}):=\sum_{i=1}^k \lambda_i^{\blue \downarrow}(\bm{X})$ denote the sum of the $k$ largest eigenvalues of a {\blue symmetric} matrix $\bm{X} \in \mathcal{S}^n$. Then, the epigraph of $S_k$, $S_k(\bm{X}) \leq t_k$, admits the {\blue following semidefinite} representation:
\begin{align*}
    t_k \geq k s_k+\mathrm{tr}(\bm{Z}_k), \ \bm{Z}_k+s_k \mathbb{I} \succeq \bm{X}, \bm{Z}_k \succeq \bm{0}. 
\end{align*}
\end{lemma}

{\blue Based on this result, we can relax the assumption that the penalty term $\Omega(\bm{X})$ corresponds to the trace of an operator function. Instead, we can assume:}
\begin{assumption} $\Omega(\bm{X})=\sum_{i \in [n]} p_i \lambda_i^{\blue \downarrow}(f_\omega(\bm{X}))$, where $p_1 \geq \ldots \geq p_n \geq 0$ and {\blue where $\omega$ is 
a function satisfying Assumption \ref{ass:coercive} and whose associated operator function, $f_\omega$, is matrix convex.}
\label{assumption:weightedsep}
 \end{assumption}
 
This assumption is particularly suitable for Markov Chain problems \citep[see, e.g.,][Chapter 4.6]{boyd2004convex}, where we are interested in controlling the behaviour of the {\color{black}largest eigenvalue (which always equals $1$) plus the }second largest eigenvalue of a matrix. However, it might appear to be challenging to model, since, e.g., $\lambda_2^{\blue \downarrow}(\bm{X})$ is a non-convex function. 
By applying a telescoping sum argument reminiscent of the one in \citep[Prop. 4.2.1]{ben2001lectures}, namely
\begin{align*}
    {\blue \Omega(\bm{X}) =} \sum_{i=1}^n p_i {\blue \lambda_i^{\downarrow}(f(\bm{X}))} =\sum_{i=1}^n (p_i-p_{i+1}) {\blue S_i(f(\bm{X}))}
\end{align*}
with the convention $p_{n+1}=0$,
Lemma \ref{lemma:sdrk} allows us to rewrite low-rank problems where $\Omega(\bm{X})$ satisfies Assumption \ref{assumption:weightedsep} in the form:
\begin{align}\label{prob:lrsdo9}
    \min_{\bm{Y} \in \mathcal{Y}^k_n}\min_{\substack{\bm{X} \in \mathcal{S}^n_+, \\
    \bm{Z}_i \in \mathcal{S}^n_+, s_i, t_i \in \mathbb{R}_+ \ \forall i \in [n]}} \ & \langle \bm{C}, \bm{X} \rangle+\mu \cdot \mathrm{tr}(\bm{Y})+\sum_{i=1}^n (p_i-p_{i+1})t_i
    \\
    \text{s.t.} \quad & \langle \bm{A}_i, \bm{X}\rangle=b_i \ \forall i \in [m], \ \bm{X}=\bm{Y}\bm{X}, \ \bm{X} \in \mathcal{K}, \nonumber\\
    & t_i \geq i s_i+\mathrm{tr}(\bm{Z}_i), \ \bm{Z}_i+s_i \mathbb{I} \succeq f(\bm{X}), \bm{Z}_i \succeq \bm{0} \ \forall i \in [n],\nonumber
\end{align}
where $t_i$ models the sum of the $i$ largest eigenvalues of $f(\bm{X})$. Applying the MPRT then yields the following {\blue extension} to Theorem \ref{lemma:equivalence}:
\begin{proposition}
Suppose Problem \eqref{prob:lrsdo9} attains a finite optimal value. Then, the following problem attains the same value:
\begin{align}\label{prob:lrsdo4}
    \min_{\bm{Y} \in \mathcal{Y}^k_n}\min_{\substack{\bm{X} \in \mathcal{S}^n_+, \\
    \bm{Z}_i \in \mathcal{S}^n_+, s_i, t_i \in \mathbb{R}_+ \ \forall i \in [n]}} \ & \langle \bm{C}, \bm{X} \rangle+\mu \cdot \mathrm{tr}(\bm{Y})+\sum_{i=1}^n (p_i-p_{i+1})t_i
    \\
    \text{s.t.} \quad & \langle \bm{A}_i, \bm{X}\rangle=b_i \ \forall i \in [m], \ {\color{black}\bm{Y}^{-\frac{1}{2}}\bm{X}\bm{Y}^{-\frac{1}{2}} \in \mathcal{K}, \nonumber}\\
    & t_i \geq i s_i+i-\mathrm{tr}(\bm{Y})+\mathrm{tr}(\bm{Z}_i)\ \forall i \in [n],\nonumber \\
    &\bm{Z}_i+s_i \mathbb{I} \succeq g_f(\bm{X}, \bm{Y}){\color{black}+\omega(0)(\mathbb{I}-\bm{Y})}, \bm{Z}_i \succeq \bm{0} \ \forall i \in [n].\nonumber
\end{align}
\end{proposition}
{\blue The proof of this reformulation is almost identical to the proof of Theorem \ref{lemma:equivalence}, after observing that \eqref{eqn:traceq} holds not only for the traces but for the matrices directly, i.e., if $\bm{X}$ and $\bm{Y} \in \mathcal{Y}^k_n$ commute, we have 
\begin{align*}
    f(\bm{X}) = g_f(\bm{X},\bm{Y}) + \omega(0) (\mathbb{I} - \bm{Y}).
\end{align*}}
Problem \eqref{prob:lrsdo4} involves $n$ times as many variables as Problem \eqref{prob:lrsdo2} and therefore supplies substantially less tractable relaxations.  Nonetheless, it could be useful in {\blue specific} instances. In the aforementioned Markov Chain mixing problem, $p_i-p_{i+1}=0 \ \forall i \geq k$ with $k=2$, so we can omit the variables which model the eigenvalues larger than $2$ .

{\blue 
\section{Extension to the rectangular case} \label{sec:A.rectangular}
In this section, we extend the MPRT to the case where $\bm{X}$ is a generic $n \times m$ matrix and $f(\bm{X})$ is the convex quadratic penalty $f(\bm{X})=\bm{X}^\top \bm{X}$. In this case, $\operatorname{tr}(f(\bm{X}))=\Vert \bm{X}\Vert_F^2$ is the squared Frobenius norm of $\bm{X}$.  

First, observe that $f : \mathbb{R}^{n \times m} \rightarrow \mathcal{S}^m_+$. Alternatively, one could have considered $g(\bm{X}) = \bm{X}\bm{X}^\top \in \mathcal{S}_+^n$ and obtain the same penalty, i.e., $\operatorname{tr}(f(\bm{X})) = \operatorname{tr}(g(\bm{X}))$. In other words, one can arbitrarily choose whether $f$ preserves the row or the column space of $\bm{X}$. By the Schur complement lemma, the epigraph is semidefinite representable via 
\begin{align*}
    \operatorname{epi}(f) := \left\{ (\bm{X}, \bm{\theta} ) \in \mathbb{R}^{n \times m} \times \mathcal{S}^m_+ \ : \ \begin{pmatrix}
        \bm{\theta} & \bm{X}^\top \\ \bm{X} & \mathbb{I} \end{pmatrix} \succeq \bm{0} \right\},
\end{align*}
so $f$ is matrix convex. 

In the symmetric case, we considered the matrix perspective of $f$ at $(\bm{X}, \bm{Y})$, where $\bm{Y} \succeq \bm{0}$ is a matrix controlling the range of $\bm{X}$. When $\bm{X}$ is no longer symmetric, it is natural to consider a matrix perspective function which involves two projection matrices, one of which models the row space and one which models the column space, as proposed in our prior work \cite{bertsimas2020mixed}. More precisely, for $\bm{Y}, \bm{Z} \succ \bm{0}$ we define a perspective of $f$ as  
\begin{align}
    g_f(\bm{X}, \bm{Y}, \bm{Z})= \bm{Z}^\frac{1}{2}f(\bm{Y}^{-\frac{1}{2}}\bm{X}\bm{Z}^{-\frac{1}{2}})\bm{Z}^\frac{1}{2}.
\end{align}
For $f(\bm{X}) = \bm{X}^\top \bm{X}$, this function actually does not depend on $\bm{Z}$. 
Hence, we consider 
\begin{align*}
    \Tilde{g}_f(\bm{X}, \bm{Y}) = g_f(\bm{X}, \bm{Y}, \bm{Z})= \bm{X}^\top \bm{Y}^{-1} \bm{X}.
\end{align*} 
Extending this function to positive semidefinite $\bm{Y}$ using the same proof technique as in Proposition \ref{prop:operatorperspective}, we then obtain
\begin{align*}
    \Tilde{g}_f(\bm{X}, \bm{Y}) =\begin{cases} \bm{X}^\top \bm{Y}^{\dagger} \bm{X} & \mbox{ if } \bm{Y} \succeq \bm{0},\ \operatorname{Span}(\bm{X}) \subseteq \operatorname{Span}(\bm{Y}), \\ \infty & \mbox{otherwise}. \end{cases}
\end{align*} 
\begin{proof} Fix $\bm{X} \in \mathcal{S}^n$ and $\bm{Y} \succeq \bm{0}$. As in the proof of  Proposition \ref{prop:operatorperspective} denote $\bm{P}$ the orthogonal projection onto the kernel of $\bm{Y}$, and define $\bm{Y}_\varepsilon := \bm{Y} + \varepsilon \bm{P}$ for $\varepsilon > 0$. Hence, 
\begin{align*}
    \bm{X}^\top \bm{Y}_\varepsilon^{-1} \bm{X} = \bm{X}^\top \bm{Y}^{\dagger}  \bm{X} + {\varepsilon}^{-1} \bm{X}^\top  \bm{P} \bm{X}.
\end{align*}
The right-hand side admits a finite limit if and only if $$\bm{X}^\top  \bm{P} \bm{X} = \bm{0} \iff \operatorname{Span}(\bm{X}) \subseteq \operatorname{Ker}(\bm{P}) = \operatorname{Span}(\bm{Y}). \quad \qed$$
\end{proof}

Furthermore, using the Schur complement lemma as in \cite{bertsimas2020mixed}, one can show that $\Tilde{g}_f$ is SDP-representable:
\begin{align*}
    \operatorname{epi}(\Tilde{g}_f)= 
    \left\{ (\bm{X}, \bm{Y}, \bm{\theta} ) \in \mathbb{R}^{n \times m} \times \mathcal{S}^n_+ \times \mathcal{S}^m \ : \ \begin{pmatrix}
        \bm{\theta} & \bm{X}^\top \\ \bm{X} & \bm{Y} \end{pmatrix} \succeq \bm{0} \right\},
\end{align*}
and hence matrix convex. 

Finally, we can easily check that Theorem \ref{lemma:equivalence} still holds in the symmetric case because \eqref{eqn:traceq} --which simplifies to $\operatorname{tr}(f(\bm{X})) = \operatorname{tr}(\Tilde{g}_f(\bm{X}))$ in this case-- holds for any $\bm{Y} \in \mathcal{Y}^k_n$ such that $\bm{X} = \bm{Y} \bm{X}$.

\begin{remark} We believe the approach outlined above could be generalized to a broader class of function that generalizes operator functions to the non-symmetric case. Namely, we could consider functions of the form
$$f_\omega(\bm{X})=\bm{V}\mathrm{Diag}\left( \omega(\sigma_1^x), \dots, \omega(\sigma_m^x)\right) \bm{V}^\top$$
where $\bm{X}=\bm{U}\mathrm{Diag}\left( \sigma_1^x, \dots, \sigma_m^x\right)\bm{V}^\top$ is a singular value decomposition of $\bm{X}$ and $\omega$ is a convex function satisfying Assumption \ref{ass:coercive}. Again, $f_\omega$ could arbitrarily be defined as preserving $\bm{U}$ or $\bm{V}$. For these functions, the perspective $g_{f_\omega}(\bm{X},\bm{Y},\bm{X})$ is well defined for $\bm{Y},\bm{Z} \succ \bm{0}$. Unlike in the quadratic case, however, its value will depend on both $\bm{Y}$ and $\bm{Z}$. 
Developing the theoretical tools necessary to extend the MPRT to rectangular matrices, is therefore a question for future research. 
\end{remark}
}

\end{appendices}
\end{document}